\documentclass[12pt, a4 paper]{amsart}

\calclayout

\usepackage[english]{babel}
\usepackage[utf8]{inputenc}
\usepackage{amsmath}
\usepackage{graphicx}
\usepackage[colorinlistoftodos]{todonotes}
\usepackage{amssymb,amscd,latexsym,epsfig,xy}
\usepackage[mathscr]{euscript}
\usepackage{amsthm}
\usepackage{tikz}
\usepackage{MnSymbol}
\usetikzlibrary{matrix,arrows,decorations.pathmorphing}
\usepackage{hyperref}
\usepackage{enumerate}
\usepackage{float}
\usepackage{tikz-cd}
\usepackage{relsize}
\usepackage{mathtools}
\usepackage{dsfont}
\usepackage{cleveref}
\usepackage{cite}

\crefformat{section}{\S#2#1#3} 
\crefformat{subsection}{\S#2#1#3}
\crefformat{subsubsection}{\S#2#1#3}

\newtheorem{defn}{Definition}[section]
\newtheorem{thm}[defn]{Theorem}
\newtheorem{lem}[defn]{Lemma}
\newtheorem{cor}[defn]{Corollary}

\newtheorem{prop}[defn]{Proposition}

\newcommand\A{\mathbb A}
\newcommand\B{\mathbb B}
\newcommand\C{\mathbb C}
\newcommand\E{\mathbb E}

\newcommand\G{\mathbb G}

\newcommand\NN{\mathbb N}

\newcommand\PP{\mathbb P}
\newcommand\R{\mathbb R}
\newcommand\Q{\mathbb Q}
\newcommand\V{\mathbb V}
\newcommand\Z{\mathbb Z}
\newcommand{\kk}{\mathds{k}}

\newcommand\cC{\mathcal C}

\newcommand\cF{\mathcal F}

\newcommand\cI{\mathcal I}

\newcommand\cO{\mathcal O}
\newcommand\cP{\mathcal P}

\newcommand\cX{\mathcal X}

\newcommand\fd{\mathfrak d}
\newcommand\fD{\mathfrak D}

\newcommand\fm{\mathfrak m}

\newcommand\Spec{\operatorname{Spec}\,}
\newcommand\Spf{\operatorname{Spf}\,}

\newcommand\Supp{\operatorname{Supp}}
\newcommand\Lift{\operatorname{Lift}}
\newcommand\Mono{\operatorname{Mono}}

\newcommand\virt{\mathrm{virt}}
\newcommand\can{\mathrm{can}}
\newcommand\gp{\mathrm{gp}}
\newcommand\iin{\mathrm{in}}
\newcommand\out{\mathrm{out}}

\title[Quantum mirrors of log Calabi-Yau surfaces]{Quantum mirrors of log Calabi-Yau surfaces and higher genus curve counting}

\author{Pierrick Bousseau}

\date{}

\begin{document}

\begin{abstract}
Gross, Hacking and Keel have constructed mirrors of log Calabi-Yau surfaces in terms of counts of rational curves.
Using $q$-deformed scattering diagrams
defined in terms of higher-genus 
log Gromov-Witten invariants, we construct deformation quantizations
of these mirrors and we 
produce canonical bases of the corresponding  non-commutative algebras of functions. 
\end{abstract}

\maketitle

\setcounter{tocdepth}{1}

\tableofcontents

\thispagestyle{empty}

\section{Introduction}

\subsection{Context and motivations}

\subsubsection{Mirror symmetry} 
The Strominger--Yau--Zaslow (SYZ) 
\cite{MR1429831}
picture of mirror symmetry suggests
an original way of constructing algebraic varieties:
given a Calabi-Yau variety, its mirror geometry should be constructed
in terms of its enumerative geometry of holomorphic discs.
This picture has been developed by Fukaya
\cite{MR2131017},  
Kontsevich and Soibelman \cite{MR2181810},
Gross and Siebert \cite{MR2846484}, Auroux 
\cite{MR2386535} and many others. 
In particular,
Gross and Siebert have developed an algebraic approach in which the enumerative geometry of holomorphic discs is replaced by some 
genus-0 logarithmic Gromov--Witten invariants. Given the recent progress
in logarithmic Gromov--Witten theory, in particular the definition of punctured invariants by Abramovich, Chen, Gross and Siebert
\cite{abramovich2017punctured}, it is likely that this approach will lead to some general mirror symmetry construction in the algebraic setting, see Gross and Siebert
\cite{gross2016intrinsic} for an 
announcement.

\subsubsection{The work of Gross, Hacking, and Keel}
An early version of this mirror construction has been used by 
Gross, Hacking, and Keel \cite{MR3415066}  
to construct mirror families of log Calabi-Yau surfaces, with non-trivial applications to the theory of surface singularities and in particular a proof of Looijenga's conjecture on smoothing of cusp singularities. 
More precisely, the construction of \cite{MR3415066} applies to Looijenga pairs, that is, to pairs $(Y,D)$ where 
$Y$ is a smooth projective surface and $D$ is some anticanonical singular nodal curve.
The upshot is in general a formal flat family 
$\cX \rightarrow S$ of surfaces over a formal completion, near some point $s_0$, the `large volume limit of Y', of an algebraic approximation to the complexified Kähler cone of $Y$.

Furthermore, $\cX$ is an affine Poisson formal variety with a canonical linear basis of so-called
theta functions and the map 
$\cX \rightarrow S$ is Poisson if $S$ is equipped with the zero Poisson bracket.
Under some positivity assumptions on $(Y,D)$, this family can be in fact extended to an algebraic family over an algebraic base
and the generic fiber is then a smooth algebraic symplectic surface.
To simplify the exposition in this introduction, we assume for now that it is the case.

The first step of the construction involves defining the fiber $\cX_{s_0}$, that is,
the `large complex structure limit' of the family $\cX$. This step is essentially combinatorial and can be reduced to some toric geometry: $\cX_{s_0}$ is a reducible union of toric varieties.

The second step is to construct $\cX$ by smoothing of $\cX_{s_0}$.
This construction 
is based on the consideration of an 
algebraic object, a scattering diagram,
a notion introduced by Kontsevich and Soibelman
\cite{MR2181810} and further developed by 
Gross and Siebert \cite{MR2846484}, whose definition encodes genus-0 log Gromov--Witten invariants\footnote{In \cite{MR3415066}, an \emph{ad hoc} definition of genus-0 Gromov--Witten invariants is used, which was supposed to coincide with genus-0 log Gromov-Witten invariants. This fact follows from the remark at the end of \cite[\S 4]{MR3904449}. In the present paper, we use log Gromov-Witten theory systematically.}
 of $(Y,D)$. The key non-trivial property to check is the so-called consistency of the scattering diagram.
In \cite{MR3415066}, the consistency relies on the work of Gross, Pandharipande and Siebert
\cite{MR2667135}, which itself relies on connection with tropical geometry
\cite{MR2137980}, \cite{MR2259922}. Once the consistency of the scattering diagram is guaranteed, some combinatorial objects, the broken lines
\cite{cps}, are well defined and can be used to construct the algebra of functions $H^0(\cX, \cO_\cX)$ with its
linear basis of theta functions.

\subsubsection{Quantization}\footnote{The existence of theta functions is related to the geometric quantization of the real integrable system formed by a Calabi-Yau manifold with an SYZ fibration. We do \emph{not}
refer to this quantization story. In this paper, quantization always means deformation quantization of a holomorphic symplectic/Poisson variety.}
The variety $\cX$ being a Poisson variety over $S$,
it is natural to ask about its quantization, for example in the sense of 
deformation quantization.
As $\cX$ and $S$ are affine, the deformation quantization problem takes its simplest form: to construct a structure of non-commutative $H^0(S, \cO_S)[\![\hbar]\!]$-algebra on $H^0(\cX, \cO_{\cX}) \otimes 
\C[\![\hbar]\!]$ whose commutator is given at the linear order in $\hbar$ by the Poisson bracket on $H^0(\cX, \cO_{\cX})$.
There are general existence results,
\cite{MR1855264}, \cite{MR2183259},  
for deformation quantizations of smooth affine Poisson varieties. Some useful reference on deformation quantization of algebraic symplectic varieties is \cite{MR2119140}. 
In fact, on its smooth locus, the map $\cX \rightarrow S$ is relatively symplectic of relative dimension two and then the existence of a deformation is easy because the obstruction space vanishes for dimension reasons. But there are no known general results which would guarantee a priori the existence of a deformation quantization of $\cX$ over $S$ because 
$\cX \rightarrow S$ is singular, for example. over $s_0 \in S$ to start with. Specific examples of deformation quantization of such geometries usually involve some situation-specific representation theory or geometry: see,  
for example, \cite{MR2037756}, \cite{MR2329319}, \cite{MR2734346},\cite{MR3581328}.

\subsection{Main results}
The main result of the present paper is a construction of a deformation quantization of $\cX \rightarrow S$. Our construction 
follows the lines of Gross, Hacking and Keel
\cite{MR3415066} except that, rather than using only genus-0 log Gromov--Witten invariants, we use higher-genus log Gromov--Witten invariants, the genus parameter playing the role of the quantization parameter $\hbar$ on the mirror side. 

We construct a quantum version of a scattering diagram and we prove its consistency using the main result of 
\cite{bousseau2018quantum_tropical}, which itself relies on the connection with refined tropical geometry 
\cite{MR3904449}. 
Once 
the consistency of the quantum scattering diagram
is guaranteed, a quantum version of the 
broken lines is well defined and can be used to construct a deformation quantization of $H^0(\cX, \cO_\cX)$.
In fact, it follows from the use of 
\cite{bousseau2018quantum_tropical} that the dependence on the deformation parameter $\hbar$ is algebraic in 
$q=e^{i\hbar}$\footnote{Because in general $\cX$ is already a formal object, this claim has to be stated more precisely; see
Theorems
\ref{thm_q}. 
It is correct in the most naive sense if $(Y,D)$ is positive enough and $\cX$ is then really an algebraic family.},
something which in general cannot be obtained from some general deformation-theoretic argument.
In other words, the main result of the present paper can be phrased in the
following slightly vague terms (see Theorems
\ref{thm_1_ch3}--
\ref{thm_q} for precise statements).

\begin{thm} \label{main_thm_ch3}
The Gross-Hacking-Keel \cite{MR3415066} Poisson family
$\cX \rightarrow S$, the mirror of a Looijenga pair $(Y,D)$, admits
a deformation quantization, which can be constructed in a synthetic way from the higher-genus log Gromov-Witten theory of 
$(Y,D)$. Furthermore, the dependence on the deformation quantization parameter $\hbar$ is algebraic in $q=e^{i\hbar}$. 
\end{thm}

The notion of quantum scattering diagram is already suggested at the end of \cite[\S 11.8]{MR2181810}
and is used by Soibelman 
\cite{MR2596639} to construct non-commutative deformations of 
non-archimedean K3 surfaces. The connection with quantization, for example in the context of cluster varieties
\cite{MR2567745, MR2641183}, was expected, and  quantum broken lines have been studied by 
Mandel \cite{mandel2015scattering}.
The key novelty of the present paper, 
building on the previous work 
\cite{MR3904449,bousseau2018quantum_tropical}
of the author, is the connection between these algebraic/combinatorial $q$-deformations and the geometric deformation given by higher-genus log Gromov--Witten theory.

This connection between higher-genus Gromov--Witten theory and quantization is perhaps a little surprising, even if similarly-looking statement are known or expected. In 
\cref{section_physics}, 
we explain that 
Theorem \ref{main_thm_ch3} should be viewed as an example of a higher genus mirror symmetry relation, the deformation quantization
being a two-dimensional reduction of the three-dimensional higher-genus B-model (BCOV theory). We also comment on the relation with expectations from string theory, in a way parallel to 
\cite[\S 8]{bousseau2018quantum_tropical}.

In the context of mirror symmetry,
there is a well-known symplectic interpretation of some 
non-commutative deformations on the B-side,
involving deformation of the complexified 
symplectic form which does not preserve the
Lagrangian nature of the fibers of the SYZ fibration. An example of this phenomenon has been studied by
Auroux, Katzarkov and Orlov
\cite{MR2257391} in the context of mirror
symmetry for del Pezzo surfaces.
There is some work in progress by Sheridan and Pascaleff
about generalizing this approach to study 
non-commutative deformations of mirrors of log Calabi-Yau varieties.
This approach remains entirely in the traditional realm of genus-0 
holomorphic curves and so is completely 
different from our approach using higher-genus curves. The compatibility of these two approaches can be understood via a chain of string-theoretic dualities.

It is natural to ask how the deformation quantization given by Theorem 
\ref{main_thm_ch3} is
related to previously known examples of quantization.
In \cref{section_examples_ch3} we treat a simple example and we recover a well-known description of the $A_2$ quantum $\cX$-cluster variety \cite{MR2567745}. 

For $Y$ a cubic surface in $\PP^3$ and $D$
a triangle of lines on $Y$, the quantum scattering diagram can be explicitly computed and so 
using techniques similar to those developed in 
\cite{gross2019mirror}, one should be able to show that
 the deformation quantization given by Theorem 
\ref{main_thm_ch3} coincides with the one constructed by Oblomkov 
\cite{MR2037756} using Cherednik algebras (double affine Hecke algebras). We leave this verification, and the general relation to quantum $\cX$-cluster varieties, to future work.

Similarly, if $Y$ is a del Pezzo surface of degree $1$, $2$ or $3$ and $D$ a nodal cubic, it would be interesting to compare Theorem \ref{main_thm_ch3}
with the construction of Etingof, 
Oblomkov and Rains \cite{MR2329319}
using Cherednik algebras. In these cases, the quantum 
scattering diagrams are extremely complicated and new ideas are probably required.

\subsection{Plan of the paper}

In 
\cref{section_basics}, we set up our notation and we give precise versions of the main results.
In 
\cref{section_scattering} we describe the formalism of 
quantum scattering diagrams and quantum broken lines. In 
\cref{section_canonical}
we explain how to
associate to every Looijenga pair $(Y,D)$ 
a canonical quantum scattering diagram constructed in terms of higher-genus 
log Gromov--Witten invariants of 
$(Y,D)$. The key result in our construction is Theorem 
\ref{thm_consistency}
establishing the consistency of the canonical quantum scattering diagram.
The proof of 
\mbox{Theorem 
\ref{thm_consistency}}
follows the reduction steps used by Gross, Hacking and Keel \cite{MR3415066} in the genus-0 case. In the final step, we use the main result of \cite{bousseau2018quantum_tropical}
in place of the main result of \cite{MR2667135}.
In 
\cref{section_ext} we finish the proofs of the main theorems.
In 
\ref{section_examples_ch3}, we work out an explicit example.
Finally, in 
\cref{section_physics}, we discuss the relation of our main result,
\mbox{Theorem \ref{main_thm_ch3}}, with higher-genus mirror symmetry and some string-theoretic arguments.

\subsection*{Acknowledgements}
I would like to thank my supervisor Richard Thomas for continuous support and numerous discussions, suggestions and corrections. 

I thank Rahul Pandharipande, Tom Bridgeland, Liana Heuberger, Michel van Garrel and Bernd Siebert for various invitations to conferences and seminars where this work has been presented. I thank Mark Gross and Johannes Nicaise for corrections on a  first draft of this paper. I thank the referee for corrections and useful suggestions.

This work is supported by the EPSRC award 1513338, Counting curves in algebraic geometry, Imperial College London, and has benefited from the EPRSC [EP/L015234/1], EPSRC Centre for Doctoral Training in Geometry and Number Theory (The London School of Geometry and Number Theory), University College London.

\section{Basics and main results}
\label{section_basics}

\subsection{Looijenga pairs}
Let $(Y,D)$ be a Looijenga pair\footnote{We follow the terminology of Gross, Hacking and Keel
\cite{MR3415066}}: $Y$ is a smooth projective complex surface  
and $D$ is a singular reduced normal crossings anticanonical divisor on $Y$. Writing the irreducible components
\[ D=D_1+\dots +D_r\,, \] 
$D$ is a 
cycle of $r$ irreducible smooth rational curves 
$D_j$ if $r \geqslant 2$, or 
an irreducible nodal rational curve 
if $r=1$.
The complement 
$U \coloneqq Y-D$ is a non-compact Calabi-Yau surface, equipped with a holomorphic symplectic form 
$\Omega_U$, defined up to non-zero scaling and having first-order poles along $D$. 
We refer to 
\cite{MR632841,
friedman2015geometry,
MR3415066,
MR3314827}
for more background on 
Looijenga pairs.
 
There are two basic operations on 
Looijenga pairs:
\begin{itemize}
\item \emph{Corner blow-up}. If 
$(Y,D)$ is a Looijenga pair, then 
the blow-up 
$\tilde{Y}$ of $Y$ at one of the corners of $D$, equipped with the preimage 
$\tilde{D}$ of $D$, is a Looijenga pair.
\item \emph{Boundary blow-up}. If 
$(Y,D)$ is a Looijenga pair, 
then the blow-up
$\tilde{Y}$ of $Y$ at a smooth 
point of $D$, equipped with the strict transform 
$\tilde{D}$ of $D$, is a Looijenga 
pair.
\end{itemize}

A corner blow-up does not change the 
interior $U$ of a Looijenga pair 
$(Y,D)$.
An interior blow-up changes the interior
of a Looijenga pair:
if $(\tilde{Y},\tilde{D})$ is an interior blow-up of $(Y,D)$, then, for example, we have 
\[ e(\tilde{U})=e(U)+1\,,\]
where $U$ is the interior of 
$(Y,D)$,  
$\tilde{U}$
is the interior of $(\tilde{Y},
\tilde{D})$, and 
$e(-)$ denotes the topological Euler characteristic.

If $\bar{Y}$ is a smooth toric variety and $\bar{D}$ is its toric boundary divisor, then 
$(\bar{Y}, \bar{D})$ is a Looijenga pair, of interior 
$U=(\C^{*})^2$. In particular, we have
$e(U)=e((\C^{*})^2)=0$. Such Looijenga pairs are called toric. A Looijenga pair $(Y,D)$ is toric if and only if
its interior $U=Y-D$ has a vanishing Euler topological characteristic: $e(U)=0$. 

A toric model of a Looijenga pair
$(Y,D)$ is a toric Looijenga pair 
$(\bar{Y}, \bar{D})$ such that 
$(Y,D)$ is obtained from 
$(\bar{Y}, \bar{D})$ by successively applying a finite 
number of boundary blow-ups.

If $(Y,D)$ is a Looijenga pair, then, by
Proposition 1.3 of \cite{MR3415066}, 
there exists a Looijenga pair 
$(\tilde{Y},\tilde{D})$, obtained from 
$(Y,D)$ by successively applying a finite number of corner blow-ups, which admits a toric model. In particular, 
we have $e(U) \geqslant 0$, where 
$U$ is the interior of $(Y,D)$.

Let $(\bar{Y}, \bar{D})$ be a toric model of a Looijenga pair $(Y,D)$
of interior $U$. 
Let $\bar{\omega}$ be a torus invariant real symplectic form on 
$(\C^{*})^2= \bar{Y}-\bar{D}$. Then the corresponding 
moment map for the torus action gives 
$\bar{Y}$ the structure of 
toric fibration, whose restriction to 
$U$ is a smooth fibration in Lagrangian tori. By definition of a toric model,
we have a map $p \colon (Y,D) \rightarrow (\bar{Y}, \bar{D})$, composition of successive boundary blow-ups. Let $E_j$ denote the exceptional divisors, $j=1, \dots, e(U)$. Then for $\epsilon_j$ small enough positive real numbers, there exists a symplectic form 
$\omega$ in the class
\[ p^{*} [\bar{\omega}]
- \sum_{j=1}^{e(U)} 
\epsilon_j E_j\]
with respect to which $Y$ admits an almost toric fibration, whose restriction to 
$U$ is a fibration in Lagrangian tori 
with $e(U)$ nodal fibers 
\cite{MR3502098}.

Toric models of a given Looijenga pair are very far from being unique but are always related by sequences of corner blow-ups/blow-downs and boundary blow-ups/blow-downs. The corresponding 
almost toric fibrations are related by 
nodal trades 
\cite{MR2024634}.

Following \cite[\S 6.3]{MR3415066},
we say that $(Y,D)$ is positive if
one of the following equivalent conditions is satisfied.
\begin{itemize}
\item There exist positive integers $a_1,\dots,a_r$
such that, for all
\mbox{$1 \leqslant k \leqslant r$}, we have 
\[\left(\sum_{j=1}^r a_j D_j\right)\cdot D_k>0\,.\]
\item $U$ is deformation
equivalent to an affine surface.
\item
$U$ is the minimal resolution of
$\Spec (H^0(U,\cO_U))$, which is an affine surface with at worst Du Val singularities. 
\end{itemize}

\subsection{Tropicalization of Looijenga pairs}
\label{section_tropicalization}
We refer to 
\cite[\S 1.2 and 2.1]{MR3415066} 
and 
\cite[\S 1]{gross2016theta} 
for details.
Let 
$(Y,D)$ be a Looijenga pair. Let 
$D_1, \dots, D_r$ be the component of $D$, ordered 
in a cyclic order, the index $j$ of 
$D_j$ being considered modulo $r$.
For every $j$ modulo $r$,
we consider an integral affine cone 
$\sigma_{j,j+1}=(\R_{\geqslant 0})^2$,
of edges $\rho_j$ and $\rho_{j+1}$.
We abstractly glue together the cones 
$\sigma_{j-1,j}$ and 
$\sigma_{j,j+1}$ along the edge 
$\rho_j$. We obtain a topological 
space $B$, homeomorphic to 
$\R^2$, equipped with a cone decomposition 
$\Sigma$ in two-dimensional cones $\sigma_{j,j+1}$, all meeting at a point that we call $0 \in B$, and pairwise meeting along one-dimensional cones
$\rho_j$. The pair $(B, \Sigma)$ is the dual intersection complex of 
$(Y,D)$.
We define an integral linear
structure on 
$B_0=B-\{0\}$
by the charts 
\[ \psi_j \colon U_j 
\rightarrow \R^2 \,,\]
where
$U_j
\coloneqq 
\mathrm{Int}(\sigma_{j-1,j}\cup \sigma_{j,j+1})$ and $\psi_j$ is defined on the closure of $U_j$ by
\[ \psi_j(v_{j-1}) = (1,0)\,,
\psi_j(v_{j})=(0,1)
\,, \psi_j(v_{j+1})=(-1,-D_j^2) \,,\]
where $v_j$ is a primitive generator of $\rho_j$ and 
$\psi_j$
is defined linearly on the two-dimensional cones.
Let $\Lambda$ be the sheaf of integral tangent vectors of $B_0$.
It is a 
locally constant sheaf on $B_0$ of fiber 
$\Z^2$.

The integral linear structure 
on $B_0$ extends to $B$ through 
$0$ if and only if 
$(Y,D)$ is toric. In this case, 
$B$ can be identified with $\R^2$
as an integral linear manifold and 
$\Sigma$ is simply the fan of the 
toric variety $Y$.
In general, the integral linear structure is singular at $0$, with a non-trivial
monodromy along a loop going around $0$.

As $B_0$ is an integral linear manifold, its set 
$B_0(\Z)$ of integral points is well defined. We denote 
$B(\Z)\coloneqq B_0(\Z) \cup \{ 0\}$.
If $(Y,D)$ is toric, with 
$Y-D=(\C^{*})^2$, then 
$B(\Z)$ is the lattice of cocharacters 
of $(\C^{*})^2$, that is, the lattice of one-parameter subgroups 
$\C^{*} \rightarrow (\C^{*})^2$. 
Thus, intuitively, a point of $B_0(\Z)$ is a way to go to infinity in 
$(\C^{*})^2$.
This intuition remains true in the non-toric case: a point in 
$B_0(\Z)$ is a way to go to infinity in the interior $U$ of the pair 
$(Y,D)$. 

More precisely, if we equip 
$(Y,D)$ with its divisorial log structure, then $p \in B(\Z)$ defines a tangency condition along $D$ for a marked point $x$ on a stable log curve $f \colon C
\rightarrow (Y,D)$. If $p=0$, then ${f(x) \notin D}$. 
If $p = m_j v_j$,
$m_j \in \NN$, then $f(x) \in D_j$ with tangency order $m_j$ along 
$D_j$ and tangency order zero along 
$D_{j-1}$ and 
$D_{j+1}$.
If $p=m_j v_{j+1} 
+m_{j+1} v_{j+1}$,
$m_j, m_{j+1} \in \NN$, then 
$f(x) \in D_{j} \cap D_{j+1}$
with tangency order 
$m_j$ along 
$D_j$
and tangency order 
$m_{j+1}$
along 
$D_{j+1}$\footnote{This makes sense precisely because we are using 
log geometry.}.

Let $P$ be a toric 
monoid and $P^{\mathrm{gp}}$
be its group completion, a finitely generated abelian group.
Denote $P_\R^{\mathrm{gp}}
\coloneqq P^{\mathrm{gp}} \otimes_\Z \R$,
a finite-dimensional $\R$-vector space.  
Let $\varphi$ be a convex
$P_\R^{\mathrm{gp}}$-valued 
multivalued $\Sigma$-piecewise linear function on 
$B_0$.
Let $\Lambda_j$ be the fiber of
the sheaf $\Lambda$ of integral tangent 
vectors over the chart $U_j$.
Let $n_{j-1,j}, n_{j,j+1}
\in \Lambda_i^\vee \otimes P^{\mathrm{gp}}$ 
be the slopes of 
$\varphi|_{\sigma_{j-1,j}}$ and 
$\varphi|_{\sigma_{j,j+1}}$.
Let $\Lambda_{\rho_j}$ be the fiber of 
the sheaf of integral tangent 
vectors to the ray $\rho_j$.
Let $\delta_j \colon
\Lambda_j \rightarrow \Lambda_j /
\Lambda_{\rho_j} \simeq \Z$
be the quotient map. We fix signs 
by requiring $\delta_j$
to be non-negative on tangent vectors 
pointing from $\rho_j$
to $\sigma_{j,j+1}$.
Then 
$(n_{j,j+1}-n_{j-1,j})(\Lambda_{\rho_j})=0$
and hence there exists 
$\kappa_{\rho_j,\varphi} \in P$
with 
\[ n_{j,j+1}-n_{j-1,j} 
= \delta_j \kappa_{\rho_j,\varphi} \,, \]
called the kink of 
$\varphi$ along 
$\rho_j$.

Let 
$\B_{0, \varphi}$ be the 
$P_\R^{\mathrm{gp}}$-torsor,
which is set-theoretically
$B_0 
\times P_\R^{\mathrm{gp}}$
but with an integral affine structure
twisted by $\varphi$:
for each chart 
$\psi_j \colon U_j \rightarrow \R^2$
of $B_0$, we define
a chart on 
$\B_{0, \varphi}$ by
\[(x,p)\mapsto
\begin{dcases}(\psi_j(x), p)
 & \text{if  } x \in \sigma_{j-1,j} \\
(\psi_j(x), p + \tilde{\delta}_j(x)
\kappa_{\rho_j, \varphi}) & \text{if  } 
x \in \sigma_{j,j+1}\,,
\end{dcases}\]
where 
$\tilde{\delta}_j \colon
\sigma_{j,j+1} \rightarrow \R_{\geqslant 0}$
is the integral affine map
of differential $\delta_j$.
By definition, 
$\varphi$ can be viewed as a section of 
the projection
$\pi \colon \B_{0,\varphi} \rightarrow B_0$.
Then 
$\cP \coloneqq \varphi^{*}
\Lambda_{\B_{0,\varphi}}$
is
a locally constant sheaf 
on $\B_{0,\varphi}$, of 
fiber $\Z^2 \oplus 
P^{\gp}$, and the 
projection
$\pi \colon 
\B_{0,\varphi}
\rightarrow B_0$
induces a short exact sequence
\[ 0 \rightarrow 
\underline{P}^{\gp}
\rightarrow \cP
\xrightarrow{\mathfrak{r}} \Lambda 
\rightarrow 0 \]
of locally constant sheaves
on $B_0$, where 
$\underline{P}^{\gp}$
is the constant sheaf on 
$B_0$ of fiber 
$P^{\gp}$, and where $\mathfrak{r}$ is the derivative 
of $\pi$.

The sheaf $\Lambda$ is naturally a sheaf of symplectic lattices: we have a skew-symmetric non-degenerate form 
\[\langle -,- \rangle \colon \Lambda \otimes \Lambda \rightarrow \Z \,.\]
We extend $\langle -,- \rangle$
to a skew-symmetric form on 
$\cP$ of kernel $\underline{P}^\gp$.

Let $P$ be a toric 
monoid and let 
$\eta \colon NE(Y)
\rightarrow P$ be a morphism
of monoids.
Then there exists a unique (up to 
a linear function) 
convex
$P_\R^{\mathrm{gp}}$-valued 
multivalued $\Sigma$-piecewise linear
function $\varphi$ on 
$B_0$ with kinks 
$\kappa_{\rho_j,\varphi} = \eta([D_j])$. 

\subsection{Algebras and quantum algebras}

When we write `$A$ is an 
$R$-algebra', we mean that 
$A$ is an associative algebra with unit over a commutative ring with unit $R$. In particular, $R$ is naturally
contained in the center of 
$A$.
We fix $\kk$ an algebraically closed field
of characteristic zero and $i \in \kk$
a square root of $-1$.

For every monoid $M$\footnote{All the monoids considered will be commutative 
and with an identity element.} equipped with 
a skew-symmetric bilinear form 
\[ \langle -,- \rangle
\colon  M \times M\rightarrow \Z\,, \]
we denote by $\kk[M]$ the 
monoid algebra of $M$, consisting of monomials 
$z^m$, $m \in M$, such that 
$z^m \cdot z^{m'}=z^{m+m'}$. It is a Poisson algebra, of Poisson bracket determined by 
\[\{z^m, z^{m'}\}=\langle m,m' \rangle z^{m+m'} \,.\]

We denote by 
$\kk_q \coloneqq \kk[q^{\pm \frac{1}{2}}]$
and $\kk_{q}[M]$
the possibly non-commutative $\kk_{q}$-algebra structure on $\kk[M] \otimes_\kk \kk_{q}$ such that
\[ \hat{z}^m. \hat{z}^{m'}=q^{\frac{1}{2} \langle m,m'
\rangle} \hat{z}^{m+m'}\,.\]

We denote
$\kk_{\hbar}
\coloneqq \kk[\![\hbar]\!]$.
We view $\kk_{\hbar}$ as a complete 
topological ring for the 
$\hbar$-adic topology and in particular, we will use the operation of completed tensor product $\hat{\otimes}$ with $\kk_{\hbar}$:
\[ (-) \hat{\otimes}_\kk \kk_{\hbar}
\coloneqq \varprojlim_j \, (-)\otimes_\kk (\kk[\hbar]/\hbar^j) \,.\]
We view $\kk_{\hbar}$
as a $\kk_{q}$-module by the change of
variables 
\[q=e^{i\hbar}=\sum_{k \geqslant 0} \frac{(i\hbar)^k}{k!}\,.\]

We denote $\kk_{\hbar}[M]
\coloneqq \kk_q[M] \hat{\otimes}_{\kk_q}
\kk_{\hbar}$.
The possibly 
non-commutative algebra $\kk_{\hbar}[M]$
is a deformation quantization of the 
Poisson algebra $\kk[M]$
in the sense that
$\kk_{\hbar}[M]$ is flat as 
$\kk_{\hbar}$-module, we recover $\kk[M]$
in
the limit $\hbar \rightarrow 0$,
$q \rightarrow 1$, 
and the linear term
in $\hbar$
of the commutator 
$[\hat{z}^m, \hat{z}^{m'}]$
in $\kk_{\hbar}[M]$
is determined by the 
Poisson bracket 
$\{z^m, z^{m'}\}$ in 
$\kk[M]$:
\[ [\hat{z}^m, \hat{z}^{m'}]
= (q^{\frac{1}{2} \langle m, m' \rangle}
- 
q^{-\frac{1}{2} \langle m, m' \rangle})
\hat{z}^{m + m'}
= \langle m,m' \rangle i \hbar 
\hat{z}^{m+m'} + \cO(\hbar^2) \,. \]
We will often apply the constructions 
$\kk[M]$ and $\kk_{\hbar}[M]$ to $M$ a fiber of the locally constant sheaves $\Lambda$ or $\cP$.

In particular, considering the toric monoid $P$ with the zero 
skew-symmetric form,  
we denote 
\[ R \coloneqq \kk[P]\] 
and 
\[R^{\hbar} \coloneqq \kk_{\hbar}[P]
=R \hat{\otimes}_\kk \kk_{\hbar}\,.\]
For every monomial ideal $I$ of $R$, we denote 
\[R_I \coloneqq R/I\]
\[R_I^q \coloneqq R/I \otimes_\kk \kk_q
=R_I[q^{\pm \frac{1}{2}}]\]
and 
\[R_I^{\hbar} \coloneqq R^{\hbar}/I
=R_I \hat{\otimes}_\kk \kk_{\hbar}
=R_I [\![ \hbar ]\!]\,.\] 
Observe that the algebras $R^{\hbar}$,
$R_I^q$
and $R_I^{\hbar}$
are commutative.

\subsection{Ore localization}
\label{section_ore}

As should be clear from the previous section, we will be dealing with 
non-commutative rings. 
Unlike what happens for commutative rings, it is not possible in general to localize 
with respect to an arbitrary multiplicative subset of a 
non-commutative ring, because of left-right issues. These left-right issues are absent
by definition if the multiplicative subset satisfies the so-called Ore conditions.

We refer, for example, to 
\cite[\S 2.1]{MR1662244}
and   
\cite[\S 1.3]{ginzburg1998lectures}
for short presentations
of these elementary notions of non-commutative algebra.
A multiplicative subset $S
\subset A -\{0\}$ of an associative ring $A$ is said to satisfy the Ore conditions if
\begin{itemize}
\item for all $a \in A$ and $s \in S$, there exist
$b \in A$ and $t\in S$ such that\footnote{Informally, $as^{-1}
=t^{-1}b$, that is, every fraction with a denominator on the right can be rewritten as a fraction with a denominator on the left.} $ta=bs$;
\item for all $a \in A$, if there exists 
$s \in S$ such that $as=0$, then there exists 
$t \in S$ such that $ta=0$; 
\item for all $b \in A$ and $t \in S$, there exists 
$a \in A$ and $s \in S$ such that\footnote{Informally, 
$t^{-1}b=as^{-1}$, that is, every fraction with a denominator on the left can be rewritten as a fraction with a denominator on the right.} $ta=bs$;
\item for all $a \in A$, if there exists 
$s \in S$ such that $sa=0$, then there
exists $t \in S$ such that $at=0$.
\end{itemize}

If $S$ is a multiplicative subset of 
an associative ring $A$ and if $S$ satisfies the Ore conditions, then there is a well-defined localized ring $A[S^{-1}]$.

Let $R$ be a commutative ring. Denote 
$R^{\hbar} \coloneqq R[\![\hbar]\!]$.

\begin{lem} \label{lem_loc}
Let $A$ be an $R^{\hbar}$-algebra such that 
$A_0 \coloneqq A/\hbar A$ is a 
commutative $R$-algebra. Assume that $A$
is $\hbar$-nilpotent, that is, that there exists $j$ such that $\hbar^j A=0$.
Denote by $\pi \colon A \rightarrow A_0$ the natural projection. Let $\overline{S}
\subset A_0 -\{0\}$ be a multiplicative subset. Then the multiplicative subset 
$S \coloneqq \pi^{-1}(\overline{S})$
of $A$ satisfies the Ore conditions.
\end{lem}

\begin{proof}
See the proof of Proposition 2.1.5 in 
\cite{MR1662244}.
\end{proof}
 
\begin{defn} \label{def_loc}
Let $A$ be an $R^{\hbar}$-algebra such that 
$A_0 \coloneqq A/\hbar A$ is a commutative 
$R$-algebra. Assume that $A$ is $\hbar$-complete, that is, that $A = \varprojlim_j
A/\hbar^j A$. By Lemma 
\ref{lem_loc}, each $A/\hbar^j A$ 
defines a sheaf of algebras on 
$X_0 \coloneqq \Spec A_0$, which we denote
by $\cO_{X_0}^{\hbar}/\hbar^j$.
We define 
\[ \cO_{X_0}^{\hbar}
\coloneqq \varprojlim_j \cO_{X_0}^{\hbar}/\hbar^j
\,,\]
which is a sheaf in $R^{\hbar}$-algebras over $X_0$ such that $\cO_{X_0}^{\hbar}/\hbar
=\cO_{X_0}$.
\end{defn}

Definition 
\ref{def_loc} gives us a systematic way to turn certain non-commutative algebras into 
sheaves of non-commutatives algebras. 

\subsection{The Gross-Hacking-Keel mirror family} \label{subsection_ghk}

We fix $(Y,D)$ a Looijenga pair.
Let $NE(Y)_\R \subset A_1(Y,\R)$ be the cone generated by curve classes and let $NE(Y)$ be the monoid 
$NE(Y)_\R \cap A_1(Y,\Z)$.

Let 
$\sigma_P \subset A_1(Y,\R)$ be a strictly 
convex polyhedral cone containing 
$NE(Y)_\R$.
Let $P \coloneqq \sigma_P \cap A_1(Y,\Z)$ be the associated monoid and let
$R \coloneqq \kk[P]$ be the corresponding $\kk$-algebra. We denote 
$t^\beta$ the monomial in $R$ defined by $\beta \in P$. Let $\fm_R$ be the maximal monomial ideal of $R$.
For every monomial ideal 
$I$ of $R$ with radical $\fm_R$,
we denote $R_I\coloneqq R/I$ and 
$S_I\coloneqq \Spec R_I$.

Let $T^D \coloneqq \mathbb{G}_m^r$
be the torus whose character group has a basis 
$e_{D_j}$ indexed by the irreducible
components $D_j$ of $D$. The map 
\[\beta \mapsto \sum_{j=1}^r (\beta \cdot D_j)
e_{D_j}\] 
induces an action of 
$T^D$ on $S_I$.

\cite[Theorem 0.1]{MR3415066} 
gives the existence of a flat
$T^D$-equivariant morphism 
\[ X_I \rightarrow S_I \,,\]
with $X_I$ affine.
The algebra of functions of $X_I$ is given as 
$R_I$-module by 
\[H^0(X_I, \cO_{X_I}) =A_I \coloneqq \bigoplus_{p \in B(\Z)} R_I \vartheta_p \,,\]
The algebra structure on
$H^0(X_I, \cO_{X_I})$ is determined by genus-0 log Gromov--Witten
invariants of $(Y,D)$.

By Theorem 0.2. of \cite{MR3415066}, there exists a unique smallest radical monomial ideal $J_{\min} \subset R$ such that the following statements hold.
\begin{itemize}
\item For every monomial ideal $I$ of $R$
of radical containing $J_{\min}$, there is a 
finitely generated $R_I$-algebra structure on $A_I$ compatible with the 
$R_{I+\fm^N}$-algebra structure on 
$A_{I+\fm^N}$ given by \mbox{Theorem 0.1} of
\cite{MR3415066} for all $N>0$.
\item The zero locus $V(J_{\min})\subset \Spec R$
contains the union of the closed toric strata corresponding to faces $F$ of 
$\sigma_P$ such that there exists 
$1 \leqslant j \leqslant r$ such that 
$[D_j] \notin F$. If $(Y,D)$ is positive, then 
$J_{\min}=0$ and $V(J_{\min})=\Spec R$.
\item Let $\hat{R}^{J_{\min}}$ denote the $J_{\min}$-adic completion of $R$. The algebras $A_I$ determine a $T^D$-equivariant formal flat family of affine surfaces 
$\cX^{J_{\min}} \rightarrow \Spf \hat{R}^{J_{\min}}$.
The theta functions $\vartheta_p$
determine a canonical embedding 
$\cX^{J_{\min}} \subset \mathbb{A}^{\max(r,3)}
\times \Spf \hat{R}^{J_{\min}}$.
In particular, if $(Y,D)$ is positive, then we get an algebraic family 
$\cX \rightarrow \Spec R$ and the theta functions
$\vartheta_p$ 
determine a canonical embedding 
$\cX \subset \A^{\max(r,3)}
\times \Spec R$. 
\end{itemize}

\subsection{Deformation quantization}

We now discuss the notion of deformation quantization. There are two technical aspects to keep in mind: first, we work relative to a non-trivial base; and 
second, we work in general with formal schemes. We refer to \cite{MR1855264}, \cite{MR2183259}, \cite{MR2119140}, 
for general facts about deformation
quantization in algebraic geometry.

\begin{defn} \label{def_poisson}
A Poisson scheme over a scheme $S$ is a scheme 
$\pi \colon X \rightarrow S$ over $S$, equipped with a 
$\pi^{-1} \cO_S$-bilinear Poisson bracket, 
that is, a 
$\pi^{-1} \cO_S$-bilinear skew-symmetric
map of sheaves
\[\{-,-\} \colon \cO_X \times \cO_X
\rightarrow \cO_X\,,\]
which is a biderivation
\[\{a,bc\}=\{a,b\}c+\{a,c\}b \,,\]
and a Lie bracket 
\[ \{a,\{b,c\}\}+
\{b,\{c,a\}\}
+\{c,\{a,b\}\}
=0 \,.\]
\end{defn}

The two definitions below give two notions of deformation quantization of a Poisson scheme.

\begin{defn} \label{def_quant_sheaf}
Let $\pi \colon (X,\{-,-\})
\rightarrow S$ be a Poisson scheme 
over a scheme $S$.
A \emph{deformation quantization}
of 
$(X,\{-,-\})$
over $S$
is a sheaf $\cO_X^{\hbar}$
of associative 
flat $\pi^{-1}
\cO_S \hat{\otimes} \kk_{\hbar}$-algebras on $X$, complete
in the $\hbar$-adic topology, 
equipped with an isomorphism
$\cO_X^{\hbar} / \hbar 
\cO_X^{\hbar}
\simeq  \cO_{X}$,
such that for every $f$ and 
$g$ in $\cO_X$, and lifts
$\tilde{f}$ and
$\tilde{g}$ of $f$ and $g$
in $\cO_X^{\hbar}$, we have 
\[ [\tilde{f},\tilde{g}]
=i\hbar \{f,g\} \mod \hbar^2 \,,\]
 where 
$[\tilde{f},
\tilde{g}]\coloneqq \tilde{f}
\tilde{g}-\tilde{g}\tilde{f}$ is the commutator in 
$\cO_X^{\hbar}$.
\end{defn}

\begin{defn} \label{def_quant_affine}
Let $\pi \colon (X,\{-,-\})
\rightarrow S$ be a
Poisson scheme over a scheme $S$.
Assume that both 
$X$ and $S$ are affine.
A \emph{deformation quantization} of
$(X,\{-,-\})$ over $S$  
is a 
flat 
$H^0(S,
\cO_S) 
\hat{\otimes }\kk_{\hbar}$-algebra $A$, complete
in the $\hbar$-adic topology, 
equipped with an isomorphism
$A / \hbar A \simeq  H^0(X, \cO_{X})$,
such that for every $f$ and 
$g$ in $H^0(X,\cO_X)$,
and lifts $\tilde{f}$ and 
$\tilde{g}$ 
of $f$ and $g$ in $A$,
we have 
\[[\tilde{f},\tilde{g}]
=i\hbar \{f,g\} \mod \hbar^2 \,,\]
 where 
$[\tilde{f},
\tilde{g}]
\coloneqq \tilde{f} \tilde{g}-\tilde{g}
\tilde{f}$ is the commutator in 
$A$.
\end{defn}

The compatibility of these two definitions is guaranteed by the following lemma.

\begin{lem} \label{lem_comp_affine}
When both $X$ and $S$ are affine, the notions of deformation quantization given by Definitions \ref{def_quant_sheaf} and \ref{def_quant_affine} are equivalent. 
\end{lem}

\begin{proof}
One goes from a sheaf quantization to an algebra quantization by taking
global sections. One goes from an algebra quantization to a sheaf quantization by Ore localization; see
\cref{section_ore}.
\end{proof}

Definitions 
\ref{def_quant_sheaf}
and \ref{def_quant_affine}
and Lemma \ref{lem_comp_affine} 
have obvious 
analogues if one replaces schemes by formal schemes\footnote{or, in fact, any locally ringed space.}.

\subsection{Main results}
We fix $(Y,D)$ a Looijenga pair 
and we use notation introduced 
in
\cref{subsection_ghk}.
Our main result, Theorem
\ref{thm_1_ch3}, is the construction of a deformation quantization of the Gross-Hacking-Keel mirror family by a higher-genus 
deformation of the Gross-Hacking-Keel
construction.

\begin{thm} \label{thm_1_ch3}
Let $I$ be a monomial ideal of $R$
with radical $\fm_R$. Then there exists 
a flat
$T^D$-equivariant $R_I^{\hbar}$-algebra $A_I^{\hbar}$, such that 
$A_I^{\hbar}$ is a deformation quantization over 
$S_I$ of the Gross-Hacking-Keel mirror family $X_I \rightarrow S_I$, 
and $A_I^{\hbar}$ is given as  
$R_I^{\hbar}$-module by 
\[ A_I^{\hbar} = \bigoplus_{p \in B(\Z)}
R_I^{\hbar} \hat{\vartheta}_p \,,\]
where the algebra structure is determined by higher-genus log Gromov--Witten invariants of $(Y,D)$, with genus expansion parameter identified with the quantization parameter $\hbar$.
\end{thm}

Taking the limit over all monomial ideals $I$ of $R$ with radical $\fm_R$, we get a deformation quantization of the formal family
\[\varinjlim_I X_I \rightarrow 
\varinjlim_I S_I \,.\]

The following theorem is 
a quantum version of
\cite[Theorem 0.2]{MR3415066}.

\begin{thm}\label{thm_2_ch3}
There is a unique smallest radical 
monomial $J_{\min}^{\hbar}
\subset R$ such that the following statements hold.
\begin{itemize}
\item For every monomial ideal $I$ of $R$ of radical containing $J_{\min}^{\hbar}$, there is a finitely generated $R_I^{\hbar}$-algebra structure on 
\[A_I^{\hbar}=\bigoplus_{p \in B(\Z)}R_I^{\hbar} \hat{\vartheta}_p \,,\] compatible with the 
$R_{I+\fm_R^k}^{\hbar}$-algebra structure on $A_{I+\fm_R^k}^{\hbar}$ given by Theorem 
\ref{thm_1_ch3} for all $k>0$.
\item The zero locus $V(J_{\min})
\subset R$ contains the union of the closed toric strata corresponding to faces $F$
of $\sigma_P$ such that there exists 
$1\leqslant j \leqslant r$ such that 
$[D_j]\notin F$. If $(Y,D)$ is positive, then $J_{\min}^{\hbar}=0$, that is,
$V(J_{\min}^{\hbar})=\Spec R$
and $A_0^{\hbar}$ is a deformation quantization of the mirror family 
$\cX \rightarrow \Spec R$.
\end{itemize} 
\end{thm} 

The following result controls the dependence 
in $\hbar$ of the deformation quantization given by Theorem 
\ref{thm_2_ch3}: this dependence is algebraic 
in $q=e^{i\hbar}$.

\begin{thm} \label{thm_q}
Let $I$ be a monomial ideal of $R$
with radical containing $J_{\min}^{\hbar}$. Then there exists 
a flat $R_I^q$-algebra $A_I^q$
such that
\[ A_I^{\hbar} = 
A_I^q \hat{\otimes}_{\kk_q} \kk_{\hbar} \,,\]
where
$\kk_{\hbar}$ is viewed as a $\kk_q$-module via $q=e^{i \hbar}$.
\end{thm}

The proof of Theorems 
\ref{thm_1_ch3}--\ref{thm_q}
is given in
\cref{section_scattering}--
\cref{section_ext}.
In \cref{section_scattering}, we explain how a consistent quantum scattering diagram can be used as input to a 
construction of quantum modified 
Mumford degeneration, giving a deformation quantization of the modified Mumford 
degeneration of 
\cite{MR3415066},
\cite{gross2016theta}, constructed from a
classical scattering diagram.
In \cref{section_canonical},
we explain how to construct a 
quantum scattering diagram from 
higher-genus log Gromov-Witten theory of a 
Looijenga pair and we prove its consistency using the main result of 
\cite{bousseau2018quantum_tropical}.
We finish the proof of Theorems 
\ref{thm_1_ch3}--\ref{thm_q} 
in \cref{section_ext}.

\section{Quantum modified Mumford degenerations}
\label{section_scattering}

In this section we explain how to construct a
quantization of the mirror family of a given Looijenga pair $(Y,D)$ starting from its tropicalization $(B, \Sigma)$ and a consistent quantum scattering diagram.

In \cref{subsection_building_blocks}
we describe the rings $R^{\hbar}_{\sigma,I}$ and 
$R^{\hbar}_{\rho,I}$ involved 
in the 
construction of the quantum version of 
modified Mumford degenerations.
In 
\cref{section_quantum_scattering_diagrams}
we review the notion of quantum scattering diagrams.
In 
\cref{subsection_gluing_ch3} we explain how 
a consistent quantum scattering diagram
gives a way to glue together the rings $R^{\hbar}_{\sigma,I}$ and 
$R^{\hbar}_{\rho,I}$ to produce
a quantum modified Mumford degeneration.
In 
\cref{section_broken_theta} we review the notions of quantum broken lines and theta functions and we use them in
\cref{section_deformation_quant_mirror}
to prove that the quantum modified Mumford 
degeneration is indeed a deformation
quantization of the modified Mumford 
degeneration of \cite{MR3415066}.
In \cref{section_algebra_structure} we express the structure constants of the 
quantum algebra of global sections in terms of quantum broken lines.

\subsection{Building blocks}
\label{subsection_building_blocks}
The goal of this section is 
to define non-commutative deformations 
$R_{\sigma,I}^{\hbar}$ and 
$R_{\rho,I}^{\hbar}$ of the rings 
$R_{\sigma,I}$ and 
$R_{\rho,I}$ defined in \cite[\S 2.1-2.2]{MR3415066}. 
The way to go from 
$R_{\sigma,I}$ to $R_{\sigma,I}^{\hbar}$
is fairly obvious.
The deformation $R_{\rho,I}^{\hbar}$
of $R_{\rho,I}$ is perhaps not so obvious.

We fix a Looijenga pair $(Y,D)$, 
its tropicalization $(B, \Sigma)$, 
a toric monoid $P$, a radical 
monomial ideal $J$ of $P$, and a 
$P_\R^{\mathrm{gp}}$-valued multivalued convex $\Sigma$-piecewise linear function
$\varphi$ on $B$.

For any locally constant sheaf $\cF$ on 
$B_0$ and any simply connected subset
$\tau$ of $B_0$, we write $\cF_\tau$ for the stalk of this local system at any point of 
$\tau$. We will constantly use this notation for $\tau$ a cone of $\Sigma$.

If $\tau$ is a cone of $\Sigma$, we define the localized fan $\tau^{-1} \Sigma$ as being the fan in $\Lambda_{\R, \tau}$ defined as follows.
\begin{itemize}
\item If $\tau$ is two-dimensional, then
$\tau^{-1} \Sigma$ consists just of the entire space $\Lambda_{\R, \tau}$.
\item If $\tau$ is one-dimensional, then 
$\tau^{-1} \Sigma$ consists of the tangent line 
of $\tau$ in $\Lambda_{\R, \tau}$
along with the two half-planes with boundary this tangent line.
\end{itemize}

For each $\tau$ cone of $\Sigma$, the 
$\Sigma$-piecewise $P$-convex function $\varphi \colon B_0 \rightarrow 
\B_{0,\varphi}$ determines a 
$\tau^{-1} \Sigma$-piecewise linear 
$P$-convex function 
$\varphi_\tau \colon 
\Lambda_{\R, \tau} \rightarrow \cP_{\R, \tau}$: if we choose 
$\sigma$ a two-dimensional cone of 
$\Sigma$ containing $\tau$, we have an identification 
$\cP_{\R,\tau}\simeq \Lambda_{\R,\tau} \oplus P_\R$,
and we define 
\[ \varphi_\tau \colon \Lambda_{\R,\tau}
\rightarrow \cP_{\R,\tau}=\Lambda_{\R,\tau} \oplus P_\R\]
\[ m \mapsto (m, \varphi|_\sigma(m))\,.\]
It follows from the definition of 
$\cP$ given in
\cref{section_tropicalization} that 
$\varphi_\tau$ is well defined, that is, independent of the choice of $\sigma$.
By construction, $\varphi_\tau 
\colon \Lambda_{\R,\tau}
\rightarrow \cP_{\R,\tau}$ is a section of the natural projection map
$\mathfrak{r} \colon \cP_{\R,\tau}
\rightarrow \Lambda_{\R,\tau}$
discussed in 
\cref{section_tropicalization}.

We define the toric monoid $P_{\varphi_\tau}
\subset \cP_{\tau}$ by 
\begin{equation*}
P_{\varphi_{\tau}} \coloneqq \{s \in \cP_{\tau}\,|\,\hbox{$s = p +
\varphi_{\tau}(m)$ for some $p\in P$, $m\in \Lambda_{\tau}$}\}.
\end{equation*}
If $\rho$ is a one-dimensional cone of $\Sigma$, bounding the 
two-dimensional cones 
$\sigma_+$ and $\sigma_-$ of 
$\Sigma$, we have $P_{\varphi_\rho}
\subset P_{\varphi_{\sigma_+}}$,
$P_{\varphi_\rho}
\subset P_{\varphi_{\sigma_-}}$, and  
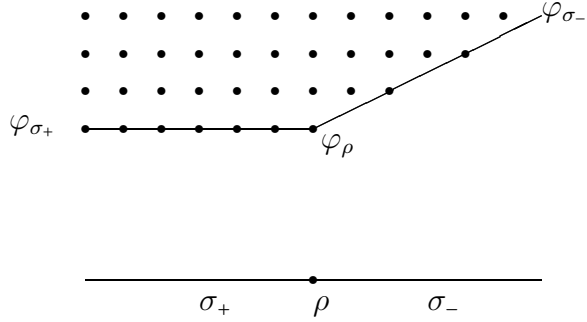
\[ P_{\varphi_{\sigma_+}} \cap P_{\varphi_{\sigma_-}}
=P_{\varphi_\rho} \,.\]
The monoids $P_{\varphi_{\sigma_+}}$,
$P_{\varphi_{\sigma_-}}$ and  
$P_{\varphi_\rho}$ are represented in 
Figures 1-3.

\begin{figure}[h]
\centering
\setlength{\unitlength}{1cm}
\begin{picture}(6,5)
\put(3,1){\circle*{0.1}}
\put(3,1){\line(1,0){3}}
\put(3,1){\line(-1,0){3}}
\put(3,0.6){$\rho$}
\put(1.5,0.6){$\sigma_+$}
\put(4.5,0.6){$\sigma_-$}
\put(-1,3){$\varphi_{\sigma_+}$}
\put(6,4.5){$\varphi_{\sigma_-}$}
\put(3.1,2.75){$\varphi_{\rho}$}
\put(3,3){\circle*{0.1}}
\put(3,3){\line(-1,0){3}}
\put(3,3){\line(2,1){3}}
\put(3,3){\line(2,1){3}}
\multiput(0,3)(0.5,0){12}
{\circle*{0.1}}
\multiput(0,3.5)(0.5,0){12}
{\circle*{0.1}}
\multiput(0,4)(0.5,0){12}
{\circle*{0.1}}
\multiput(0,4.5)(0.5,0){12}
{\circle*{0.1}}
\end{picture}
\caption{$P_{\varphi_{\sigma_+}}$}
\end{figure}

\begin{figure}[h]
\centering
\setlength{\unitlength}{1cm}
\begin{picture}(6,5)
\put(3,1){\circle*{0.1}}
\put(3,1){\line(1,0){3}}
\put(3,1){\line(-1,0){3}}
\put(3,0.6){$\rho$}
\put(1.5,0.6){$\sigma_+$}
\put(4.5,0.6){$\sigma_-$}
\put(-1,3){$\varphi_{\sigma_+}$}
\put(6,4.5){$\varphi_{\sigma_-}$}
\put(3.1,2.75){$\varphi_{\rho}$}
\put(3,3){\circle*{0.1}}
\put(3,3){\line(-1,0){3}}
\put(3,3){\line(2,1){3}}
\put(3,3){\line(2,1){3}}
\multiput(0,3)(0.5,0){6}
{\circle*{0.1}}
\multiput(0,3.5)(0.5,0){9}
{\circle*{0.1}}
\multiput(0,4)(0.5,0){11}
{\circle*{0.1}}
\multiput(0,4.5)(0.5,0){12}
{\circle*{0.1}}
\multiput(0,2.5)(0.5,0){5}
{\circle*{0.1}}
\multiput(0,2)(0.5,0){3}
{\circle*{0.1}}
\multiput(0,1.5)(0.5,0){1}
{\circle*{0.1}}
\end{picture}
\caption{$P_{\varphi_{\sigma_-}}$}
\end{figure}

\begin{figure}[h]
\centering
\setlength{\unitlength}{1cm}
\begin{picture}(6,5)
\put(3,1){\circle*{0.1}}
\put(3,1){\line(1,0){3}}
\put(3,1){\line(-1,0){3}}
\put(3,0.6){$\rho$}
\put(1.5,0.6){$\sigma_+$}
\put(4.5,0.6){$\sigma_-$}
\put(-1,3){$\varphi_{\sigma_+}$}
\put(6,4.5){$\varphi_{\sigma_-}$}
\put(3.1,2.75){$\varphi_{\rho}$}
\put(3,3){\circle*{0.1}}
\put(3,3){\line(-1,0){3}}
\put(3,3){\line(2,1){3}}
\put(3,3){\line(2,1){3}}
\multiput(0,3)(0.5,0){6}
{\circle*{0.1}}
\multiput(0,3.5)(0.5,0){9}
{\circle*{0.1}}
\multiput(0,4)(0.5,0){11}
{\circle*{0.1}}
\multiput(0,4.5)(0.5,0){12}
{\circle*{0.1}}
\end{picture}
\caption{$P_{\varphi_{\rho}}$}
\end{figure}

For every $\sigma$ two-dimensional cone of 
$\Sigma$, we define $R_{\sigma,I}^{\hbar}
\coloneqq
\kk_{\hbar}[P_{\varphi_\sigma}]/I$,
a
deformation quantization of 
$R_{\sigma,I} \coloneqq 
\kk [P_{\varphi_\sigma}]/I$.
We have a natural trivialization
$P_{\varphi_\sigma}=P \oplus \Lambda_\sigma$ and so 
$R_{\sigma,I}^{\hbar}$
is simply the algebra of functions on a trivial family of two-dimensional quantum tori parametrized by $\Spec R_I$.

Let $\rho$ be a one-dimensional cone of 
$\Sigma$. Let $\kappa_{\rho,\varphi}
\in P$ be the kink of $\varphi$ across
$\rho$, so that $z^{\kappa_{\rho,\varphi}}
\in R_I$. Let $X$ be an invertible formal variable. We fix elements
$\hat{f}_{\rho^\out}
\in R_I^{\hbar}[X^{-1}]$ and
$\hat{f}_{\rho^\iin}
\in R_I^{\hbar}[X]$.

Let $R^{\hbar}_{\rho,I}$ be the
$R_I^{\hbar}$-algebra generated by
formal variables 
$X_+$, $X_-$ and $X$, with $X$ invertible, and with relations
\[ X X_+ = q X_+ X \,,  \]
\[ X X_- = q^{-1} X_- X \,,\]
\[ X_+ X_-=q^{\frac{1}{2} D_\rho^2}
\hat{z}^{\kappa_{\rho,\varphi}} \hat{f}_{\rho^\out}(q^{-1} X)
\hat{f}_{\rho^\iin}(X)
X^{-D_{\rho}^2} \,,\]
\[X_- X_+=q^{-\frac{1}{2} D_\rho^2}
\hat{z}^{\kappa_{\rho,\varphi}} \hat{f}_{\rho^\out}(X)
\hat{f}_{\rho^\iin}(qX)
X^{-D_{\rho}^2} \,,
\]
where $q=e^{i \hbar}$.
The $R_I^{\hbar}$-algebra $R_{\rho,I}^{\hbar}$ is flat as $R_I^{\hbar}$-module and so is a deformation quantization of 
\[R_{\rho,I}
\coloneqq R_I[X_+,X_-,X^{\pm}]/(X_+ X_-
-z^{\kappa_{\rho,\varphi}}
X^{-D_\rho^2} f_{\rho^\out}(X)
f_{\rho^\iin}(X)) \,.\]

Let $\sigma_+$ and $\sigma_-$ be the two-dimensional cones of $\Sigma$ bounding $\rho$, and let 
$\rho_+$ and $\rho_-$ be the other boundary rays of $\sigma_+$ and 
$\sigma_-$ respectively, 
such that $\rho_-$, $\rho$ and 
$\rho_+$ are in anticlockwise order.

The precise form of $R_{\rho,I}^{\hbar}$
is justified by the following 
proposition.

\begin{prop} \label{prop_R}
The map of $R_I^{\hbar}$-algebras
\[\tilde{\psi}_{\rho,-} \colon
R_I^{\hbar}
\langle X_+,X_-,X^\pm \rangle
\rightarrow R_{\sigma_-,I}^{\hbar} \]
defined by 
\begin{align*}
\tilde{\psi}_{\rho,-}(X)&=
\hat{z}^{\varphi_\rho(m_\rho)}\,,\\
\tilde{\psi}_{\rho,-}(X_-)&=
\hat{z}^{\varphi_\rho(m_{\rho_-})}\,,\\
\tilde{\psi}_{\rho,-}(X_+)&=
\hat{f}_{\rho^\iin}(\hat{z}^{\varphi_\rho(m_\rho)})
\hat{z}^{\varphi_\rho(m_{\rho_+})}
\hat{f}_{\rho^\out}(\hat{z}^{\varphi_\rho(m_\rho)}) \\
&=
\hat{z}^{\varphi_\rho(m_{\rho_+})}
\hat{f}_{\rho^\iin}(q \hat{z}^{\varphi_\rho(m_\rho)})
\hat{f}_{\rho^\out}(\hat{z}^{\varphi_\rho(m_\rho)}) \,,
\end{align*}
induces a map of 
$R_I^{\hbar}$-algebras
\[\hat{\psi}_{\rho,-} \colon
R_{\rho,I}^{\hbar}
\rightarrow R_{\sigma_-,I}^{\hbar}\,. \]

The map of $R_I^{\hbar}$-algebras
\[\tilde{\psi}_{\rho,+} \colon 
R_I^{\hbar}
\langle X_+,X_-,X^\pm \rangle
\rightarrow R_{\sigma_+,I}^{\hbar} \,,\]
defined by 
\begin{align*}
\tilde{\psi}_{\rho,+}(X)&=
\hat{z}^{\varphi_\rho(m_\rho)}\,,\\
\tilde{\psi}_{\rho,+}(X_-)&=
\hat{f}_{\rho^\out}(\hat{z}^{\varphi_\rho(m_\rho)})
\hat{z}^{\varphi_\rho(m_{\rho_-})}
\hat{f}_{\rho^\iin}(\hat{z}^{\varphi_\rho(m_\rho)})\\
&=
\hat{z}^{\varphi_\rho(m_{\rho_-})}
\hat{f}_{\rho^\out}(q^{-1} \hat{z}^{\varphi_\rho(m_\rho)})
\hat{f}_{\rho^\iin}(\hat{z}^{\varphi_\rho(m_\rho)}) \,,\\
\tilde{\psi}_{\rho,+}(X_+)&=\hat{z}^{\varphi_\rho(m_{\rho_+})}
\,
\end{align*}
induces a map of 
$R_I^{\hbar}$-algebras
\[\hat{\psi}_{\rho,+} \colon
R_{\rho,I}^{\hbar}
\rightarrow R_{\sigma_+,I}^{\hbar}\,. \]
\end{prop}

\begin{proof}
We have to check that 
$\tilde{\psi}_{\rho,-}$ and
$\tilde{\psi}_{\rho,+}$
map the relations defining 
$R_{\rho,I}^{\hbar}$ to zero.

We have 
$\langle m_\rho, m_{\rho_+} \rangle=1$ and
$\langle m_\rho, m_{\rho_-} \rangle=-1$.
It follows that 
\[ \tilde{\psi}_{\rho,-}(X X_+ -qX_+ X)
=0\,,
\]
\[ \tilde{\psi}_{\rho,+}(X X_+ -qX_+ X) =0 \,,\]
and 
\[\tilde{\psi}_{\rho,-}(X X_- -q^{-1}X_-X)
=0\,,\]
\[
\tilde{\psi}_{\rho,+}(X X_- -q^{-1}X_-X)
=0 \,.\]

Furthermore, we have 
\[m_{\rho_-}+D_\rho^2 m_\rho + m_{\rho_+}
=0\] 
so 
\[\langle m_{\rho_+}, m_{\rho_-} \rangle =D_\rho^2\] 
and 
\[\varphi_\rho(m_{\rho_-})
+\varphi_\rho(m_{\rho_+})
=\kappa_{\rho,\varphi}-D_\rho^2 \varphi_\rho(m_\rho)\,.\]

It follows that 
\begin{align*}
\tilde{\psi}_{\rho,-}(X_+ X_-)
&=
\hat{f}_{\rho^\iin}(\hat{z}^{\varphi_\rho(m_\rho)})
\hat{z}^{\varphi_\rho(m_{\rho_+})}
\hat{f}_{\rho^\out}(\hat{z}^{\varphi_\rho(m_\rho)})
\hat{z}^{\varphi_\rho(m_{\rho_-})}\\
&=
q^{\frac{1}{2}D_\rho^2}
\hat{f}_{\rho^\iin}(\hat{z}^{\varphi_\rho(m_\rho)})
\hat{f}_{\rho^\out}(q^{-1} \hat{z}^{\varphi_\rho(m_\rho)})
\hat{z}^{\kappa_{\rho,\varphi}-D_\rho^2 \varphi_\rho(m_\rho)} \\
&=q^{\frac{1}{2} D_\rho^2}
\hat{z}^{\kappa_{\rho,\varphi}} \tilde{\psi}_{\rho,-}\left(
\hat{f}_{\rho^\iin}(X)
\hat{f}_{\rho^\out}(q^{-1} X)X^{-D_{\rho}^2}
\right) \,, \\
\end{align*}

\begin{align*}
\tilde{\psi}_{\rho,+}(X_+ X_-)
&=
\hat{z}^{\varphi_\rho(m_{\rho_+})}
\hat{f}_{\rho^\out}(\hat{z}^{\varphi_\rho(m_\rho)})\hat{z}^{\varphi_\rho(m_{\rho_-})}
\hat{f}_{\rho^\iin}(\hat{z}^{\varphi_\rho(m_\rho)})\\
&=q^{\frac{1}{2}D_\rho^2}
\hat{f}_{\rho^\iin}(\hat{z}^{\varphi_\rho(m_\rho)})
\hat{f}_{\rho^\out}(q^{-1} \hat{z}^{\varphi_\rho(m_\rho)})
\hat{z}^{\kappa_{\rho,\varphi}-D_\rho^2 \varphi_\rho(m_\rho)} \\
&=q^{\frac{1}{2} D_\rho^2}
\hat{z}^{\kappa_{\rho,\varphi}} \tilde{\psi}_{\rho,+}\left(
\hat{f}_{\rho^\iin}(X)
\hat{f}_{\rho^\out}(q^{-1} X)X^{-D_{\rho}^2}
\right) \,,
\end{align*}
and
\begin{align*}
\tilde{\psi}_{\rho,-}(X_-X_+)
&=
\hat{z}^{\varphi_\rho(m_{\rho_-})}
f_{\rho^\iin}(\hat{z}^{\varphi_\rho(m_\rho)})
\hat{z}^{\varphi_\rho(m_{\rho_+})}
f_{\rho^\out}(\hat{z}^{\varphi_\rho(m_\rho)})
\\
&=q^{-\frac{1}{2}D_\rho^2}
\hat{z}^{\kappa_{\rho,\varphi}-D_\rho^2 \varphi_\rho(m_\rho)}
f_{\rho^\iin}(\hat{z}^{q \varphi_\rho(m_\rho)})
f_{\rho^\out}( \hat{z}^{\varphi_\rho(m_\rho)})
\\
&= q^{-\frac{1}{2} D_\rho^2}
\hat{z}^{\kappa_{\rho,\varphi}} \tilde{\psi}_{\rho,-} 
\left(
\hat{f}_{\rho^\iin}(qX)
\hat{f}_{\rho^\out}(X)
X^{-D_{\rho}^2} \right)\,,
\end{align*}

\begin{align*}
\tilde{\psi}_{\rho,+}(X_-X_+)
&=\hat{f}_{\rho^\out}(\hat{z}^{\varphi_\rho(m_\rho)})
\hat{z}^{\varphi_\rho(m_{\rho_-})}
f_{\rho^\iin}(\hat{z}^{\varphi_\rho(m_\rho)})
\hat{z}^{\varphi_\rho(m_{\rho_+})}
\\
&=q^{-\frac{1}{2}D_\rho^2}\hat{f}_\rho( \hat{z}^{\varphi_\rho(m_\rho)})
f_{\rho^\iin}(q \hat{z}^{\varphi_\rho(m_\rho)})
\hat{z}^{\kappa_{\rho,\varphi}-D_\rho^2 \varphi_\rho(m_\rho)}
\\
&= q^{-\frac{1}{2} D_\rho^2}
\hat{z}^{\kappa_{\rho,\varphi}} \tilde{\psi}_{\rho,+} 
\left( 
\hat{f}_{\rho^\iin}(qX)
\hat{f}_{\rho^\out}(X)
X^{-D_{\rho}^2} \right)\,.
\end{align*}
\end{proof}

In the special case where $D_{\rho}^2=0$
and 
$\hat{f}_{\rho^\iin}
=1$, our description of $R_{\rho,I}^{\hbar}$ by generators and relations coincides with the description given by Soibelman in \cite[\S 7.5]{MR2596639} of a local model for deformation quantization of a neighborhood of a focus-focus singularity.

The algebra $R_{\sigma,I}^{\hbar}$
is a deformation quantization of 
$R_{\sigma,I}$, 
and the algebra $R_{\rho,I}^{\hbar}$
is a deformation quantization of 
$R_{\rho,I}$. The maps
$\hat{\psi}_{\rho,+}$ and 
$\hat{\psi}_{\rho,-}$ are quantizations of the maps $\psi_{\rho,-}$
and $\psi_{\rho,+}$
defined by 
\cite[formula (2.8)]{MR3415066}. Following \cite{MR3415066}, we denote 
$U_{\sigma,I}
\coloneqq \Spec R_{\sigma,I}$
and $U_{\rho,I}
\coloneqq \Spec R_{\rho,I}$.
If $\rho$ is a one-dimensional cone of $\Sigma$,
and $\sigma_+$ and $\sigma_-$
are the two-dimensional cones of 
$\Sigma$ bounding $\rho$, then the maps 
$\psi_{\rho,-}$ and 
$\psi_{\rho,+}$ induce
open immersions 
\[ U_{\sigma -,I}
\hookrightarrow U_{\rho,I} \]
and 
\[ U_{\sigma +,I}
\hookrightarrow U_{\rho,I}\,.\]

Using Ore localization
(see Definition \ref{def_loc}), we can 
produce from $R_{\sigma,I}^{\hbar}$ and 
$R_{\rho,I}^{\hbar}$ some sheaves
of flat $\kk_{\hbar}$-algebras 
$\cO_{U_{\sigma,I}}^{\hbar}$ and 
$\cO_{U_{\rho,I}}^{\hbar}$
on $U_{\sigma,I}$ and 
$U_{\rho,I}$, such that 
\[ \cO_{U_{\sigma,I}}^{\hbar}/
\hbar \cO_{U_{\sigma,I}}^{\hbar}
\simeq \cO_{U_{\sigma,I}}\]
and 
\[ \cO_{U_{\rho,I}}^{\hbar}/
\hbar \cO_{U_{\rho,I}}^{\hbar}
\simeq \cO_{U_{\rho,I}}\]
respectively.

\subsection{Quantum scattering diagrams}
\label{section_quantum_scattering_diagrams}

Quantum scattering diagrams have been studied by Filippini and Stoppa \cite{MR3383167} in dimension two
and by Mandel
\cite{mandel2015scattering} in higher dimensions. Mandel \cite{mandel2015scattering} also studied quantum broken lines and quantum theta functions. Both 
\cite{MR3383167} and \cite{mandel2015scattering}
work with smooth integral affine manifolds.
We need to make some changes to include the case we care about, where the integral affine manifold is the tropicalization 
$B$ of a Looijenga pair and has a singularity at the origin with a non-trivial monodromy around it. 

As in the previous section, 
we fix $(Y,D)$ a Looijenga pair, 
its tropicalization
$(B, \Sigma)$, a toric monoid $P$, a radical 
monomial ideal $J$ of $P$, and a 
$P_\R^{\mathrm{gp}}$-valued multivalued
convex $\Sigma$-piecewise linear 
function $\varphi$ on $B$.
Recall from 
\cref{section_tropicalization} that we then have an exact sequence
\[ 0 \rightarrow \underline{P}^{\gp}
\rightarrow \cP \xrightarrow{\mathfrak{r}} 
\Lambda \rightarrow 0\,\] 
of locally constant sheaves on $B_0$.

We explained in  \cref{subsection_building_blocks} how to define for every cone $\tau$ of $\Sigma$ 
a toric monoid 
$P_{\varphi_\tau}$.
We denote by 
\[ \kk_{\hbar}\widehat{[P_{\varphi_{\tau}}]}\]
the $J$-adic completion of the
$\kk_{\hbar}$-algebra
$\kk_{\hbar}[P_{\varphi_\tau}]$.
The map $\mathfrak{r} \colon \cP \rightarrow \Lambda$
induces a morphism of monoids 
$\mathfrak{r} \colon P_{\varphi_\tau}
\rightarrow \Lambda_\tau$.

\begin{defn}
\label{def_quantum_scattering}
A \emph{quantum scattering diagram} $\hat{\fD}$ for the data 
$(B,\Sigma)$, $P$, $J$ and $\varphi$
is a set 
\[\hat{\fD}=\{(\fd,\hat{H}_{\fd})\}\]
where
\begin{itemize}
\item $\fd \subset B$ is a ray of rational slope in $B$ with endpoint the origin 
$0 \in B$.
\item Let 
$\tau_\fd$ be the smallest cone of $\Sigma$
containing $\fd$ and let $m_\fd \in \Lambda_{\tau_\fd}$ be the primitive generator of $\fd$ pointing away from the origin. Then we have
either 
\[ \hat{H}_\fd = \sum_{\substack{p \in P_{\varphi_{\tau_\fd}} \\ \mathfrak{r}(p) \in \Z_{< 0} m_\fd}} H_p \hat{z}^p \in \kk_{\hbar}\widehat{[P_{\varphi_{\tau_\fd}}]}  \,,\]
or 
\[ \hat{H}_\fd = \sum_{\substack{p \in P_{\varphi_{\tau_\fd}} \\ \mathfrak{r}(p) \in \Z_{> 0} m_\fd}} H_p \hat{z}^p \in \kk_{\hbar}\widehat{[P_{\varphi_{\tau_\fd}}]}  \,.\]
In the first case, we say that the ray 
$(\fd, \hat{H}_{\fd})$ 
is \emph{outgoing}, and 
in the second case, we say that the 
ray 
$(\fd, \hat{H}_{\fd})$ is \emph{ingoing}.
\item Let $\tau_{\fd}$ be the smallest cone of $\Sigma$
containing $\fd$. If $\dim \tau_\fd =2$, or if 
$\dim \tau_\fd=1$ and 
$\kappa_{\tau_\fd,\varphi} \notin J$, then
$\hat{H}_\fd =0 \mod J$.
\item For any ideal $I \subset P$ of radical $J$, there are only finitely many rays 
$(\fd, \hat{H}_\fd)$
such that $\hat{H}_\fd \neq 0 \mod I$.

\end{itemize}
\end{defn}

Given a ray 
$(\fd, \hat{H}_\fd)$ of a quantum scattering diagram, we call 
$\hat{H}_\fd$ the Hamiltonian attached to $\rho$. This terminology is justified by
\cref{section_aut_ch3}, where we attach to 
$(\fd, \hat{H}_\fd)$ the automorphism
$\hat{\Phi}_{\hat{H}_{\fd}}$ given by the time-one evolution according to the quantum Hamiltonian 
$\hat{H}_\fd$.

\subsection{Quantum automorphisms}
\label{section_aut_ch3}

Let $(\fd,\hat{H}_{\fd})$ be a ray of a quantum scattering diagram $\hat{\fD}$ for the data 
$(B,\Sigma)$, $P$, $J$ and $\varphi$.
Let $\tau_{\fd}$ be the smallest cone of 
$\Sigma$ containing $\fd$ and let 
$m_\fd \in \Lambda_{\tau_\fd}$
be the primitive generator of $\fd$ 
pointing away from the origin.
Denote $m(\hat{H}_\fd)=m_\fd$
if $(\fd,\hat{H}_\fd)$ is outgoing 
and $m(\hat{H}_\fd)=-m_\fd$
if $(\fd,\hat{H}_\fd)$ is ingoing.
Writing
\[\hat{H}_{\fd}=\sum_{p \in P_{\varphi_{\tau_\fd}}} H_p \hat{z}^p \in 
\kk_{\hbar}\widehat{[P_{\varphi_{\tau_\fd}}]} \,,\]
we denote 
\[\hat{f}_{\fd}\coloneqq \exp
\left( 
\sum_{\substack{p \in P_{\varphi_{\tau_\fd}} \\ \mathfrak{r}(p) =\ell
m(\hat{H}_\fd)
}}
(q^\ell-1)H_p \hat{z}^p 
\right) \in \kk_{\hbar}\widehat{[P_{\varphi_{\tau_\fd}}]}\,, \]
where $q=e^{i \hbar}$. Remark that, by our definition of
$m(\hat{H}_\fd)$, we have $\ell \leqslant 0$ when writing 
$\mathfrak{r}(p)=\ell m(\hat{H}_\fd)$.

We write
\[\hat{f}_\fd = \sum_{p \in P_{\varphi_{\tau_\fd}}} 
f_{p}\hat{z}^p \,.\]
For every $j \in \Z$, we define
\[ \hat{f}_\fd(q^j \hat{z})
\coloneqq \sum_{
\substack{p \in P_{\varphi_{\tau_\fd}} \\ \mathfrak{r}(p) =\ell 
m(\hat{H}_\fd)
}} q^{\ell j} f_{p} \hat{z}^p
\in \kk_{\hbar}\widehat{[P_{\varphi_{\tau_\fd}}]}\,,\]
where $q=e^{i\hbar}$.

\begin{lem} \label{lem_autom}
The automorphism 
$\hat{\Phi}_{\hat{H}_\fd}$ of 
$\kk_{\hbar}\widehat{[P_{\varphi_{\tau_\fd}}]}$
given by conjugation by $\exp \left(\hat{H}_\fd \right)$,
\[ \hat{z}^p \mapsto 
\exp \left( \hat{H}_\fd \right)
\hat{z}^p \exp 
\left( -\hat{H}_\fd \right) \,,\]
is equal to 
\[
\hat{z}^p \mapsto
\begin{dcases}
\hat{z}^p \prod_{j=0}^{\langle m(\hat{H}_\fd), \mathfrak{r}(p) \rangle -1}
\hat{f}_\fd(q^j \hat{z}) & \text{if  } \langle m(\hat{H}_\fd), \mathfrak{r}(p) \rangle \geqslant 0 \\
\hat{z}^p \prod_{j=0}^{|\langle 
m(\hat{H}_\fd)
,\mathfrak{r}(p) \rangle| -1} 
\hat{f}_{\fd}( q^{-j-1}\hat{z})^{-1} & \text{if  } 
\langle m(\hat{H}_\fd), \mathfrak{r}(p) \rangle <0 \,.
\end{dcases}\]
\end{lem}

\begin{proof}
Using 
$\hat{z}^{p'} \hat{z}^p
=q^{\langle \mathfrak{r}(p'), \mathfrak{r}(p) \rangle}
\hat{z}^p \hat{z}^{p'}$, we get
\[
\exp \left( \hat{H}_\fd \right)
\hat{z}^p \exp 
\left( -\hat{H}_\fd \right) 
= \hat{z}^p \exp \left( 
\sum_{
\substack{p' \in P_{\varphi_{\tau_\fd}} \\ \mathfrak{r}(p') =\ell m(\hat{H}_\fd)}} (q^{\ell \langle m(\hat{H}_\fd),\mathfrak{r}(p)
\rangle}-1)H_{p'} \hat{z}^{p'}
\right) \,.\]
If $\langle m(\hat{H}_\fd), \mathfrak{r}(p) \rangle \geqslant 0$, 
this can be written
\[ 
 \hat{z}^p \exp \left( 
\sum_{
\substack{p' \in P_{\varphi_{\tau_\fd}} \\ \mathfrak{r}(p') =\ell m(\hat{H}_\fd)}} \frac{1-q^{\ell \langle m(\hat{H}_\fd), \mathfrak{r}(p)
\rangle}}{1-q^\ell}(q^\ell-1)H_{p'} \hat{z}^{p'}
\right)\] 
\[=\hat{z}^p \exp \left( 
\sum_{
\substack{p' \in P_{\varphi_{\tau_\fd}} \\ \mathfrak{r}(p') =\ell m(\hat{H}_\fd)}} \sum_{j=0}^{\langle m(\hat{H}_\fd), \mathfrak{r}(p)
\rangle -1}
q^{\ell j} 
(q^\ell-1)H_{p'} \hat{z}^{p'}
\right)\] 
\[=\hat{z}^p \prod_{j=0}^{\langle m(\hat{H}_\fd), \mathfrak{r}(p) \rangle -1}
\hat{f}_\fd(q^j \hat{z})\,.\]
If $\langle m(\hat{H}_\fd), \mathfrak{r}(p) \rangle < 0$, 
this can be written
\[ 
 \hat{z}^p \exp \left( -
\sum_{
\substack{p' \in P_{\varphi_{\tau_\fd}} \\ \mathfrak{r}(p') =\ell m(\hat{H}_\fd)}} 
\frac{1-q^{-\ell |\langle m(\hat{H}_\fd), \mathfrak{r}(p)
\rangle|}}{1-q^{-\ell}}q^{-\ell}(q^\ell-1)H_{p'} \hat{z}^{p'}
\right)\] 
\[=\hat{z}^p \exp \left( -
\sum_{
\substack{p' \in P_{\varphi_{\tau_\fd}} \\ \mathfrak{r}(p') =\ell m(\hat{H}_\fd)}} 
\sum_{j=0}^{|\langle m(\hat{H}_\fd), \mathfrak{r}(p)
\rangle| -1}
(q^{-j-1})^\ell 
(q^\ell-1)H_{p'} \hat{z}^{p'}
\right)\] 
\[=\hat{z}^p \prod_{j=0}^{|\langle m(\hat{H}_\fd) , \mathfrak{r}(p) \rangle| -1} 
\hat{f}_{\fd}( q^{-j-1}\hat{z})^{-1} \,.\]
\end{proof}

One can equivalently write $\hat{\Phi}_{\hat{H}_\fd}$ as
\[
\hat{z}^p \mapsto 
\begin{dcases}
\left( \prod_{j=0}^{\langle m(\hat{H}_\fd), \mathfrak{r}(p) \rangle -1}
\hat{f}_\rho(q^{-j-1} z) \right) \hat{z}^p & \text{if  } \langle m(\hat{H}_\fd), \mathfrak{r}(p) \rangle \geqslant 0 \\
\left( \prod_{j=0}^{|\langle m(\hat{H}_\fd),
\mathfrak{r}(p) \rangle| -1} 
\hat{f}_\rho( q^{j}z)^{-1} \right) \hat{z}^p & \text{if  } 
\langle m(\hat{H}_\fd), \mathfrak{r}(p) \rangle <0 \,.
\end{dcases}\]

A direct application of the definition of 
$\hat{f}_\fd$ gives the following Lemma.

\begin{lem}\label{lem_ex_ch3}
If
\[\hat{H}=i \sum_{\ell \geqslant 1} \frac{(-1)^{\ell-1}}{\ell}\frac{\hat{z}^{-\ell \varphi (m_\fd)}}{2 \sin \left( \frac{\ell \hbar}{2} \right)}  
=
-\sum_{\ell \geqslant 1} \frac{(-1)^{\ell-1}}{\ell}\frac{\hat{z}^{-\ell\varphi(m_\fd)}}{q^{\frac{\ell}{2}}
-q^{-\frac{\ell}{2}}} \,,\]
where $q=e^{i\hbar}$,
we have $m(\hat{H})=m_\fd$ and
\[\hat{f} = \exp \left( - 
\sum_{\ell \geqslant 1} \frac{(-1)^{\ell-1}}{\ell}
\frac{q^{-\ell}-1}{q^{\frac{\ell}{2}}
-q^{-\frac{\ell}{2}}} \hat{z}^{-\ell \varphi (m_\fd)}
\right)= 1+ q^{-\frac{1}{2}}\hat{z}^{-\varphi(m_\fd)} \,.\]

If
\[\hat{H}=i \sum_{\ell \geqslant 1} \frac{(-1)^{\ell-1}}{\ell}\frac{\hat{z}^{\ell\varphi (m_\fd)}}{2 \sin \left( \frac{\ell \hbar}{2} \right)}  
=
-\sum_{\ell \geqslant 1} \frac{(-1)^{\ell-1}}{\ell}\frac{\hat{z}^{\ell \varphi(m_\fd)}}{q^{\frac{\ell}{2}}
-q^{-\frac{\ell}{2}}} \,,\]
where $q=e^{i\hbar}$,
we have $m(\hat{H})=-m_\fd$ and
\[\hat{f} = \exp \left( - 
\sum_{\ell \geqslant 1} \frac{(-1)^{\ell-1}}{\ell}
\frac{q^{-\ell}-1}{q^{\frac{\ell}{2}}
-q^{-\frac{\ell}{2}}} \hat{z}^{\ell \varphi (m_\fd)}
\right)= 1+ q^{-\frac{1}{2}}\hat{z}^{\varphi(m_\fd)} \,.\]
\end{lem}

\subsection{Gluing}
\label{subsection_gluing_ch3}

We fix a quantum scattering diagram $\hat{\fD}$
for the data $(B,\Sigma)$, $P$, $J$ and $\varphi$, and an ideal $I$ of radical $J$.

Let $\rho$ be a one-dimensional cone of 
$\Sigma$, bounding the two-dimensional cones 
$\sigma_+$ and $\sigma_-$, such that 
$\sigma_-$, $\rho$, $\sigma_+$ are 
in anticlockwise order.
Identifying $X$ with $\hat{z}^{\varphi_\rho
(m_\rho)}$,
we define $\hat{f}_{\rho^\out} \in R_{I}^{\hbar}[X^{-1}]$ by
\[ \hat{f}_{\rho^\out} \coloneqq \prod_{
\substack{\fd \in \hat{\fD}, \fd=\rho
\\ \text{outgoing}}}
\hat{f}_\fd \mod I \,,\]
where the product is over the outgoing 
rays of $\hat{\fD}$ of support $\rho$,
and we define 
$\hat{f}_{\rho^\iin} \in R_{I}^{\hbar}[X]$ by
\[ \hat{f}_{\rho^\iin} \coloneqq \prod_{
\substack{\fd \in \hat{\fD}, \fd=\rho
\\ \text{ingoing}}}
\hat{f}_\fd \mod I \,,\]
where the product is over the ingoing 
rays of $\hat{\fD}$ of support $\rho$.

By  \cref{subsection_building_blocks}, we then have $R_I^{\hbar}$-algebras 
$R^{\hbar}_{\sigma_+,I}$,
$R^{\hbar}_{\sigma_-,I}$,
$R^{\hbar}_{\rho,I}$.

Let $(\fd, \hat{H}_{\fd})$ be a ray of 
$\hat{\fD}$ 
such that $\tau_\fd=\sigma$ is a
two-dimensional cone of 
$\Sigma$.
Let $m_\fd \in \Lambda_{\tau_\fd}$ be the primitive generator of $\fd$ pointing away from the origin.
Let $\gamma$ be a path in 
$B_0$ which crosses $\fd$ transversally at time $t_0$. We define
\[ \hat{\theta}_{\gamma, \fd} 
\colon R^{\hbar}_{\sigma, I}
\rightarrow R^{\hbar}_{\sigma,I} \,,\]
\[ \hat{z}^p \mapsto \hat{\Phi}_{\hat{H}_\fd}^\epsilon (\hat{z}^p) \,,\]
where $\epsilon \in \{ \pm 1\}$
is the sign of 
$-\langle m(\hat{H}_\fd), \gamma'(t_0) \rangle$.

Let $\hat{\fD}_I \subset \hat{\fD}$ be the finite set of rays $(\fd, \hat{H}_\fd)$ with 
$\hat{H}_\fd \neq 0 \mod I$,
that is,
$\hat{f}_\fd \neq 1 \mod I$.
If $\gamma$ is a path in $B_0$
entirely contained in the interior of a two-dimensional cone $\sigma$ of $\Sigma$, and crossing elements of $\fD_I$ transversally, we define 
\[ \hat{\theta}_{\gamma, \hat{\fD}_I}
\coloneqq \hat{\theta}_{\gamma, \fd_n} \circ 
\dots \circ \hat{\theta}_{\gamma, \fd_1} \,, \]
where $\gamma$ crosses the elements
$\fd_1, \dots, \fd_n$ of $\hat{\fD}_I$
in the given order.

For every $\sigma$ two-dimensional cone of $\Sigma$, bounded by rays $\rho_R$ and $\rho_L$, such that 
$\rho_R,\sigma,\rho_L$ are in  
anticlockwise order, we choose a path 
$\gamma_\sigma \colon [0,1]
\rightarrow B_0$ whose image is 
entirely contained in the interior of 
$\sigma$, with 
$\gamma(0)$ close to $\rho_R$ and $\gamma(1)$ close to $\rho_L$, such that $\gamma_\sigma$ crosses every ray of $\hat{\fD}_I$ contained in $\sigma$ 
transversally exactly once.
Let 
\[ \hat{\theta}_{\gamma_\sigma,\hat{\fD}_I}
\colon R_{\sigma,I}^{\hbar}
\rightarrow R_{\sigma,I}^{\hbar} \]
be the corresponding automorphism.
In the classical limit, $\hat{\theta}_{\gamma, \hat{\fD}_I}$
induces an automorphism
$\theta_{\gamma,\fD_I}$ of $U_{\sigma,I}$.
Gluing together the open sets 
$U_{\sigma,I} \subset U_{\rho_R,I}$ and 
$U_{\sigma,I} \subset U_{\rho_L,I}$ along these automorphisms, we get the scheme $X_{I,\fD}^\circ$ defined in
\cite{MR3415066}. 

Recall from the end of 
\cref{subsection_building_blocks}
that by Ore localization the algebras 
$R_{\sigma,I}^{\hbar}$ and 
$R_{\rho,I}^{\hbar}$
produce sheaves 
$\cO_{U_{\sigma,I}}^{\hbar}$
and $\cO_{U_{\rho,I}}^{\hbar}$
on $U_{\sigma,I}$ and 
$U_{\rho,I}$
respectively.
Using 
$\hat{\theta}_{\gamma_\sigma, \hat{\fD}_I}$, we can glue together the sheaves $\cO_{U_{\rho,I}}^{\hbar}$
to get a sheaf of $R_I^{\hbar}$-algebras
$\cO_{X_{I,\fD}^\circ}^{\hbar}$ on 
$X_{I,\fD}^\circ$.

From the fact that the sheaves
$\cO_{U_{\rho,I}}^{\hbar}$
are deformation quantizations of 
$U_{\rho,I}$, we deduce that the sheaf
$\cO_{X_{I,\fD}^\circ}^{\hbar}$
is a deformation quantization of 
$X_{I,\fD}^\circ$.
In particular, we have $\cO^{\hbar}_{X_{I,\fD}^\circ}/
\hbar 
\cO^{\hbar}_{X_{I,\fD}^\circ}
= \cO_{X_{I,\fD}^\circ}$
and $\cO^{\hbar}_{X_{I,\fD}^\circ}$
is a sheaf a flat 
$R_I^{\hbar}$-algebras.

Let $\rho$ be a one-dimensional cone
of $\Sigma$.
Let $\sigma_+$ and 
$\sigma_-$ be the two two-dimensional cones 
of $\Sigma$ bounding $\rho$, and let 
$\rho_+$ and $\rho_-$ be the other boundary rays of $\sigma_+$ and $\sigma_-$
respectively, such that 
$\rho_-$, $\rho$ and
$\rho_+$ are in anticlockwise order.
According to \cite[Remark 2.6]{MR3415066},
we have, in $U_{\rho,I}$,
\[ U_{\rho_-,I} \cap U_{\rho_+,I}
\simeq (\mathbb{G}_m)^2 \times \Spec (R_I)_{z^{\kappa_{\rho,\varphi}}}\,,\]
where $(R_I)_{z^{\kappa_{\rho,\varphi}}}$
is the localization of $R_I$ defined by inverting $z^{\kappa_{\rho,\varphi}}$.
Similarly, the restriction of 
$\cO_{X_{I,\fD}^\circ}^{\hbar}$
to $U_{\rho_-,I} \cap U_{\rho_+,I}$
is the Ore localization of 
$\kk_{\hbar}[M] \hat{\otimes} (R_I)_{z^{\kappa_{\rho,\varphi}}}$,
where $M=\Z^2$ is the character lattice of 
$(\mathbb{G}_m)^2$, equipped with the standard unimodular integral symplectic pairing. We have a natural identification 
$M=\Lambda_\rho$. Restricted to 
$\kk_{\hbar}[M] \hat{\otimes} (R_I)_{z^{\kappa_{\rho,\varphi}}}$, and assuming that 
$\hat{f}_{\rho^\iin} = 1 \mod \hat{z}^{\kappa_{\rho,\varphi}}$
and
\mbox{$\hat{f}_{\rho^\out} = 1 \mod \hat{z}^{\kappa_{\rho,\varphi}}$}, the expression $\hat{\psi}_{\rho_+} \circ \hat{\psi}_{\rho_-}^{-1}$ makes sense\footnote{Without restriction,
$\hat{\psi}_{\rho_-}$ is not invertible and so $\hat{\psi}_{\rho_-}^{-1}$ does not make sense. } and is given by 
\[(\hat{\psi}_{\rho_+} \circ \hat{\psi}_{\rho_-}^{-1}) 
(\hat{z}^{\varphi_\rho(m_\rho)})=
\hat{z}^{\varphi_\rho(m_\rho)} \,,\]
\[(\hat{\psi}_{\rho_+} \circ \hat{\psi}_{\rho_-}^{-1}) 
(\hat{z}^{\varphi_\rho(m_{\rho_-})})
= 
\hat{f}_{\rho^\out}
(\hat{z}^{\varphi_\rho(m_\rho)} ) 
\hat{z}^{\varphi_\rho(m_{\rho_-})}
\hat{f}_{\rho^\iin}
(\hat{z}^{\varphi_\rho(m_\rho)} )
\,,\]
\[(\hat{\psi}_{\rho_+} \circ \hat{\psi}_{\rho_-}^{-1})
(\hat{z}^{\varphi_\rho(m_{\rho_+})}) 
=
\hat{f}_{\rho^\iin}^{-1}
(\hat{z}^{\varphi_\rho(m_\rho)} )
\hat{z}^{\varphi_\rho(m_{\rho_+})}
\hat{f}_{\rho^\out}^{-1}
(\hat{z}^{\varphi_\rho(m_\rho)} )\,.\]
As $\langle m_{\rho},m_{\rho_-}\rangle
=-1$ and 
$\langle m_{\rho},m_{\rho_+} \rangle=1$,
this implies that $\hat{\psi}_{\rho_+} \circ \hat{\psi}_{\rho_-}^{-1}$
coincides with the transformation
\[\hat{\theta}_{\gamma,\rho}
= \prod_{\fd \in \fD, \fd=\rho}
\hat{\theta}_{\gamma,\fd},,\] 
where $\hat{\theta}_{\gamma,\fd}$
is defined by the same formulas as above and with $\gamma$ a path intersecting $\rho$ at a single point and going from 
$\sigma_-$ to $\sigma_+$.

\subsection{Result of the gluing for $I=J$.}
\label{gluing_I=J}

Assume $r \geqslant 3$
and $\kappa_{\rho,\varphi} \in J$ for every $\rho$ one-dimensional cone of 
$\Sigma$.
The Lemma
\ref{lem_special_fiber} below
gives an explicit 
description of $\cO^{\hbar}_{X_{I,\fD}^\circ}$ for $I=J$.

Denote by $\kk[\Sigma]$ the 
$\kk$-algebra with a $\kk$-basis
$\{z^m \,|m \in B(\Z)\}$ with multiplication given by
\[z^m \cdot z^{m'} =
\begin{dcases}
z^{m+m'} & \text{if $m$ and $m'$ lie in a common cone of $\Sigma$} \\
0 & \text{otherwise.} 
\end{dcases}\]
Let $0$ be the closed point of
$\Spec \kk[\Sigma]$ whose ideal is generated by 
$\{z^m \,|m \neq 0\}$.
Denote $R_J[\Sigma]
\coloneqq R_J \otimes_{\kk} \kk[\Sigma]$.
According to 
\cite[Lemma 2.12]{MR3415066}, we have
\[ X_J^\circ \simeq  
(\Spec R_J[\Sigma])-((\Spec R_J)
\times \{0\})\,.\]

Denote by $\kk_{\hbar}[\Sigma]$ the 
$\kk_{\hbar}$-algebra with a $\kk_{\hbar}$-basis
$\{\hat{z}^m \,|m \in B(\Z)\}$ with multiplication given by
\[\hat{z}^m \cdot \hat{z}^{m'} =
\begin{dcases}
q^{\frac{1}{2} \langle m,m'\rangle}
\hat{z}^{m+m'} & \text{if $m$ and $m'$ lie in a common cone of $\Sigma$} \\
0 & \text{otherwise.} 
\end{dcases}\]
Denote $R_J^{\hbar}[\Sigma]
\coloneqq R_J \hat{\otimes}_{\kk} 
\kk_{\hbar}[\Sigma]$.

\begin{lem}\label{lem_special_fiber}
Assume $r \geqslant 3$
and $\kappa_{\rho,\varphi} \in J$ for every $\rho$ one-dimensional cone of 
$\Sigma$.
Then $\Gamma (X_{J,\fD}^\circ, 
\cO_{X_{J,\fD}^\circ}^{\hbar})
=R_J^{\hbar}[\Sigma]$,
and the sheaf $\cO_{X_{J,\fD}^\circ}^{\hbar}$
is the restriction to 
$X_J^\circ$ of the Ore localization
(see \cref{section_ore}) 
of $R_J^{\hbar}[\Sigma]$
over $\Spec R_J[\Sigma]$.
\end{lem}

\begin{proof}
By definition of a quantum scattering diagram, 
if $\fd$ is contained in the interior of a two-dimensional cone of $\Sigma$, we have 
$\hat{H}_\fd = 0 \mod J$ and 
so the corresponding automorphism $
\hat{\Phi}_{\hat{H}_\fd}$ is the identity.
As we are assuming $\kappa_{\rho,\varphi}
\in J$, $R_{\rho,J}^{\hbar}$ is the $R_J^{\hbar}$-algebra generated by formal variables 
$X_+$, $X_-$ and $X$, with $X$ invertible, and with relations 
\[XX_+ =q XX_+ \,,\]
\[XX_-=q^{-1}X_- X \,,\]
\[X_+ X_-=X_- X_+=0 \,,\]
where $q=e^{i \hbar}$.
Let $\sigma_+$ and 
$\sigma_-$ be the two two-dimensional cones 
of $\Sigma$ bounding $\rho$, and let 
$\rho_+$ and $\rho_-$ be the other boundary rays of $\sigma_+$ and $\sigma_-$
respectively, such that 
$\rho_-$, $\rho$ and
$\rho_+$ are in anticlockwise order.

From 
$\varphi_\rho(m_{\rho_-})
+\varphi_\rho(m_{\rho_+})
=\kappa_{\rho, \varphi}-D_\rho^2 \varphi_\rho(m_\rho)$ and 
$\kappa_{\rho, \varphi} \in J$, we deduce that $\hat{z}^{\varphi_\rho(m_{\rho_-})}
\hat{z}^{\varphi_\rho(m_{\rho_+})}
=0$ in $R_{\rho,I}^{\hbar}$, 
$R_{\sigma_-I}^{\hbar}$ and 
$R_{\sigma_+,I}^{\hbar}$.
As $\hat{z}^{\varphi_\rho(m_{\rho_-})}$
is invertible in $R_{\sigma_-I}^{\hbar}$,
we have $\hat{z}^{\varphi_\rho(m_{\rho_+})}
=0$ in $R_{\sigma_-I}^{\hbar}$. 
Similarly, as 
$\hat{z}^{\varphi_\rho(m_{\rho_+})}$
is invertible in $R_{\sigma_+I}^{\hbar}$,
we have $\hat{z}^{\varphi_\rho(m_{\rho_-})}
=0$ in $R_{\sigma_+I}^{\hbar}$.

So the map
$\hat{\psi}_{\rho,-} \colon
R_{\rho,J}^{\hbar}
\rightarrow R_{\sigma_-,J}^{\hbar}$ is given by 
$\hat{\psi}_{\rho,-}(X)=
\hat{z}^{\varphi_\rho(m_\rho)}$,
$\hat{\psi}_{\rho,-}(X_-)=
\hat{z}^{\varphi_\rho(m_{\rho_-})}$,
$\hat{\psi}_{\rho,-}(X_+)=0$.
Similarly, 
the map
$\hat{\psi}_{\rho,+} \colon
R_{\rho,J}^{\hbar}
\rightarrow R_{\sigma_+,J}^{\hbar}$ is given by 
$\hat{\psi}_{\rho,+}(X)=
\hat{z}^{\varphi_\rho(m_\rho)}$,
$\hat{\psi}_{\rho,+}(X_-)=0$,
$\hat{\psi}_{\rho,+}(X_+)=
\hat{z}^{\varphi_\rho(m_{\rho_+})}$.
The result follows.
\end{proof}

\subsection{Quantum broken lines and theta functions}

\label{section_broken_theta}

We fix $(Y,D)$ a Looijenga pair, 
its tropicalization $(B, \Sigma)$, 
a toric monoid $P$, a radical 
monomial ideal $J$ of $P$, $\varphi$ a 
$P_\R^{\mathrm{gp}}$-valued multivalued convex $\Sigma$-piecewise linear function
$\varphi$ on $B$, and $\hat{\fD}$
a quantum scattering diagram for the data 
$(B,\Sigma)$, $P$, $J$ and $\varphi$.

Quantum broken lines and quantum theta functions have been studied by Mandel
\cite{mandel2015scattering}, for smooth integral affine manifolds.
We make below the easy combination of the 
notion of quantum broken lines and theta functions used by 
\cite{mandel2015scattering} with the notion of 
classical broken lines and theta functions  used in \cite[\S 2.3]{MR3415066}
for
the tropicalization 
$B$ of a Looijenga pair.

\begin{defn}
\label{defn_broken_line_ch3}
A \emph{quantum broken line} 
of charge $p \in B_0(\Z)$
with endpoint $Q$ in 
$B_0$
is a proper continuous 
piecewise integral affine map
\[ \gamma \colon
(-\infty, 0]
\rightarrow B_0 \]
with only finitely 
many domains of linearity,
together with, for each 
$L \subset (-\infty, 0]$
a maximal
connected domain of linearity
of $\gamma$,
a choice of monomial
$m_L=c_L \hat{z}^{p_L}$
where 
$c_L \in \kk_{\hbar}^{*}$
and
$p_L \in 
\Gamma (L, \gamma^{-1}(\cP)|_L)$, such that
the following statements hold. 
\begin{itemize}
\item For each $L$
and $t \in L$, we have 
$-\mathfrak{r}(p_L)=\gamma'(t)$, that is,
the direction of the line is determined by the monomial attached to it.
\item We have $\gamma(0)=Q
\in B_0$.
\item For the unique unbounded domain of linearity $L$, 
$\gamma|_L$ goes off for 
$t \rightarrow -\infty$ 
to infinity in a two-dimensional cone $\sigma$
of $\Sigma$
containing
$p$ and 
$m_L=\hat{z}^{\varphi_\sigma(p)}$, that is,
the charge $p$ is the asymptotic direction of the broken line.
\item Let $t \in (-\infty,0)$
be a point at which $\gamma$ is not linear, passing from the domain of
linearity $L$ to 
the domain of linearity $L'$.
Let $\tau$ be a cone of $\Sigma$ 
containing $\gamma(t)$.
Let 
$(\fd_1,\hat{H}_{\fd_1}),\dots,
(\fd_N,\hat{H}_{\fd_N})$ be the rays of 
$\hat{\fD}$ that contain $\gamma(t)$.
Then $\gamma$ passes from one side of these rays to the other side at time $t$.

Expand the product of 
\[  \prod_{\substack{
1 \leqslant k \leqslant N \\
\langle m(H_{\fd_k}),
\mathfrak{r}(p_L)
\rangle>0}}
\prod_{j=0}^{\langle m(H_{\fd_k}), \mathfrak{r}(p_L) \rangle -1}
\hat{f}_{\fd_k}(q^j \hat{z})\]
and
\[ \prod_{\substack{
1 \leqslant k'
\leqslant N \\
\langle m(H_{\fd_{k'}}),
\mathfrak{r}(p_L)
\rangle<0}} 
\prod_{j'=0}^{|\langle m(H_{\fd_{k'}}), \mathfrak{r}(p_L) \rangle| -1}
\hat{f}_{\fd_{k'}}(q^{-j'-1} \hat{z}) \,,\]
as a formal power series in 
$\kk_{\hbar}\widehat{[P_{\varphi_{\tau}}]}$.
Then there is a term $c \hat{z}^s$ in this sum with 
\[m_{L'}=m_L \cdot (c\hat{z}^s) \,.\]
\end{itemize}
\end{defn}

Let $Q \in B - \Supp_I(\hat{\fD})$
be in the interior of a two-dimensional cone 
$\sigma$ of $\Sigma$. Let $\gamma$ be a quantum broken line 
with endpoint $Q$. We denote by
$\Mono(\gamma) \in \kk_{\hbar}[P_{\varphi_\sigma}]$ the monomial attached to the last domain of linearity of $\gamma$.

The following finiteness result is formally identical to \cite[Lemma 2.25]{MR3415066}.

\begin{lem} \label{lem_finite}
Let $Q \in B - \Supp_I(\hat{\fD})$
be in the interior of a two-dimensional cone 
$\sigma$ of $\Sigma$. Fix $p \in B_0(\Z)$.
Let $I$ be an ideal of radical $J$.
Assume that $\kappa_{\rho,\varphi}
\in J$ for at least one ray $\rho$ of 
$\Sigma$. Then the following statements hold.
\begin{itemize}
\item The collection of quantum broken lines
$\gamma$ of charge $p$ with endpoint $Q$ and such that 
$\Mono(\gamma) \notin I \kk_{\hbar}
[P_{\varphi_\sigma}]$ is finite. 
\item If one boundary ray of the connected component of $B-\Supp_I(\hat{\fD})$ containing $Q$ is a ray $\rho$ of
$\Sigma$, then for every quantum broken line 
$\gamma$ of charge $p$ with endpoint $Q$, we have $\Mono(\gamma) \in \kk_{\hbar}[P_{\varphi_\rho}]$.
\end{itemize}
\end{lem} 

\begin{proof}
Identical to the proof of
\cite[Lemma 2.25]{MR3415066}.
\end{proof}

Let $Q \in B - \Supp_I(\hat{\fD})$
be in the interior of a two-dimensional cone 
$\sigma$ of $\Sigma$. Fix $p \in B_0(\Z)$.
Let $I$ be an ideal of radical $J$.
We define
\[ \Lift_Q(p) \coloneqq \sum_\gamma \Mono(\gamma)
 \in \kk_{\hbar} [P_{\varphi_\sigma}]/I \,, \]
where the sum is over all the quantum
broken lines $\gamma$ of charge $p$ with endpoint $Q$.
According to Lemma 
\ref{lem_finite}, there are only finitely many such $\gamma$ with $\Mono(\gamma)
\notin I \kk_{\hbar}
[P_{\varphi_\sigma}]$ and so $\Lift_Q(p)$ is well defined.

The following definition is formally identical to  
\cite[Definition 2.26]{MR3415066}.

\begin{defn} \label{defn_consistency_ch3}
Assume that $\kappa_{\rho,\varphi}
\in J$ for at least one one-dimensional cone $\rho$ of 
$\Sigma$. We say that a quantum scattering diagram
$\hat{\fD}$ for the data $(B,\Sigma)$, $P$, $J$ and
$\varphi$
is \emph{consistent}
if for every ideal $I$ of $P$ of radical $J$ and for all $p \in B_0(\Z)$, the following holds. Let $Q \in B_0$ be chosen so that the line joining the origin and $Q$ has irrational slope, and $Q' \in B_0$ similarly.

\begin{itemize}
\item If $Q$ and $Q'$ are contained in a common two-dimensional cone 
$\sigma$ of $\Sigma$, then we have 
\[\Lift_{Q'}(p)=\hat{\theta}_{\gamma, \hat{\fD}_I}(\Lift_Q(p))\]
in $R_{\sigma,I}^{\hbar}$, for every 
$\gamma$ path in the interior of $\sigma$ connecting $Q$ and $Q'$, and intersecting transversely the rays of $\hat{\fD}$.
\item If $Q_-$ is contained in a two-dimensional cone $\sigma_-$ of 
$\Sigma$, and $Q_+$ is contained in a two-dimensional cone $\sigma_+$ of $\Sigma$, such that $\sigma_+$ and $\sigma_-$ intersect along a one-dimensional cone $\rho$ of $\Sigma$, and furthermore $Q_-$ and $Q_+$ are contained in connected components of 
$B - \Supp_I(\hat{\fD})$ whose closures contain $\rho$, then 
$\Lift_{Q_+}(p) \in R_{\sigma_+, I}^{\hbar}$
and 
$\Lift_{Q_-}(p) \in R_{\sigma_-, I}^{\hbar}$ are both images under 
$\hat{\psi}_{\rho,+}$ and
$\hat{\psi}_{\rho,-}$ respectively of a single element 
$\Lift_\rho(p) \in R_{\rho,I}^{\hbar}$.
\end{itemize}
 
\end{defn}

The following construction is formally identical to  
\cite[Construction 2.27]{MR3415066}.
Suppose that $D$ has 
$r \geqslant 3$ irreducible components,
and that 
$\hat{\fD}$ is a consistent quantum scattering diagram for the data 
$(B,\Sigma)$, $P$, $J$ and 
$\varphi$. Assume that 
$\kappa_{\rho,\varphi} \in J$ for all 
one-dimensional cones $\rho$ of $\Sigma$.
Let $I$ be an ideal of $P$ of radical 
$J$. We construct below an element 
\[ \hat{\vartheta}_p \in \Gamma (X^\circ_{I,\fD},
\cO_{X^\circ_{I,\fD}}^{\hbar})\]
for each $p \in B(\Z)=B_0(\Z) \cup \{ 0\}$.

We first define $\hat{\vartheta}_0 \coloneqq 1$. 
Let $p \in B_0(\Z)$. 
Recall that $X^\circ_{I,\fD}$ is defined by gluing together schemes $U_{\rho,I}$, indexed by $\rho$ rays of $\Sigma$, 
and that 
$\cO_{X^\circ_{I,\fD}}^{\hbar}$ is defined 
by gluing together sheaves $\cO_{U_{\rho,I}}^{\hbar}$ on $U_{\rho,I}$, such that 
$\Gamma (U_{\rho,I},\cO_{X^\circ_{I,\fD}}^{\hbar})=R_{\rho,I}^{\hbar}$. 
So, to define $\hat{\vartheta}_p$, it is enough to define elements of $R_{\rho,I}^{\hbar}$ compatible with the gluing functions.
But, by definition, the consistency of $\hat{\fD}$ gives us such elements $\Lift_\rho(p) \in 
R_{\rho,I}^{\hbar}$.

The quantum theta functions 
$\hat{\vartheta}_p \in \Gamma (X^\circ_{I,\fD},
\cO_{X^\circ_{I,\fD}}^{\hbar})$
reduce in the classical limit to the theta 
functions 
$\vartheta_p \in \Gamma (X^\circ_{I,\fD},
\cO_{X^\circ_{I,\fD}})$
defined in \cite{MR3415066}.

\subsection{Deformation quantization of the mirror family}
\label{section_deformation_quant_mirror}

Suppose $D$ has $r \geqslant 3$ irreducible components, and let 
$\varphi$ be a $P_\R^{\mathrm{gp}}$-valued convex $\Sigma$-piecewise linear function on $B$ such that
$\kappa_{\rho,\varphi} \in J$ for all 
one-dimensional cones $\rho$ of $\Sigma$.
Let $\hat{\fD}$ be a consistent quantum scattering diagram for the data $(B,\Sigma)$, $P$, $J$ and $\varphi$. 
Let $I$ be an ideal of $P$ of radical $J$.

Denote by
\[ X_{I,\fD} \coloneqq \Spec \Gamma
(X^\circ_{I,\fD},\cO_{X^\circ_{I,\fD}})\]
the affinization of $X^\circ_{I,\fD}$
and $j \colon X_{I,\fD}^\circ \rightarrow X_{I,\fD}$
the affinization morphism. 
It is proved in \cite[ Theorem 2.28]{MR3415066}
that $j$ is an open immersion, that
$j_{*} \cO_{X^\circ_{I,\fD}} 
= \cO_{X_{I,\fD}}$, and that 
$X_I$ is flat over $R_I$.
More precisely, the 
$R_I$-algebra 
\[ A_I \coloneqq \Gamma
(X^\circ_{I,\fD},\cO_{X^\circ_{I,\fD}})
= \Gamma
(X_{I,\fD},\cO_{X_{I,\fD}})\]
is free as $R_I$-module and the set of theta functions 
$\vartheta_p$, $p \in B(\Z)$ is an $R_I$-module basis of $A_I$.

\begin{thm}
\label{thm_construction_ch3}
Suppose $D$ has $r \geqslant 3$ irreducible components, and let 
$\varphi$ be a $P_\R^{\mathrm{gp}}$-valued convex $\Sigma$-piecewise linear function on $B$ such that
$\kappa_{\rho,\varphi} \in J$ for all 
one-dimensional cones $\rho$ of $\Sigma$.
Let $\hat{\fD}$ be a consistent quantum scattering diagram for the data $(B,\Sigma)$, $P$, $J$ and $\varphi$. Let $I$ be an ideal of $P$ of radical $J$. 
Then the following statements hold.
\begin{itemize}
\item The sheaf $\cO^{\hbar}
_{X_{I,\fD}} \coloneqq j_{*}
\cO^{\hbar}_{X^\circ_{I,\fD}}$ of $R_I^{\hbar}$-algebras is a deformation quantization of 
$X_{I,\fD}$ over $R_I$ in the sense of
Definition \ref{def_quant_sheaf}.
\item The  
$R_I^{\hbar}$-algebra 
\[A_I^{\hbar} \coloneqq \Gamma (X_{I,\fD}^\circ, \cO_{X_{I,\fD}^\circ}^{\hbar})
=\Gamma (X_{I,\fD}, \cO_{X_{I,\fD}}^{\hbar})\] is a
deformation quantization of $X_{I,\fD}$
over $R_I$ in the sense of Definition 
\ref{def_quant_affine}.
\item The $R_I^{\hbar}$-algebra 
$A_I^{\hbar}$ is free as $R_I^{\hbar}$-module.
\item  The set of quantum theta functions 
\[\{ \vartheta_p^{\hbar} | \, p\in B(\Z) \}\]
is an $R_I^{\hbar}$-module basis for $A_I^{\hbar}$. 
\end{itemize}
\end{thm}

\begin{proof}
We follow the structure of the proof of \cite[Theorem 2.28]{MR3415066}.

We first prove the result for $I=J$.
As $r \geqslant 3$ and 
$\kappa_{\rho,\varphi}
\in J$ for all one-dimensional cones
$\rho$ of $\Sigma$, the only broken line 
contributing to 
$\Lift_Q(p)$, 
for every $Q$
in $B_0$ and 
$p \in B_0(\Z)$, is the straight line 
of endpoint $Q$ and direction $p$, and this provides a non-zero 
contribution only if $Q$ and $p$ lie in the same two-dimensional cone of $\Sigma$.
Combined with Lemma
\ref{lem_special_fiber}, this implies
that the map
\[\bigoplus_{p \in B(\Z)} R_J^{\hbar}
\, \hat{\vartheta}_p \rightarrow
A_J^{\hbar} \coloneqq
\Gamma (X_{J,\fD}^\circ, 
\cO_{X_{J,\fD}^\circ}^{\hbar})
=R_J^{\hbar}[\Sigma]
\]
is given by 
\[ \hat{\vartheta}_p \mapsto \hat{z}^p \,\]
and so is an isomorphism.

We now treat the case of a general ideal $I$ of $P$ of radical $J$.
By construction,
$\cO_{X_{I,\fD}^\circ}^{\hbar}$
is a deformation quantization of 
$X^\circ_{I,\fD}$ over $R_I$. In particular, $\cO_{X_{I,\fD}^\circ}^{\hbar}$
is a sheaf in flat $R_I^{\hbar}$-algebras.
As used in \cite{MR3415066}, the fibers of
$X_{J,\fD} \rightarrow \Spec R_J$ satisfy Serre's condition $S_2$ by 
\cite{MR1923963}.
We have $\cO_{X_{J,\fD}}^{\hbar}
\simeq \cO_{X_{J,\fD}} \hat{\otimes}
\kk_{\hbar}$
as $\kk_{\hbar}$-module and so it follows that 
$j_* j^{*} \cO_{X_{J,\fD}}^{\hbar}
=\cO_{X_{J,\fD}}^{\hbar}$.
The existence of quantum theta functions 
$\hat{\vartheta}_p$ guarantees that the natural map
\[ \cO_{X_{I,\fD}}^{\hbar}
\coloneqq j_*
\cO_{X_{I,\fD}}^{\hbar}
\rightarrow j_* j^{*} \cO_{X_{J,\fD}}^{\hbar} = \cO_{X_{J,\fD}}^{\hbar}\]
is surjective.
So the result follows from the following Lemma, analogous to
\cite[Lemma 2.29]{MR3415066}.

\begin{lem}
Let $X_0/S_0$ be a flat family 
of surfaces such whose fibers satisfy 
Serre's condition $S_2$. 
Let $j \colon X_0^\circ \subset X_0$
be the inclusion of an open subset
such that the complement has finite fiber. 
Let $S_0 \subset S$ be an infinitesimal thickening of $S_0$, and $X/S$ a flat deformation of $X_0/S_0$, inducing a flat deformation $X^\circ/S$ of $X_0^\circ/S_0$.
Let $\cO_{X_0}^{\hbar}$
be a deformation quantization of $X_0/S_0$
such that $\cO_{X_0}^{\hbar} \simeq
\cO_{X_0} \hat{\otimes} \kk_{\hbar}$
as $\cO_{S_0} \hat{\otimes} \kk_{\hbar}$-module, and so $j_* j^{*}\cO_{X_0}^{\hbar}=\cO_{X_0}^{\hbar}$ by the relative $S_2$ condition satisfied by $X_0/S_0$.
Let $\cO_{X^\circ}^{\hbar}$ be a deformation quantization of $X^\circ/S$, restricting 
to $j^{*}\cO_{X_0}^{\hbar}$ over $X_0^\circ$.
If the natural map
\[\cO_{X}^{\hbar} \coloneqq 
j_{*} \cO_{X^\circ}^{\hbar} \rightarrow
j_{*} j^{*} \cO_{X_0}^{\hbar}=\cO_{X_0}^{\hbar}\]
is surjective, then $\cO_{X}^{\hbar}$
is a deformation quantization of 
$X/S$. 
\end{lem}

\begin{proof}
We have to prove that $\cO_X^{\hbar}$ 
is flat over $\cO_S \hat{\otimes} \kk_{\hbar}$.

Let $\cI \subset \cO_S$ be the nilpotent ideal defining $S_0 \subset S$. 
Let $X_n$, $X^\circ_n$, $S_n$ be the $n$th-order infinitesimal thickening of $X_0$,
$X_0^\circ$, $S_0$ in $S$, that is,
$\cO_{X_n}=\cO_{X}/
\cI^{n+1}$,
$\cO_{X^\circ_n}=\cO_{X^\circ}/
\cI^{n+1}$ and $\cO_{S_n}=\cO_S/\cI^{n+1}$.

We define $\cO_{X_n}^{\hbar} \coloneqq 
j_{*} \cO_{X^\circ_n}^{\hbar}$. We show by
induction on $n$ that $\cO_{X_n}^{\hbar}$
is flat over $\cO_{S_n} \hat{\otimes} \kk_{\hbar}$.

For $n=0$, we have
$j_{*} \cO_{X^\circ_0}^{\hbar}=j_{*}j^{*} \cO_{X_0}^{\hbar}=
\cO_{X_0}^{\hbar}$, which is flat over $\cO_{S_0} \hat{\otimes} \kk_{\hbar}$ by assumption.

Assume that the induction hypothesis is true for $n-1$. Since 
$\cO_{X^\circ_n}^{\hbar}$ is flat over 
$\cO_{S_n} \hat{\otimes} \kk_{\hbar}$, 
we have an exact sequence
\[ 0 \rightarrow \cI^n/\cI^{n+1}
\otimes \cO_{X_0^\circ}^{\hbar}
\rightarrow 
\cO_{X_n^\circ}^{\hbar}
\rightarrow 
\cO_{X_{n-1}^\circ}^{\hbar}
\rightarrow 0 \,.\]
Applying $j_*$, we get an exact sequence
\[ 0 \rightarrow j_*(\cI^n/\cI^{n+1}
\otimes j^* \cO_{X_0}^{\hbar})
\rightarrow  
\cO_{X_n}^{\hbar}
\rightarrow 
\cO_{X_{n-1}}^{\hbar} \,.\]
We have $j_*(\cI^n/\cI^{n+1}
\otimes j^* \cO_{X_0}^{\hbar})
=\cI^n/\cI^{n+1}
\otimes \cO_{X_0}^{\hbar}$.

By assumption, the natural map 
$\cO_X^{\hbar} \rightarrow j_* j^* \cO_{X_0}^{\hbar}=\cO_{X_0}^{\hbar}$ is surjective. By the induction hypothesis, we have $\cO_{X_{n-1}}^{\hbar}/\cI
=\cO_{X_0}^{\hbar}$. As $\cI$ is nilpotent, it follows that the map 
$\cO_{X_n}^{\hbar} \rightarrow \cO_{X_{n-1}}^{\hbar}$ is surjective.
So we have an exact sequence
\[ 0 \rightarrow \cI^n/\cI^{n+1} \otimes\cO_{X_0}^{\hbar}
\rightarrow  
\cO_{X_n}^{\hbar}
\rightarrow 
\cO_{X_{n-1}}^{\hbar} 
\rightarrow 0\,,\]
implying that $\cO_{X_n}^{\hbar}$
is flat over $\cO_{S_n} \hat{\otimes} \kk_{\hbar}$.
\end{proof}

\end{proof}

\subsection{The algebra structure}
\label{section_algebra_structure}
This section is a 
$q$-deformed version of \cite[\S 2.4]{MR3415066}.

We saw in the previous section that the 
$R_I^{\hbar}$-algebra 
\[ A_I^{\hbar} \coloneqq \Gamma(X_{I,\fD}^\circ, \cO_{X_{I,\fD}^\circ}^{\hbar})\] 
is free 
as $R_I^{\hbar}$-module, admitting a basis of quantum theta functions 
$\hat{\vartheta}_p$, $p \in B(\Z)$.
Theorem \ref{thm_product} below gives a combinatorial expression for the structure constants of the algebra $A_I^{\hbar}$ in the basis of quantum 
theta functions.

If $\gamma$ is a quantum broken line of endpoint 
$Q$ in a cone $\tau$ of $ \Sigma$,
we can write the monomial 
$\Mono(\gamma)$
attached to the segment ending at $Q$
as 
\[ \Mono(\gamma) = c(\gamma) \hat{z}^{\varphi_{\tau} (s(\gamma))} \]
with $c(\gamma)
\in \kk_{\hbar}[P_{\varphi_\tau}]$
and $s(\gamma) 
\in \Lambda_{\tau}$.

\begin{thm} 
\label{thm_product}
Let $p \in B(\Z)$ and let $z \in B 
-\Supp_I(\hat{\fD}^\can)$ be  very close to $p$.
For every $p_1$,
$p_2 \in B(\Z)$, the structure constants $C_{p_1, p_2}^p
\in R_I^{\hbar}$ in the 
product expansion
\[ \hat{\vartheta}_{p_1}
\hat{\vartheta}_{p_2}
= \sum_{p \in B(\Z)} C_{p_1, p_2}^p \hat{\vartheta}_p\]
are given by
\[ C_{p_1,p_2}^p
= \sum_{\gamma_1, \gamma_2}
c(\gamma_1) c(\gamma_2) 
q^{\frac{1}{2} \langle
s(\gamma_1), s(\gamma_2) \rangle} \,,\]
where the sum is over all broken lines $\gamma_1$ and $\gamma_2$ of asymptotic charges $p_1$ and $p_2$,
satisfying $s(\gamma_1)+s(\gamma_2)=p$, and both ending at the point $z \in B_0$. 
\end{thm}

\begin{proof}
Let $\tau$ be the smallest cone of $\Sigma$ containing 
$p$. Working in the algebra
$\kk_{\hbar}[P_{\varphi_\tau}]/I$, we have 
\[\Lift_z(p_1) \Lift_z(p_2) =
\sum_{p \in B(\Z)} C^p_{p_1,p_2} \Lift_z(p) \,.\]
By definition, we have 
\[ \Lift_z(p_1)
= \sum_{\gamma_1} c(\gamma_1) 
\hat{z}^{\varphi_\tau(s(\gamma_1))}\,,\]
and 
\[ \Lift_z(p_2)
= \sum_{\gamma_2} c(\gamma_2) 
\hat{z}^{\varphi_\tau(s(\gamma_2))}\,.\]
As $p$ and $z$ belong to the cone $\tau$, the only quantum broken line of charge $p$ ending at $z$ is the straight line $z+\R_{\geqslant 0}$
equipped with the monomial 
$\hat{z}^{\varphi_\tau(p)}$, and so 
we have 
\[\Lift_z(p) = \hat{z}^{\varphi_\tau(p)} \,.\]
The result then follows from the multiplication rule
\[\hat{z}^{\varphi_\tau(s(\gamma_1))}\hat{z}^{\varphi_\tau(s(\gamma_2))}
=q^{\frac{1}{2}
\langle s(\gamma_1), s(\gamma_2)\rangle} \hat{z}^{\varphi_\tau(p)}\,.\]
\end{proof}

In the formula given by the previous theorem, the non-commutativity of the product of the quantum theta functions comes from the twist by the power of $q$,
\[ q^{\frac{1}{2} \langle
s(\gamma_1), s(\gamma_2) \rangle} \,, \]
which is obviously not symmetric in 
$\gamma_1$ and $\gamma_2$ as 
$\langle -,- \rangle$ is skew-symmetric.

Taking the classical limit 
$\hbar \rightarrow 0$, we get an explicit 
formula for the Poisson bracket of classical theta functions, which could have been written and proved in \cite{MR3415066}.

\begin{cor} \label{cor_poisson}
Let $p \in B(\Z)$ and let 
$z \in B- \Supp_I(\fD^\can)$ be very close to $p$.
For every 
$p_1, p_2 \in B(\Z)$, the Poisson bracket of the classical theta functions 
$\vartheta_{p_1}$ and 
$\vartheta_{p_2}$ is
given by 
\[ \{ \vartheta_{p_1}, \vartheta_{p_2} \} 
= \sum_{p \in B(\Z)} P^p_{p_1,p_2}
\vartheta_p \,,\]
where 
\[P^p_{p_1,p_2} \coloneqq \sum_{\gamma_1, \gamma_2} \langle s(\gamma_1), s(\gamma_2)
\rangle c(\gamma_1) c(\gamma_2) \,,\]
where the sum is over all broken lines 
$\gamma_1$ and $\gamma_2$ of asymptotic charges $p_1$ and $p_2$, satisfying
$s(\gamma_1)+s(\gamma_2)=p$, and both ending at the point $z \in B_0$. 
\end{cor}

\section{The canonical quantum scattering diagram}
\label{section_canonical}
In this section we construct a quantum deformation of the 
canonical scattering diagram constructed in 
\cite[\S 3]{MR3415066} and we prove its consistency. 
In \cref{section_log_gw_ch3} we 
give the definition of a family of higher-genus log Gromov--Witten invariants of a Looijenga pair. In 
\cref{section_defn_ch3}
 we use these invariants to construct the quantum
canonical scattering diagram
of a Looijenga pair and we state its consistency in 
Theorem \ref{thm_consistency}.
The proof of 
Theorem \ref{thm_consistency} occupies  
\cref{subsection_GS_locus}--
\cref{section_end_proof_ch3}, 
and follows the general structure of the proof given in the classical case by 
\cite{MR3415066}, the use of 
\cite{MR2667135}
being replaced 
by the use of \cite{bousseau2018quantum_tropical}.

\subsection{Log Gromov--Witten invariants}
\label{section_log_gw_ch3}

We fix a Looijenga pair $(Y,D)$, its tropicalization $(B, \Sigma)$, 
a toric monoid $P$ and a morphism $\eta \colon NE(Y)
\rightarrow P$ of monoids.
Let $\varphi$ be the unique (up to addition of a linear function) 
$P_\R^{\mathrm{gp}}$-valued multivalued convex $\Sigma$-piecewise 
linear function on $B$
such that $\kappa_{\rho, \varphi}
=\eta([D_\rho])$ for every $\rho$
one-dimensional cone of $\Sigma$, where 
$[D_\rho] \in NE(Y)$ is the class of the divisor $D_\rho$ dual to $\rho$.

Let $\fd \subset B$ be a ray with endpoint the origin and with rational slope. 
Let $\tau_{\fd} \in \Sigma$ be the smallest cone containing 
$\fd$ and let $m_{\fd} \in \Lambda_{\tau_\fd}$ be the primitive generator of $\fd$
pointing away from the origin.

Let us first assume that $\tau=\sigma$ is a two-dimensional cone of $\Sigma$.
The ray $\fd$ is then contained in the interior of $\sigma$. Let 
$\rho_R$ and $\rho_L$ be the two rays of 
$\Sigma$ bounding $\sigma$.
Let $m_{\rho_R}, m_{\rho_L} \in \Lambda_\sigma$ be primitive generators of
$\rho_R$, $\rho_L$ pointing away from the origin. 
As $\sigma$ is isomorphic as integral affine manifold to the standard positive quadrant
$(\R_{\geqslant 0})^2$ of $\R^2$, there exists a unique decomposition
\[ m_\fd = n_R m_{\rho_R}+n_L m_{\rho_{L}}\]
with $n_R$ and $n_L$ positive integers.
Let $NE(Y)_\fd$ be the set of classes
$\beta \in NE(Y)$ such that there exists a positive integer $\ell_\beta$
such that 
\begin{align*}
\beta \cdot D_{\rho_R} &=\ell_\beta n_R \,,\\
\beta \cdot D_{\rho_L}&=\ell_\beta n_L \,,\\
\beta \cdot D_\rho&=0 \,,
\end{align*}
for every one-dimensional cone $\rho$ of $\Sigma$ distinct of 
$\rho_R$ and $\rho_L$. 

If $\tau = \rho$ is a 
one-dimensional cone of $\Sigma$,
we define $NE(Y)_\fd$ as being the set of classes 
$\beta \in NE(Y)$ such that there exists a positive integer 
$\ell_\beta$ such that 
\[\beta \cdot D_{\rho}=\ell_\beta \,,\]
and 
\[\beta \cdot D_{\rho'}=0 \,,\]
for every one-dimensional cone $\rho'$ of $\Sigma$ distinct 
from $\rho$.

The upshot of the preceding discussion is that, for any ray $\fd$ with endpoint the origin and of rational slope, we have defined a subset $NE(Y)_\fd$ of $NE(Y)$.

We equip $Y$ with the divisorial 
log structure defined by the normal crossing divisor $D$. The resulting log scheme is log smooth. As reviewed in 
\cref{section_tropicalization}, integral points $p \in B(\Z)$ of the tropicalization
naturally define tangency conditions for stable log maps to $Y$.

For every $\beta \in NE(Y)_\fd$, let 
$\overline{M}_{g}(Y/D,\beta)$ be the moduli space of genus -$g$ stable log maps to $(Y,D)$,
of class $\beta$, and satisfying 
the tangency condition $\ell_\beta m_\fd
\in B(\Z)$.
By the work of Gross and Siebert 
\cite{MR3011419}
and of Abramovich and Chen
\cite{MR3224717, MR3257836},
$\overline{M}_{g}(Y/D,\beta)$  
is a proper Deligne--Mumford stack
of virtual dimension $g$
and admits a virtual fundamental class
\[ [\overline{M}_{g} (Y/D,\beta)]^{\virt}
\in A_g(\overline{M}_{g} (Y/D,
\beta), \Q)\,.\]
If $\pi \colon \cC \rightarrow 
\overline{M}_{g} (Y/D, \beta)$ 
is the universal curve,
of relative dualizing sheaf 
$\omega_\pi$, then the Hodge bundle
\[ \E \coloneqq \pi_{*}\omega_\pi \]
is a rank-$g$ vector bundle over 
$\overline{M}_{g} (Y/D, \beta)$. 
Its Chern classes 
are classically 
\cite{MR717614} called the lambda classes,
\[\lambda_j \coloneqq c_j(\E) \,,\]
for $j=0, \dots, g$.
We define genus-$g$ 
log Gromov--Witten invariants of $(Y,D)$ by  
\[N_{g,\beta}^{Y/D} \coloneqq
\int_{[\overline{M}_{g} (Y/D, \beta)]^{\virt}} 
(-1)^g \lambda_g  \in \Q \,.\]

\subsection{Definition}
\label{section_defn_ch3}

Using the higher-genus log Gromov--Witten invariants defined in the previous section, we can define a natural deformation of the 
canonical scattering diagram defined in  \cite[\S 3.1]{MR3415066}.

\begin{defn} \label{def_canonical_scattering}
We define $\hat{\fD}^{\can}$ as being the set of pairs $(\fd, \hat{H}_\fd)$, where $\fd$ is a ray of rational slope in $B$ with endpoint the origin, and, denoting $\tau_\fd$ the smallest cone of $\Sigma$ containing $\fd$, and 
$m_\fd \in \Lambda_{\tau_\fd}$ the 
primitive generator of $\fd$ 
pointing away from the origin, $\hat{H}_\fd$ is given by 
\[ \hat{H}_{\fd} \coloneqq 
\left(\frac{i}{\hbar} \right) \sum_{\beta \in NE(Y)_{\fd}} 
\left( \sum_{g \geqslant 0} N_{g,\beta}^{Y/D} \hbar^{2g} \right) 
\hat{z}^{\eta(\beta)- \varphi_{\tau_\fd}(\ell_\beta m_\fd)} \in \kk_{\hbar}\widehat{[P_{\varphi_{\tau_\fd}}]}\,.\]
\end{defn}

The following lemma is formally almost identical to 
\cite[Lemma 3.5]{MR3415066}.

\begin{lem}\label{lem_assumptions_ch3}
Let $J$ be a radical ideal of $P$. Suppose that the map 
$\eta \colon NE(Y) \rightarrow P$
satisfies the following conditions
\begin{itemize}
\item If $\fd$ is contained in the interior
of a two-dimensional cone of $\Sigma$, then 
$\eta(\beta) \in J$ for every $\beta \in NE(Y)_\fd$ such that $N_{g,\beta} \neq 0$ for some $g$.
\item If $\fd$ is a ray $\rho$ of $\Sigma$ 
and 
$\kappa_{\rho,\varphi} \notin J$,
then $\eta(\beta) \in J$ for every $\beta \in NE(Y)_\fd$ such that $N_{g,\beta} \neq 0$ for some $g$.
\item For any ideal $I$ in $P$ of radical $J$, there are only finitely may classes 
$\beta \in NE(Y)$ such that 
$N_{g,\beta} \neq 0$ for some $g$ and such that
$\eta (\beta) \notin I$.
\end{itemize}
Then $\hat{\fD}^{\can}$ is a quantum scattering diagram for the data $(B,\Sigma)$, 
$P$, $J$ and $\varphi$. Furthermore, the quantum scattering diagram 
$\hat{\fD}^{\can}$ has only outgoing
rays.
\end{lem}

\begin{proof}
The assumptions guarantee the finiteness requirements in the definition of a quantum scattering diagram: see  
\cref{section_quantum_scattering_diagrams}.
The ray $(\fd, \hat{H}_{\fd})$
is outgoing because 
$\mathfrak{r}(\eta(\beta)-\varphi_{\tau_\fd}(\ell_\beta 
m_{\fd}))
= -\ell_\beta m_\fd \in \Z_{<0} m_\fd$.
\end{proof}

\begin{lem}
The classical limit of the canonical quantum  
scattering diagram $\hat{\fD}^\can$ is the canonical scattering diagram defined in \cite[\S 3.1]{MR3415066}.
\end{lem}

\begin{proof}
It follows from an analogue of the cycle arguments detailed in
\cite[Proposition 11]{MR3904449} and 
\cite[Lemma 15]{bousseau2018quantum_tropical},
and from the log birational invariance of 
logarithmic Gromov--Witten invariants
\cite{MR3778185},
that the relative genus-$0$ Gromov--Witten invariants of non-compact surfaces used in \cite{MR3415066}
coincide with the genus-$0$ log Gromov--Witten invariants $N_{0,\beta}^{Y/D}$.
More precisely, given a stable log map
$f \colon C \rightarrow Y$ defining a point in $\overline{M}_0(Y/D,\beta)$, we claim that no component of $C$ is mapped inside $D$. Indeed, if it were not the case, one could argue at the tropical level:
knowing the asymptotic behavior of the tropical map to the tropicalization $B$ of $Y$ imposed by the tangency condition $\ell_\beta m_\fd$, and repeatedly using the tropical balancing condition \cite[Proposition 1.15]{MR3011419}, we  would get that $C$ needs to contain a cycle of components mapping surjectively to $D$, contradicting the genus-$0$ assumption.

By Lemma \ref{lem_autom}, the quantum automorphism $\hat{\Phi}_{\hat{H}_\fd}$
coincides in the classical limit 
$\hbar \rightarrow 0$, $q=e^{i \hbar} \rightarrow 1$
with the automorphism
\[ z^p \mapsto z^p f_\fd^{\langle m(\hat{H}_\fd),\mathfrak{r}(p)\rangle} \]
of \cite{MR3415066}, where, using 
$\mathfrak{r}(\eta(\beta)- \varphi_{\tau_\fd}(\ell_\beta m_\fd))=-\ell_{\beta} m_\fd$, we have
\[ f_\fd = \lim_{\hbar \rightarrow 0}
\hat{f}_\fd=\lim_{\hbar \rightarrow 0}
\exp \left( \left(\frac{i}{\hbar} \right) \sum_{\beta \in NE(Y)_{\fd}} (e^{-i\ell_\beta \hbar}-1) 
\left( \sum_{g \geqslant 0} N_{g,\beta}^{Y/D} \hbar^{2g} \right) 
\hat{z}^{\eta(\beta)- \varphi_{\tau_\fd}(\ell_\beta m_\fd)} \right)\]
that is,
\[f_\fd=\exp \left(\sum_{\beta \in NE(Y)_{\fd}} \ell_\beta N_{0,\beta}^{Y/D} 
\hat{z}^{\eta(\beta)- \varphi_{\tau_\fd}(\ell_\beta m_\fd)} \right)\,, \]
which coincides with 
\cite[Definition 3.3]{MR3415066}.
\end{proof}

\subsection{Consistency}
The following result states that the quantum scattering diagram 
$\hat{\fD}^{\can}$, defined in  
\cref{section_defn_ch3}, is consistent in the sense of
\cref{section_broken_theta}.

\begin{thm}
\label{thm_consistency}
Suppose that 
\begin{itemize}
\item for any class $\beta \in NE(Y)$ such that $N_{g,\beta} \neq 0$ for some $g$, we have $\eta(\beta) \in J$;
\item for any ideal $I$ of $P$
of radical $J$, there are only finitely many classes $\beta \in NE(Y)$ such that $N_{g,\beta} \neq 0$ for some $g$ and 
$\eta(\beta) \notin I$;
\item $\eta([D_\rho]) \in J$
for at least one boundary component 
$D_\rho \subset D$. 
\end{itemize}
Then the canonical quantum scattering 
diagram 
$\hat{\fD}^\can$ is 
consistent.
\end{thm}

Let us review the various steps taken by \cite{MR3415066}  to prove the consistency 
of the canonical scattering diagram in the classical case.
\begin{itemize}
\item Step I. We can replace $(Y,D)$ by a 
corner blow-up of $(Y,D)$.
\item Step II. Changing the monoid $P$.
\item Step III. Reduction to the Gross-Siebert locus.
\item Step IV. Pushing the singularities at infinity.
\item Step V. $\bar{\fD}$ satisfies the required compatibility condition.
\end{itemize}

Step I (see \cite[Proposition 3.10]{MR3415066}) is easy in the classical case. The quantum case is similar: the scattering diagram changes only in a trivial way under corner blow-up and we will not say more.

Step II (see \cite[Proposition 3.12]{MR3415066}) is more subtle and involves some regrouping of monomials in the comparison of the broken lines for two different monoids. 
Exactly the same regrouping operation deals with the quantum case too.

Step III in \cite{MR3415066}  requires an understanding of 
genus-0 multicover contributions of exceptional divisors of a toric model. We explain in 
\cref{subsection_GS_locus} how the quantum analogue is obtained from the knowledge of higher-genus multicover contributions. 

Step IV in \cite{MR3415066} is the reduction of the consistency of 
$\fD^{\can}$ to the consistency of a scattering diagram 
$\nu(\fD^{\can})$ on an 
integral affine manifold without singularities. We explain in 
\cref{subsection_pushing},
\cref{section_comparing_ch3} and
\cref{section_end_proof_ch3} how the consistency of the quantum scattering diagram  $\hat{\fD}^{\can}$ can be reduced to the consistency of a
quantum scattering diagram 
$\nu(\hat{\fD}^{\can})$ on an integral affine 
manifold without singularities.

Step V in \cite{MR3415066} 
is the proof of consistency of 
$\nu(\fD^\can)$ and
ultimately relies on the main result of \cite{MR2667135}.
We explain in 
\cref{section_consistency_ch3} how its $q$-analogue,
the consistency of 
$\nu(\hat{\fD}^\can)$, 
ultimately relies on the main result of \cite{bousseau2018quantum_tropical}.

\subsection{Reduction to the Gross-Siebert locus.}
\label{subsection_GS_locus}
We start by recalling some notation of 
\cite{bousseau2018quantum_tropical}.

Let $\fm =(m_1, \dots, m_n)$ be an $n$-tuple of primitive non-zero vectors of 
$M=\Z^2$. 
We add extra rays to the fan given by the rays  
$-\R_{\geqslant 0} m_1, \dots, -
\R_{\geqslant 0} m_n$ such that the resulting fan defines a smooth projective toric surface $\bar{Y}_\fm$.
 The choice of the added
rays will be irrelevant for us
(ultimately because of the log birational invariance result in logarithmic Gromov--Witten theory proved in 
\cite{MR3778185}) and so is not included in the notation. 
Denote by $\partial \bar{Y}_\fm$ the anticanonical toric divisor of $\bar{Y}_\fm$, and let $D_{m_1}, \dots, D_{m_n}$
be the irreducible components of  $\partial \bar{Y}_\fm$ dual to the rays 
$-\R_{\geqslant 0} m_1, \dots, -
\R_{\geqslant 0} m_n$.

For every $j=1, \dots, n$, we blow up a point
$x_j$ in general position on the toric divisor $D_{m_j}$. 
Remark that it is possible to have 
$\R_{\geqslant 0} m_j = \R_{\geqslant 0} m_{j'}$, and so $D_{m_j}=D_{m_{j'}}$, for $j \neq j'$, and that in this case we blow
up several distinct points on the same toric divisor. We denote $Y_{\fm}$ the resulting
projective surface and 
$\pi \colon 
Y_{\fm} \rightarrow 
\bar{Y}_\fm$ the blow-up morphism.
Let $E_j \coloneqq \pi^{-1}(x_j)$ be the exceptional divisor over $x_j$.
We denote 
$\partial Y_{\fm}$ the strict transform of $\partial \bar{Y}_\fm$.

Using Steps I and II and the deformation invariance of log Gromov--Witten invariants in log smooth families,
we can make the following assumptions
(see \cite[Assumptions 3.13]{MR3415066}).
\begin{itemize}
\item There exists an $n$-tuple
$\fm=(m_1,\dots,m_n)$ of primitive non-zero vectors 
of $M=\Z^2$ such that $(Y,D)=(Y_\fm,\partial Y_\fm)$.
\item The map $\eta \colon NE(Y) \rightarrow P$ is an
inclusion and $P^{\times}=\{0\}$.
\item There is an ample divisor $H$ on 
$\overline{Y}$ such that there is a face of $P$ whose intersection with 
$NE(Y)$ is the face $NE(Y) \cap (p^{*}H)^\perp$ generated by the classes $[E_{ j}]$
of exceptional divisors. 
Let $G$ be the prime monomial ideal of $R$
generated by the complement of this face.
\item $J=P-\{0\}$.
\end{itemize}

Following \cite[Definition 3.14]{MR3415066}, the Gross-Siebert locus is the open torus orbit 
$T^{\mathrm{gs}}$ of the toric face 
$\Spec \kk[P]/G$ of $\Spec \kk[P]$.

\begin{prop} \label{lem_mod_G}
For each ray $\rho$ of $\Sigma$,with primitive generator
$m_\rho \in \Lambda_{\rho}$
pointing away from the origin, the
Hamiltonian $\hat{H}_{\rho}$ attached to $\rho$
in the scattering diagram $\hat{\fD}^\can$ satisfies 
\[ \hat{H}_{\rho} = i \sum_{j, D_{m_j}=D_\rho}
\sum_{\ell \geqslant 1}\frac{1}{\ell} \frac{(-1)^{\ell-1}}{2 \sin \left( \frac{\ell \hbar}{2} \right)} \hat{z}^{\ell [E_{j}]-\ell \varphi_{\rho}(m_\rho)}   \mod G \,.\]
\end{prop}

\begin{proof}
The only contributions to $\hat{H}_{\rho} \mod G$ come from the multiple covers of the exceptional divisors $E_{j}$.
The result then follows from \cite[Lemma 
23]{bousseau2018quantum_tropical}, which relies on the study of Gromov--Witten theory of local curves done 
by Bryan and Pandharipande in
\cite{MR2115262}.
\end{proof}

\begin{prop}
The canonical quantum scattering diagram 
$\hat{\fD}^\can$ is a scattering diagram for the data 
$(B,\Sigma)$, $P$, $G$ and $\varphi$.
Concretely, for every ideal $I$ of $P$ of radical $G$, there are only finitely many 
rays such $(\fd, \hat{H}_{\fd})$
such that $\hat{H}_{\fd} \neq 0 \mod I$.
\end{prop}

\begin{proof}
This follows from the argument given in the proof of \cite[Corollary 3.16]{MR3415066}. It is a geometric argument about curve classes and the genus of the curves plays no role.
\end{proof}

\begin{prop}
If $\hat{\fD}^\can$ is consistent as  a quantum scattering diagram for the data 
$(B,\Sigma)$, $P$, $G$ and 
$\varphi$, then
$\hat{\fD}^\can$ is consistent as  a quantum scattering diagram for the data 
$(B,\Sigma)$, $P$, $J$ and 
$\varphi$.
\end{prop}

\begin{proof}
Identical to the proof of 
\cite[Theorem 3.17]{MR3415066}.
\end{proof}

Following \cite[Remark 3.18]{MR3415066},
we denote by $E \subset P^{\gp}$ the sublattice generated by the face $P\,\backslash\,G$. We naturally have $T^{\mathrm{gs}}
= \Spec \kk[E] \subset \Spec \kk[P]$.
Denote $\fm_{P+E}=(P+E)\,\backslash\,E$.

The following Lemma is formally identical to Lemma 3.19 of \cite{MR3415066}.

\begin{lem}
If $\hat{\fD}^{\can}$, viewed as a quantum scattering diagram for the data 
$(B,\Sigma)$, $P+E$, $\varphi$ and 
$\fm_{P+E}$, is consistent,
then $\hat{\fD}^{\can}$, viewed as a quantum scattering 
diagram for the data 
$(B,\Sigma)$, $P$, $\varphi$ and $G$,
is consistent.
\end{lem}

\begin{proof}
Identical to the proof of 
\cite[Lemma 3.19]{MR3415066}.
\end{proof}

It follows that we can replace $P$ by $P+E$, and so from now on we assume that $P^{*}=E$ and 
$G=P\,\backslash\, E$. Concretely, this means that it is enough to check the consistency of 
$\hat{\fD}^{\can}$ by working in rings in which the monomials 
$\hat{z}^{[E_j]-\varphi_\rho(m_\rho)}$
are invertible.

\subsection{Pushing the singularities at infinity}
\label{subsection_pushing}

We first recall the notation introduced 
at the beginning of Step IV of \cite{MR3415066}.

We denote by
$M=\Z^2$ the lattice of cocharacters of the torus acting on the toric surface 
$(\bar{Y}_\fm,\partial \bar{Y}_\fm)$.
Let $(\bar{B},\bar{\Sigma})$
be the tropicalization of $(\bar{Y}_\fm,\partial \bar{Y}_\fm)$.
The affine manifold $\bar{B}$
has no singularity at the origin and so is naturally 
isomorphic to $M_\R=\R^2$.
The cone decomposition $\bar{\Sigma}$
of $M_\R=\R^2$ is simply the fan of 
$\bar{Y}$.
Let $\bar{\varphi}$ be the single-valued $P_\R^{\gp}$-valued on 
$\bar{B}$ such that 
\[ \kappa_{\bar{\rho}, \bar{\varphi}}
=\pi^{*}[\bar{D}_{\bar {\rho}}] \,,\]
for every $\bar{\rho}$ one-dimensional 
cone of $\bar{\Sigma}$ and 
where $\bar{D}_{\bar{\rho}}$ is the toric divisor dual to $\bar{\rho}$.
Since $\bar{\varphi}$
is single-valued and $\bar{B}$
has no singularities, the sheaf 
$\bar{\cP}$,
as defined in 
\cref{section_tropicalization} 
is constant with fiber 
$P^{\gp} \oplus M$.

There is a canonical piecewise linear map 
$\nu \colon B \rightarrow \bar{B}$
which restricts to an integral affine 
isomorphism 
$\nu|_\sigma \colon \sigma \rightarrow 
\bar {\sigma}$ from each two-dimensional cone 
$\sigma$ of $\Sigma$ to the corresponding
two-dimensional cone $\bar{\sigma}$ of 
$\bar{\Sigma}$. This map naturally 
identifies $B(\Z)$ with 
$\bar{B}(\Z)$. Restricted to each two-dimensional cone $\sigma$ of 
$\Sigma$, the derivative $\nu_*$ of $\nu$ induces a identification
$\Lambda_{B,\sigma} \simeq 
\Lambda_{\bar{B},\bar{\sigma}}$,
an isomorphism of monoids 
\[ \tilde{\nu}_\sigma \colon 
P_{\varphi_\sigma} \rightarrow P_{\bar{\varphi}_{\bar{\sigma}}} \]
\[p+\varphi_\sigma(m) \mapsto
p+\bar{\varphi}_{\bar{\sigma}}(\nu_*(m)) \,,\]
for $p\in P$ and $m \in \Lambda_\sigma$, 
and
so an identification of algebras 
of $\kk_{\hbar}[P_{\varphi_\sigma}]$
and $\kk_{\hbar}[P_{\bar{\varphi}_{\bar{\sigma}}}]$.

If $\rho$ is a one-dimensional cone
of $\Sigma$, then
$\nu_*$ is only defined on the tangent space to $\rho$ (not on the full
$\Lambda_\rho$ because $\nu$ is only piecewise linear) and so gives an identification
\[ \tilde{\nu}_\rho 
\colon \{p+\varphi_\rho(m)|\text{
$m$ tangent to $\rho$, $p \in P$}\}
\rightarrow 
\{p+\bar{\varphi}_{\bar{\rho}}(m)|
\text{
$m$ tangent to $\bar{\rho}$, $p \in P$}\} \]
\[ p+\varphi_\rho(m)
\mapsto 
p+\bar{\varphi}_{\bar{\rho}}(\nu_*(m)) \,.\]

We define below a quantum scattering diagram
$\nu(\hat{\fD}^\can)$  
for the data
$(\bar{B},\bar{\Sigma})$,
$P$, $\bar{\varphi}$
and $G$.
\begin{itemize}
\item For every ray 
$(\fd, \hat{H}_\fd)$ of $\hat{\fD}^\can$
contained in the interior of a
two-dimensional cone of $\Sigma$, the quantum scattering diagram $\nu(\hat{\fD}^\can)$
contains the ray 
\[ (\nu(\fd), \tilde{\nu}_{\tau_\sigma}(\hat{H}_\fd))\,,\]
which is outgoing.
\item For every ray 
$(\rho, \hat{H}_{\rho})$,
with $\rho$ a one-dimensional cone of $\Sigma$, and 
so by Proposition \ref{lem_mod_G}, 
\[ \hat{H}_{\rho} =\hat{G}_{\rho}
+ i \sum_{j, D_{m_j}=D_\rho}
\sum_{\ell \geqslant 1}\frac{1}{\ell} \frac{(-1)^{\ell-1}}{2 \sin \left( \frac{\ell \hbar}{2} \right)} \hat{z}^{\ell [E_j]-\ell \varphi_{\rho}(m_\rho)} \,,\]
with $\hat{G}_{\rho}=0 \mod G$, 
the quantum scattering diagram
$\nu(\hat{\fD}^\can)$ contains two rays: 
\[ (\bar{\rho}, \tilde{\nu}_{\tau_\fd}(\hat{G}_{\rho})) \,,\]
which is outgoing,
and 
\[ \left(\bar{\rho},
\,
i
\sum_{j,D_{m_j}
=D_\rho}
\sum_{\ell \geqslant 1}\frac{1}{\ell} \frac{(-1)^{\ell-1}}{2 \sin \left( \frac{\ell \hbar}{2} \right)} \hat{z}^{\ell \bar{\varphi}
(m_\rho)-\ell[E_j]} \right) \,,\]
which is ingoing.
\end{itemize}

In going from 
$\hat{\fD}^{\can}$ to
$\nu(\hat{\fD}^\can)$, we invert
$\hat{z}^{\ell[E_j]
-\ell \bar{\varphi}_{\bar{\rho}}
(m_\rho)}$, which becomes
$\hat{z}^{\ell \bar{\varphi}_{\bar{\rho}}
(m_\rho)-\ell[E_j]}$. 
This makes sense because we are assuming 
$P^*=E$.

\subsection{Consistency of $\nu(\hat{\fD}^\can)$}
\label{section_consistency_ch3}

Let $\hat{\fD}_\fm$ be the quantum scattering diagram for
the data $(\bar{B},\bar{\Sigma})$,
$P$, $\bar{\varphi}$ and $G$,
having, for each $\bar{\rho}$
one-dimensional cone of 
$\bar{\Sigma}$, a ray  
$(\bar{\rho},
\hat{H}_{\bar{\rho}})$
where 
\[\hat{H}_{\bar{\rho}}
\coloneqq i
\sum_{j, D_{m_j}
=D_\rho} 
\sum_{\ell \geqslant 1}\frac{1}{\ell} \frac{(-1)^{\ell-1}}{2 \sin \left( \frac{\ell \hbar}{2} \right)} \hat{z}^{\ell \bar{\varphi}
(m_\rho)-\ell[E_j]}  \,.\] 
Writing $\ell \bar{\varphi}(m_\rho)
-\ell[E_j]
=(\ell m_\rho, \bar{\varphi}(\ell m_\rho)-\ell[E_j])$, it is clear that
$\hat{H}_{\bar{\rho}} \in 
\kk_{\hbar}\widehat{[P_{\overline{\varphi}}]}$, where the monoid
\[ P_{\overline{\varphi}} = \{(m,
\bar{\varphi}(m)+p)|m \in M, p \in P\}\]
is independent of $\overline{\rho}$.

For such quantum scattering diagram
$\hat{\fD}$,
with all Hamiltonians valued in the same ring, it makes sense to define an automorphism $\hat{\theta}_{\gamma,\hat{\fD}}$ of this ring, as in  \cref{subsection_gluing_ch3}, but for $\gamma$ an arbitrary path in $\bar{B}_0$ transverse to the rays of the diagram.
By \cite[ Theorem 6]{MR2181810}, there exists another scattering diagram 
$S(\hat{\fD})$ containing 
$\hat{\fD}$, such that $S(\hat{\fD})
-\hat{\fD}$ consists only of outgoing rays
and $\hat{\theta}_{\gamma,S(\hat{\fD})}$
is the identity for $\gamma$ a loop in $\bar{B}_0$ going around the origin.
We can assume that there is at most one ray of $S(\hat{\fD})-\hat{\fD}$ in each possible outgoing direction.

The scattering diagram $S(\hat{\fD}_\fm)$ is the main object of study of 
\cite{bousseau2018quantum_tropical}\footnote{Comparing the conventions of the present paper and 
\cite{bousseau2018quantum_tropical}, the notions of outgoing and ingoing rays are exchanged. This implies that a global sign must be included in comparing the Hamiltonians of the present paper and  \cite{bousseau2018quantum_tropical}.}.

For every 
$m \in M-\{0\}$, let $P_m$ be the subset of 
$p=(p_1, \dots, p_n) \in \NN^n$ such that 
$\sum_{j=1}^n p_j m_j$ is positively collinear with $m$:
\[ \sum_{j=1}^n p_j m_j = \ell_p m \]
for some $\ell_p \in \NN$.
Given $p \in P_m$, we defined in 
\cite{bousseau2018quantum_tropical}
a curve class $\beta_p \in A_1(Y,\Z)$.

Recall that if $\fd \subset \bar{B}$
is a ray with endpoint the origin
and rational slope, we denote by $m_\fd \in M$ the primitive generator of $\fd$ pointing away from the origin.

The following Proposition expresses 
$S(\hat{\fD}_\fm)$ in terms of the log Gromov--Witten invariants $N_{g,\beta}^{Y_\fm/\partial Y_\fm}$ defined in 
\cref{section_log_gw_ch3} and entering in the definition of
$\hat{\fD}^\can$.

\begin{prop} \label{prop_scat}
The Hamiltonian $\hat{H}_\fd$ attached 
to an outgoing ray 
$\fd$ 
of $S(\hat{\fD}_\fm)-
\hat{\fD}_\fm$
is given by 
\[\hat{H}_{\fd}=
\left(
\frac{i}{\hbar}
\right)
\sum_{p \in P_{m_\fd}} 
\left(
\sum_{g \geqslant 0} N_{g,\beta_p}^{Y_\fm /
\partial Y_\fm}
\hbar^{2g} \right)
\hat{z}^{(-\ell_\beta m_{\fd}, \beta_p - \bar{\varphi}(\ell_\beta m_{\fd}))} \,, \]
where 
$(-\ell_\beta m_{\fd}, \beta_p - \bar{\varphi}(\ell_\beta m_{\fd})) \in P_{\bar{\varphi}}$.
\end{prop}

\begin{proof}
This is the main result of 
\cite{bousseau2018quantum_tropical}.
\end{proof}

\begin{prop} \label{prop_comp_ch3}
We have $S(\hat{\fD}_\fm)=\nu (\hat{\fD}^\can)$. 
\end{prop}

\begin{proof}
We compare the explicit description of 
$S(\hat{\fD}_\fm)$ given by 
Proposition \ref{prop_scat} with the explicit description of $S(\hat{\fD}_\fm)$ obtained from its definition in \cref{subsection_pushing}
and from the definition of $\hat{\fD}^\can$ in \cref{section_defn_ch3}.

The ingoing rays obviously coincide.

Let $\fd$ be an outgoing ray. The corresponding Hamiltonian in $\nu(\hat{\fD}^\can)$ involves the log Gromov--Witten invariants $N_{g,\beta}^{Y_\fm/\partial Y_\fm}$ for 
\[ \beta \in NE(Y)_\fd \cap G\,,\] 
whereas the corresponding Hamiltonian in 
$S(\hat{\fD}_\fm)$ involves the log Gromov--Witten invariants $N_{g,\beta_p}^{Y_\fm/\partial Y_\fm}$ for $p \in P_{m_\fd}$.
The only thing to show is that $N_{g,\beta}^{Y_\fm/\partial Y_\fm}=0$ if $\beta \in 
NE(Y)_\fd \cap G$ is not of the form
$\beta_p$ for some $p \in P_{m_\fd}$.

Recall that we have the blow-up morphism $\pi \colon Y_{\fm}
\rightarrow \bar{Y}_\fm$. 
Let $\beta \in NE(Y)_\fd \cap G$.
We can uniquely write 
$\beta=\pi^* \pi_* \beta - \sum_{j=1}^n
p_j E_j$ for some $p_j \in \Z$,
$j=1,\dots,n$. If $p_j \geqslant 0$ for every $j=1,\dots,n$, then $p=(p_1,\dots,p_n)\in \NN^n$ and $\beta=\beta_p$.

Assume that there exists $1\leqslant j \leqslant n$ such that $p_j<0$. Then 
\mbox{$\beta \cdot E_j=p_j<0$} and so every stable 
log map $f \colon C \rightarrow Y_\fm$ of class $\beta$ has a component dominating $E_j$. If $\fd
\neq -\R_{\geqslant 0}m_j$, 
then we can employ an analogue of the cycle argument of
\cite[Proposition 11]{MR3904449} and  
\cite[Lemma 15]{bousseau2018quantum_tropical}.
Knowing the asymptotic behavior of the tropical map to the tropicalization $B$ of $Y_\fm$, imposed by the tangency condition $\ell_\beta m_\fd$, and repeatedly using the balancing condition, we get that $C$ needs to contain a cycle of components mapping surjectively to $\partial Y_\fm$. Vanishing properties of the lambda class (see, for example, \cite[Lemma 8]{MR3904449}) then imply that $N_{g,\beta}^{Y_\fm
/\partial Y_\fm}=0$.
If $\fd= -\R_{\geqslant 0}m_j$ for some $j$, then the same argument implies the vanishing of $N_{g,\beta}^{Y_\fm/\partial Y_\fm}$, unless $\beta$ is a multiple of some $E_j$, which is not the case by the assumption $\beta \in G$.

\end{proof}

The following Proposition is the quantum version of \cite[Theorem 3.30]{MR3415066}.

\begin{prop} \label{prop_cps_ch3} 
Let $I$ be an ideal
of $P$ of radical $G$.
If $Q$ and $Q'$ are two points in general 
position in $M_\R - \Supp(S(\hat{{\fD}}_\fm))_I$, and 
$\gamma$ is a path connecting $Q$ and $Q'$
for which 
$\hat{\theta}_{\gamma,S(\hat{\fD}_\fm)_I}$ is defined, then 
\[\Lift_{Q'}(p)
=\hat{\theta}_{\gamma,S(\hat{\fD}_\fm)_I}(\Lift_Q(p))\]
as elements of 
$\kk_{\hbar}[P_{\bar{\varphi}}]/I$.
\end{prop}

\begin{proof}
The key input is that, by construction,
$\hat{\theta}_{\gamma,S(\hat{\fD}_\fm)}$ is the identity for a loop $\gamma$ in $\bar{B}_0$ going around the origin.
Proofs of the classical statement can be found in 
\cite{cps}, \cite[\S 5.4]{MR2722115} and \S 3.2 of the first arXiv version of 
\cite{MR3415066}. Putting hats everywhere, the same argument proves the quantum version, without extra complication.
\end{proof}

\subsection{Comparing $\hat{\fD}^\can$ and $\nu(\hat{\fD}^\can)$}
\label{section_comparing_ch3}

In order to obtain the consistency of 
$\hat{\fD}^\can$ from some properties of 
$\nu(\hat{\fD}^\can)$, we need to
compare the rings $R^{\hbar}_{\sigma,I}$,
$R^{\hbar}_{\rho,I}$ coming from
$(B,\Sigma)$, $\varphi$, with the 
corresponding rings
$\bar{R}^{\hbar}_{\sigma,I}$,
$\bar{R}^{\hbar}_{\rho,I}$ coming from
$(\bar{B},\bar{\Sigma})$, $\bar{\varphi}$.
Such comparison is done in the 
following Lemma.

\begin{lem}
\label{lem_isom_ch3}
There are isomorphisms 
$p_\rho \colon R^{\hbar}_{\rho,I}
\rightarrow \bar{R}^{\hbar}_{\bar{\rho},I}$
and \mbox{$p_\sigma \colon 
R^{\hbar}_{\sigma,I}
\rightarrow \bar{R}^{\hbar}_{\bar{\sigma},I}$},
intertwining 
\begin{itemize}
\item the maps
$\hat{\psi}_{\rho,-} \colon
R_{\rho,I}^{\hbar}
\rightarrow R_{\sigma_-,I}^{\hbar}$
and 
$\hat{\psi}_{\bar{\rho},-} \colon
\bar{R}_{\bar{\rho},I}^{\hbar}
\rightarrow \bar{R}_{\bar{\sigma}_-,I}^{\hbar}$, 
\item the maps $\hat{\psi}_{\rho,+} \colon
R_{\rho,I}^{\hbar}
\rightarrow R_{\sigma_+,I}^{\hbar}$
and 
$\hat{\psi}_{\bar{\rho},+} \colon
\bar{R}_{\bar{\rho},I}^{\hbar}
\rightarrow \bar{R}_{\bar{\sigma}_+,I}^{\hbar}$,
\item the automorphisms
$\hat{\theta}_{\gamma,\hat{\fD}^\can}:
R_{\sigma,I}^{\hbar} \rightarrow 
R_{\sigma,I}^{\hbar}$ and 
\mbox{$\hat{\theta}_{\bar{\gamma},\nu(\hat{\fD}^\can)}:
R_{\bar{\sigma},I}^{\hbar} \rightarrow 
R_{\bar{\sigma},I}^{\hbar}$}, where $\gamma$ is a path in $\sigma$ for which 
$\hat{\theta}_{\gamma,\fD^\can}$
is defined and $\bar{\gamma}
=\nu \circ \gamma$. 
\end{itemize}
\end{lem}

\begin{proof}
This is a quantum version of
\cite[Lemma 3.27]{MR3415066}. The isomorphism $p_\sigma$
simply comes from the isomorphism of 
monoids 
\mbox{$\tilde{\nu}_\sigma \colon 
P_{\varphi_\sigma} \rightarrow P_{\bar{\varphi}_{\bar{\sigma}}}$}.

Recall from  
\cref{subsection_building_blocks}
that the rings $R_{\rho,I}^{\hbar}$
and $\bar{R}_{\bar{\rho},I}^{\hbar}$
are generated by variables $X_+$,
$X_-$, $X$ and 
$\bar{X}_+$,$\bar{X}_-$,
$\bar{X}$ respectively and we define $p_\rho$ as the morphism of 
$R_{I}^{\hbar}$-algebras such that 
$p_\rho(X_+)=\bar{X}_+$, $p_\rho(X_-)
=\bar{X}_-$, $p_\rho(X)=\bar{X}$.
We have to check that $p_\rho$ is compatible with the relations defining
$R_{\rho,I}^{\hbar}$ and 
$\bar{R}_{\bar{\rho},I}^{\hbar}$.

We have $\hat{f}_{\rho^\iin}=1$.
Using Proposition
\ref{lem_mod_G} and 
Lemma
\ref{lem_ex_ch3}, we can write 
\[ \hat{f}_{\rho^\out}(X)
=\hat{g}_\rho(X)
\prod_{j,D_{m_j}=D_\rho}
(1+q^{-\frac{1}{2}} \hat{z}^{[E_j]} X^{-1})\,,\]
for some $\hat{g}_\rho(X)=1 \mod G$.
Using the definition of 
$\nu(\hat{\fD}^\can)$ given 
in 
\cref{subsection_pushing},
and
Lemma
\ref{lem_ex_ch3}, we have 
\[ \hat{f}_{\bar{\rho}^\iin}
(\bar{X})
=\prod_{j,D_{m_j}=D_\rho}
(1+q^{-\frac{1}{2}} \hat{z}^{-[E_j]}
\bar{X})\,,\]
and 
\[ \hat{f}_{\bar{\rho}^\out}
(\bar{X})
=\hat{g}_\rho(\bar{X}) \,.\]
We need to check that 
\[ p_\rho \left(
q^{\frac{1}{2} D_\rho^2}
\hat{z}^{[D_\rho]}
\hat{f}_{\rho^\iin}(X) \hat{f}_{\rho^\out}(q^{-1} X)X^{-D_{\rho}^2}\right)
=q^{\frac{1}{2} D_{\bar{\rho}}^2}
\hat{z}^{[D_{\bar{\rho}}]} 
\hat{f}_{\bar{\rho}^\iin}(\bar{X})
\hat{f}_{\bar{\rho}^\out}(q^{-1} \bar{X})
\bar{X}^{-D_{\bar{\rho}}^2} \,.\]
We have $D_\rho^2=D_{\overline{\rho}}^2-a_\rho$, where $a_\rho$ is the number of $j$
such that $D_{m_j}=D_\rho$, and
\[[D_\rho]=[D_{\overline{\rho}}] -\sum_{j,D_{m_j}=D_\rho} [E_j]\,.\]
Thus, the desired identity follows from
\[(1+q^{-\frac{1}{2}}\hat{z}^{[E_j]}
(q^{-1} X)^{-1})
=(1+q^{\frac{1}{2}}
\hat{z}^{[E_j]} X^{-1})
=q^{\frac{1}{2}}
\hat{z}^{[E_j]} X^{-1}
(1+q^{-\frac{1}{2}}
\hat{z}^{-[E_j]} X) \,. \]
Similarly, the relation
\[ p_\rho(q^{-\frac{1}{2} D_\rho^2}
\hat{z}^{[D_{\rho}]} \hat{f}_{\rho^\out}(X)
\hat{f}_{\rho^\iin}(qX)
X^{-D_{\rho}^2})
=
q^{-\frac{1}{2} D_{\bar{\rho}}^2}
\hat{z}^{[D_{\bar{\rho}}]} \hat{f}_{\bar{\rho}^\out}(X)
\hat{f}_{\bar{\rho}^\iin}(qX)
X^{-D_{\bar{\rho}}^2} \]
follows from 
\[(1+q^{-\frac{1}{2}}\hat{z}^{[E_j]}
 X^{-1})
=q^{-\frac{1}{2}}
\hat{z}^{[E_j]} X^{-1}
(1+q^{\frac{1}{2}}
\hat{z}^{-[E_j]} X)
=q^{-\frac{1}{2}}
\hat{z}^{[E_j]} X^{-1}
(1+q^{-\frac{1}{2}}
\hat{z}^{-[E_j]} (qX) ) \,. \]
\end{proof}

\begin{lem} \label{lem_bij_ch3}
The piecewise linear map
$\nu \colon B \rightarrow \bar{B}$
induces a bijection between broken lines of 
$\hat{\fD}^\can$ and broken lines of 
$\nu(\hat{\fD}^\can)$.
\end{lem}

\begin{proof}
This is a quantum version of
\cite[Lemma 3.28]{MR3415066}.

It is enough to compare bending and attached 
monomials of broken lines near 
a one-dimensional cone $\rho$ of $\Sigma$.
Indeed, away from such $\rho$, $\nu$ is linear and so the claim is obvious.

Let $\rho$ be a one-dimensional cone of $\Sigma$.
Let $\sigma_+$ and $\sigma_-$ be the two-dimensional
cones of $\Sigma$ bounding $\rho$, and let 
$\rho_+$ and $\rho_-$ be the other boundary
one-dimensional cones of 
$\sigma_+$ and 
$\sigma_-$ respectively, such that 
$\rho_-$, $\rho$ and $\rho_+$ are in
anticlockwise order.
Let $m_\rho$ be the primitive generator of $\rho$ pointing away from the origin.
We continue to use the notation introduced in the proof of 
\mbox{Lemma
\ref{lem_isom_ch3}}.

Let $\gamma$ be a quantum broken line in $B_0$, passing from 
$\sigma_-$ to $\sigma_+$ across $\rho$.
Let $c \hat{z}^s$, $s \in 
P_{\varphi_{\sigma_-}}$, be the monomial
attached to the domain of linearity of 
$\gamma$ preceding the crossing with $\rho$. Without loss of generality, we can assume $s=\varphi_{\sigma_-}(m_{\rho_-})$. Indeed, the pairing 
$\langle -,- \rangle$ is trivial on $P$, 
$\mathfrak{r}(s)$ is a linear combination of 
$m_\rho$ and $m_{\rho_-}$, and $\hat{z}^{\varphi_{\sigma_-}(m_\rho)}$ transforms trivially across $\rho$.

By the definition 
of a quantum broken line
(Definition \ref{defn_broken_line_ch3}), we have to show
that 
\[
p_{\sigma_+} \left(
\hat{z}^{\varphi_{\sigma_-}(m_{\rho_-})} 
\hat{f}_{\rho^\out}
(q^{-1} X)
\right) =
\hat{z}^{\bar{\varphi}_{\sigma_-}(m_{\bar{\rho}_-})} 
\hat{f}_{\bar{\rho}^\out}
(q^{-1} \bar{X})
\hat{f}_{\bar{\rho}^\iin}
(\bar{X}) \,.
\]

From the relations
\[
\hat{z}^{\varphi_{\rho}(m_{\rho_+})} 
\hat{z}^{\varphi_{\rho}(m_{\rho_-})}
= q^{\frac{1}{2} D_\rho^2}
\hat{z}^{[D_\rho]} 
X^{-D_{\rho}^2}\]
in 
$\kk[P_{\varphi_\rho}]$,
\[
\hat{z}^{\bar{\varphi}_{\bar{\rho}}(m_{\rho_+})} 
\hat{z}^{\bar{\varphi}_{\bar{\rho}}(m_{\bar{\rho}_-})}
= q^{\frac{1}{2} D_{\bar{\rho}}^2}
\hat{z}^{[D_{\bar{\rho}}]} 
\bar{X}^{-D_{\bar{\rho}}^2} \]
in 
$\kk[P_{\bar{\varphi}_{\bar{\rho}}}]$,
and using
$D_\rho^2=D_{\overline{\rho}}^2-a_\rho$
and
$D_\rho=D_{\overline{\rho}} -\sum_{j,D_{m_j}=D_\rho} [E_j]$, we get 
\[ 
p_{\sigma_+}(\hat{z}^{\varphi_{\rho}(m_{\rho_-})})
=\hat{z}^{\bar{\varphi}_{\bar{\rho}}
(m_{\bar{\rho_-}})}
\prod_{j, D_{m_j}=D_\rho}
(q^{-\frac{1}{2}}\bar{X}
\hat{z}^{-[E_j]})\,.\]
The result follows from the identity
\[
q^{-\frac{1}{2}}\bar{X}
\hat{z}^{-[E_j]}
(1+q^{-\frac{1}{2}}
\hat{z}^{[E_j]} (q^{-1}
\bar{X})^{-1})
=1+q^{-\frac{1}{2}} \hat{z}^{-[E_j]}
\overline{X}\,.\]
\end{proof}

\begin{lem} \label{lem_lift_ch3}
Let $\sigma$ be a two-dimensional cone of 
$\Sigma$. For every $Q \in \sigma$
and every $p \in B_0(\Z)$,
we have 
\[ p_\sigma(\Lift_{Q}(p))
=\Lift_{\nu(Q)}(\nu(p))\,.\]
\end{lem}

\begin{proof}
This is a direct consequence of
Lemma \ref{lem_bij_ch3}.
\end{proof}

\subsection{Conclusion of the proof of Theorem \ref{thm_consistency}}
\label{section_end_proof_ch3}
This section is parallel to the end of the proof of Theorem 3.8 given at the end of  \cite[\S 3.3]{MR3415066}.
We have to show that 
$\hat{\fD}^{\can}$ satisfies the two conditions in the definition
of consistency of a quantum scattering diagram (Definition
\ref{defn_consistency_ch3}).
\begin{itemize}
\item Let $Q$ and $Q'$ be generic points in $B_0$ contained in a common two-dimensional cone $\sigma$ of $\Sigma$, and let $\gamma$ be a path in the interior of 
$\sigma$ connecting $Q$ and $Q'$,
and transversely intersecting the rays of 
$\hat{\fD}$. We have to show that
\[\Lift_{Q'}(p)=\hat{\theta}_{\gamma, \hat{\fD}^\can}(\Lift_Q(p))\,.\]
By 
Lemma
\ref{lem_isom_ch3} and 
Lemma \ref{lem_lift_ch3}, it is enough to show that 
\[\Lift_{\nu(Q')}(\nu(p))=\hat{\theta}_{\nu(\gamma), \nu(\hat{\fD}^\can)}(\Lift_{\nu(Q)}(\nu(p)))\,,\]
which follows from
the combination of  
\mbox{Proposition 
\ref{prop_comp_ch3}} and 
Proposition 
\ref{prop_cps_ch3}.
\item Let $Q_-$ and $Q_+$ be two generic points in $B_0$, contained respectively  
in two-dimensional cones $\sigma_-$
and $\sigma_+$ of 
$\Sigma$, such that $\sigma_+$ and $\sigma_-$ intersect along a one-dimensional cone $\rho$ of $\Sigma$. Assuming further that 
$Q_-$ and $Q_+$ are contained in connected components of 
$B - \Supp_I(\hat{\fD})$ whose closures contain $\rho$, we have to show that  
$\Lift_{Q_+}(p) \in R_{\sigma_+, I}^{\hbar}$
and 
$\Lift_{Q_-}(p) \in R_{\sigma_-, I}^{\hbar}$ are both images under 
$\hat{\psi}_{\rho,+}$ and
$\hat{\psi}_{\rho,-}$ respectively of a single element 
$\Lift_\rho(p) \in R_{\rho,I}^{\hbar}$.
By 
Lemma
\ref{lem_isom_ch3} and 
Lemma \ref{lem_lift_ch3}, it is enough to prove the corresponding statement after application of $\nu$. This result follows from the combination of the remark at the end of
\cref{subsection_gluing_ch3} and 
the second point of \mbox{Lemma \ref{lem_finite}}.
\end{itemize} 

\section{Conclusion of the proofs of the main results}
\label{section_ext}

In this section we finish the proof of Theorem \ref{thm_1_ch3}--
\ref{thm_q}.

We fix $(Y,D)$ a Looijenga pair.
Let 
$\sigma_P \subset A_1(Y,\R)$ be a strictly 
convex polyhedral cone containing 
$NE(Y)_\R$.
Let $P \coloneqq \sigma_P \cap A_1(Y,\Z)$ be the associated monoid and let
$R \coloneqq \kk[P]$ be the corresponding $\kk$-algebra. 
We denote by $\eta \colon NE(Y) \rightarrow P$
the inclusion of $NE(Y)$ in $P$.
For the maximal ideal monomial of 
$J=\fm_R$ of $R$, the assumptions of
Theorem \ref{thm_consistency}
are satisfied and so the canonical quantum scattering 
diagram $\hat{\fD}^\can$ constructed from
$(Y,D)$, $P$, $\eta$ and $J$ in 
\cref{section_canonical}
is consistent.

If $D$ has $r \geqslant 3$ irreducible components, then we can apply Theorem
\ref{thm_construction_ch3} to produce, for every 
ideal $I$ of $P$ of radical $J$, the 
$R_I^{\hbar}$-algebra $A_I^{\hbar}$,
deformation quantization of 
$X_{I,\fD^\can}$. 
In \cref{subsection_torus} we
lift the torus action on $X_{I,\fD^\can}$
constructed in \cite[\S 5]{MR3415066}
to a torus action on $A_I^{\hbar}$.
This finishes the proof of Theorem \ref{thm_1_ch3} if $r \geqslant 3$.
In 
\cref{subsection_end_proof_thm1}, we finish the proof of Theorem \ref{thm_1_ch3} in general, that is \ for any $r \geqslant 1$.

In \cref{section_quant_v1_v2} we give an explicit description of the deformation quantization of the special fiber of the mirror family for $r=1$ and $r=2$. We finish the proof of Theorem
\ref{thm_2_ch3} in 
\cref{subsection_end_proof_thm2}, 
and the proof of Theorem 
\ref{thm_q} in \cref{subsection_end_proof_thm_q}.

\subsection{Torus equivariance}
\label{subsection_torus}

Let $I$ be an ideal of $P$ of radical 
$J=\fm_R$.
Recall from 
\cref{subsection_ghk} that $T^D \coloneqq \mathbb{G}_m^r$ is the torus whose character group  
$\chi(T^D)$ has a basis 
$e_{D_j}$ indexed by the irreducible
components $D_j$ of $D$, 
$1 \leqslant j \leqslant r$. 
The map 
\[\beta \mapsto \sum_{j=1}^r (\beta \cdot D_j)
e_{D_j}\] 
induces an action of 
$T^D$ on $S_I=\Spec R_I$.

Following \cite[\S 5]{MR3415066}, we consider
\[ w \colon B \rightarrow \chi(T_D) \otimes \R \,,\]
the unique piecewise linear map 
such that $w(0)=0$ and $w(m_{\rho_j})=e_{D_j}$
for all $1 \leqslant j \leqslant r$, where 
$m_{\rho_j}$ 
is the primitive generator of the ray 
$\rho_j$, viewed as an element of 
$B_0(\Z)$.

We assume that $r \geqslant 3$, so that 
$A_I^{\hbar}$ is defined by Theorem \ref{thm_construction_ch3}.

According to \cite[Theorem 5.2]{MR3415066},
the $T^D$-action on $\Spec R_I$
has a natural lift to $X_{I,\fD^{\can}}$, 
such that the decomposition 
\[H^0(X_{I,\fD^{\can}},\cO_{X_{I,\fD^{\can}}})
=A_I=\bigoplus_{p \in B(\Z)} R_I
\vartheta_p\]
as $R_I$-module is a weight decomposition
where $T^D$ acts on 
$\vartheta_p$ with weight $w(p)$.

We extend the action of $T^D$ on $R_I$ by 
$\kk$-algebra automorphisms to an action of 
$T^D$ on $R_I^{\hbar}$ by $\kk_{\hbar}$-automorphism by assigning weight zero to $\hbar$.

\begin{prop}\label{prop_torus}
The $T^D$-action on $A_I$ by 
$\kk$-algebra automorphisms, equivariant for the structure of $R_I$-algebra, lifts to a 
$T^D$-action on $A_I^{\hbar}$
by $\kk_{\hbar}$-automorphisms, equivariant for the structure of $R_I^{\hbar}$-algebra. Furthermore, the decomposition 
 \[A_I^{\hbar}=\bigoplus_{p \in B(\Z)} R_I^{\hbar}
\hat{\vartheta}_p\]
as $R_I^{\hbar}$-module 
is a weight decomposition where $T^D$ acts on 
$\hat{\vartheta}_p$ with weight $w(p)$.
\end{prop}

\begin{proof}
This is a quantum deformation of the proof of \cite[Theorem 5.2]{MR3415066}.
As 
\[ A_I^{\hbar}
=\Gamma (X^\circ_{I,\fD^\can},
\cO_{X^\circ_{I,\fD^\can}}^{\hbar})\,,\] 
it is enough 
to define the $T^D$-action on 
$\cO_{X^\circ_{I,\fD^\can}}^{\hbar}$. 
 
Let $\rho$ be a one-dimensional cone of $\Sigma$.
Let $\sigma_+$ and $\sigma_-$ be the two-dimensional
cones of $\Sigma$ bounding $\rho$, and let 
$\rho_+$ and $\rho_-$ be the other boundary
one-dimensional cones of 
$\sigma_+$ and 
$\sigma_-$ respectively, such that 
$\rho_-$, $\rho$ and $\rho_+$ are in
anticlockwise order.
According to 
\cref{subsection_building_blocks},
$R_{\rho,I}^{\hbar}$ is generated 
as $R_I^{\hbar}$-algebra by variables 
$X_+$, $X_-$ and $X$.
We define an action of 
$T^D$ on $R_{\rho,I}^{\hbar}$
by lifting the action of
$T^D$ on $R_I^{\hbar}$ and by assigning
weight $e_{D_{\rho_+}}$ to $X_+$,
weight $e_{D_{\rho_-}}$ to $X_-$,
and weight 
$e_{D_\rho}$ to $X$.
We have to check that this action is well defined, that is, preserves the relations between $X_+$, $X_-$ and $X$ defining 
$R_{\rho,I}^{\hbar}$.

The relation 
$X X_+ = q X_+ X$
(respectively,
$X X_- = q^{-1} X_- X$)
is clearly $T^D$-invariant as both 
the left-hand side and the right-hand side have weight $e_{D_\rho}+e_{D_{\rho_+}}$
(respectively, $e_{D_\rho}+e_{D_{\rho_-}}$).

Let us consider the remaining relations:
\begin{enumerate}
\item \[ X_+ X_-=q^{\frac{1}{2} D_\rho^2}
\hat{z}^{[D_{\rho}]} \hat{f}_{\rho^\out}(q^{-1} X)
\hat{f}_{\rho^\iin}(X)
X^{-D_{\rho}^2} \,,\]
\item \[X_- X_+=q^{-\frac{1}{2} D_\rho^2}
\hat{z}^{[D_{\rho}]} \hat{f}_{\rho^\out}(X)
\hat{f}_{\rho^\iin}(qX)
X^{-D_{\rho}^2} \,.
\]
\end{enumerate}

For the canonical quantum scattering diagram $\hat{\fD}^\can$, defined in Definition 
\ref{def_canonical_scattering}, we have 
$\hat{f}_{\rho^\iin}(X)=1$, and
$\hat{f}_{\rho^\out}(X)$ is a power series of the form
\[ 
\sum_{\beta \in NE(Y)_\rho} c_\beta(\hbar)
\hat{z}^{\beta}X^{-\ell_\beta} \,,\]
for some $c_\beta(\hbar) \in \Q[\![\hbar]\!]$. According to the definition of $NE(Y)_\rho$
(see \cref{section_log_gw_ch3}), 
for $\beta \in NE(Y)_\rho$, we have
$\beta \cdot D_\rho=\ell_\beta$ and 
$\beta \cdot D_{\rho'}=0$ if 
$\rho' \neq \rho$. Thus, by the definition of the action of $T^D$ on $R_I^{\hbar}$,
$T^D$ acts on $\hat{z}^{\beta}$ with weight $\ell_\beta e_{D_\rho}$.
On the other hand, $X^{-\ell_\beta}$ has weight $-\ell_\beta e_{D_\rho}$.
It follows that $\hat{f}_{\rho^\out}(X)$
has weight $0$. 
On the other hand, by definition of the action of $T^D$ on $R_I^{\hbar}$, 
$\hat{z}^{[D_\rho]}$ has weight
\[ \sum_{j=1}^r (D_\rho \cdot D_j)
e_{D_j} = e_{D_{\rho_+}}+e_{D_{\rho_-}}
+D_\rho^2 e_{D_\rho}\,.\]
Thus, the above relations (i)
and (ii) are indeed 
$T^D$-invariant: 
the left-hand sides ($X_+ X_-$ or 
$X_-X_+$) have weight
\[ e_{D_{\rho_+}}+e_{D_{\rho_-}}\]
and the right-hand sides
($q^{\frac{1}{2} D_\rho^2}
\hat{z}^{[D_{\rho}]} \hat{f}_{\rho^\out}(q^{-1} X)
X^{-D_{\rho}^2}$ 
or
$q^{-\frac{1}{2} D_\rho^2}
\hat{z}^{[D_{\rho}]} \hat{f}_{\rho^\out}(X)
X^{-D_{\rho}^2}$)
have weight
\[ ( e_{D_{\rho_+}}+e_{D_{\rho_-}}
+D_\rho^2 e_{D_\rho})-D_\rho^2 e_{D_\rho}\,.\]

Having defined an action of 
$T^D$ on the 
$R_I^{\hbar}$-algebras 
$R_{\rho,I}^{\hbar}$,
in order to define an action of 
$T^D$ on 
$\cO_{X^\circ_{I,\fD^\can}}^{\hbar}$,
it remains to check that the gluing transformations 
$\hat{\theta}_{\gamma_\sigma}^{\hbar}$
of \cref{subsection_gluing_ch3}
are $T^D$-equivariant.
Let $\sigma$ be a two-dimensional cone of $\Sigma$, bounded by rays $\rho_L$ and $\rho_R$, such that 
$\rho_L,\sigma,\rho_R$ are in  
anticlockwise order. It follows from
Proposition \ref{prop_R}
and the fact that $\hat{f}_{\rho_L^\out}$ and $\hat{f}_{\rho_R^\out}$ have weight $0$, that the actions of $T^D$
on $R_{\rho_L,I}^{\hbar}$ and 
$R_{\rho_R,I}^{\hbar}$ restrict to the same action on $R_{\sigma,I}^{\hbar}$:
for $p \in P_{\varphi_\sigma}$,
one can uniquely write 
$p=\varphi_\sigma(\mathfrak{r}(p))+p'$
for some $p' \in P$, one can uniquely write 
$\mathfrak{r}(p)=n_R m_{\rho_R}+n_L m_{\rho_L}$ with 
$m_L, m_R \in \Z$, and then $T^D$ acts on 
$\hat{z}^p$ with weight
$n_R e_{D_{\rho_R}}+n_L e_{D_{\rho_L}}
+\sum_{j=1}^r (p' \cdot D_j) e_{D_j}$.

For every ray $\fd$
of $\hat{\fD}^{\can}$ contained in $\sigma$, the monomials appearing 
in $\hat{H}_\fd$ and so in $\hat{f}_\fd$
are of the form 
$\hat{z}^{\beta -\varphi_{\sigma}(\ell_\beta 
m_\fd)}$ with 
$\beta \in NE(Y)_\fd$, see Definition 
\ref{def_canonical_scattering}.
Writing $m_\fd=n_R m_{\rho_R}+n_L m_{\rho_L}$, we get that $T^D$ acts on 
$\hat{z}^{\beta -\varphi_{\sigma}(\ell_\beta 
m_\fd)}$ with weight
\[ -\ell_\beta n_R e_{D_{\rho_R}}-
\ell_\beta n_L e_{D_{\rho_L}}
+\sum_{j=1}^r (\beta \cdot D_j) e_{D_j}\,.\]
But by definition of $NE(Y)_\fd$
(see \cref{section_log_gw_ch3}), the condition $\beta \in NE(Y)_\fd$ means that 
$\beta\cdot D_{\rho_R}=\ell_\beta n_R$, 
$\beta \cdot D_{\rho_L}=\ell_\beta n_L$,
and $\beta \cdot D_{\rho'}=0$ if $\rho' \neq \rho_R$ and $\rho' \neq \rho_L$.
Thus, all the monomials $\hat{z}^{\beta -\varphi_{\sigma}(\ell_\beta 
m_\fd)}$ have weight zero.
It follows that the gluing transformations 
$\hat{\theta}_{\gamma_\sigma}^{\hbar}$
are $T^D$-equivariant.

The check that $\hat{\vartheta}_p$ is an eigenfunction of the $T^D$-action with weight $w(p)$ is now formally identical to the corresponding classical 
check given in the proof of \cite[Theorem 5.2]{MR3415066}. As the scattering automorphisms 
have weight $0$, the weights of the monomials 
on the various domains of linearity of 
a broken line are identical and so it is enough to consider the unbounded domain of linearity.
In this case, the monomial is 
$\hat{z}^{\varphi_{\tau_p}(p)}$, 
which has weight $w(p)$. 
\end{proof}

\subsection{Conclusion of the proof of Theorem \ref{thm_1_ch3}}
\label{subsection_end_proof_thm1}

We fix $(Y,D)$ a Looijenga pair.
Let 
$\sigma_P \subset A_1(Y,\R)$ be a strictly 
convex polyhedral cone containing 
$NE(Y)_\R$.
Let $P \coloneqq \sigma_P \cap A_1(Y,\Z)$ be the associated monoid and let
$R \coloneqq \kk[P]$ be the corresponding $\kk$-algebra. For 
$J=\fm_R$ the maximal ideal monomial of $R$, the assumptions of
Theorem \ref{thm_consistency}
are satisfied and so the canonical quantum scattering 
diagram $\hat{\fD}^\can$ is consistent.

If $(Y,D)$ admits a toric model, then $D$ has $r \geqslant 3$ irreducible components,
and so we can apply Theorem
\ref{thm_construction_ch3}. Combined with \mbox{Proposition
\ref{prop_torus}},
this
proves Theorem \ref{thm_1_ch3} in this case.

In general, it is proven in \cite[\S 6.2]{MR3415066}
that 
\[H^0(X_I,\cO_{X_I})
=A_I \coloneqq \bigoplus_{p \in B(\Z)}
R_I \theta_p \,,\]
with the $R_I$-algebra structure determined by the classical version of the product formula given in Theorem
\ref{thm_product}. So Theorem
\ref{thm_1_ch3} follows from the following result.

\begin{prop} \label{prop_algebra}
For every monomial ideal $I$ of $R$
of radical $\fm_R$, the 
multiplication rule of 
Theorem \ref{thm_product}
defines a structure of 
$R_I^{\hbar}$-algebra on 
the $R_I^{\hbar}$-module
\[A_I^{\hbar} \coloneqq 
\bigoplus_{p \in B(\Z)}
R_I^{\hbar} \hat{\vartheta}_p \,.\]
\end{prop}

\begin{proof}
If $(Y,D)$ admits a toric model, then $D$ has $r \geqslant 3$ components and so the result follows from Theorem
\ref{thm_product}.

In general, there is a corner blow-up 
$(Y',D')$ of $(Y,D)$ admitting a toric model. The result for $(Y',D')$ implies the result for $(Y,D)$ as in
\cite[\S 6.2]{MR3415066}.
 
\end{proof}

\subsection{Quantization of $\V_1$ and $\V_2$}
\label{section_quant_v1_v2}

By Poposition \ref{prop_algebra}, for
every monomial ideal $I$ of $R$
of radical $\fm_R$, we have a structure of 
$R_I^{\hbar}$-algebra on 
\[ A_I^{\hbar}=\bigoplus_{p\in B(\Z)}
R_I^{\hbar} \hat{\vartheta}_p \,.\]
In this section we explicitly describe this algebra for $I=\fm_R$.

In the classical limit $\hbar=0$, we get a commutative $R_I$-algebra which, by 
\cite{MR3415066} is the algebra of functions on the variety $\V_r$, where $r$ is the number of irreducible components of $D$, and 
\begin{itemize}
\item if $r \geqslant 3$, $\V_r$ is the $r$-cycle of coordinates planes in the affine space $\A^r$, 
$\V_r
=\A^2_{x_1, x_2} \cup \A_{x_2,x_3}^2 \cup
\dots \cup \A^2_{x_r,x_1}
\subset \A^r_{x_1,\dots,x_r}$.
\item if $r=2$, $\V_2$ is a union of two affine planes\footnote{In \cite{MR3415066}, the description
$\V_2
=\Spec \kk[u,v,w]/(w^2-u^2 v^2)$
is given. This is equivalent to our description via the change of 
variables $x=\sqrt{2}u$,
$y=\sqrt{2}v$,
$z=w+uv$\,.},
\[\V_2=\Spec \kk[x,y,z]/(xyz-z^2)\,,\]
the affine cone over the union of the two rational curves $z=0$ and $xy-z=0$, intersecting in two points, embedded 
in the weighted projective plane 
$\PP^{1,1,2}$.
\item if $r=1$, $\V_1= \Spec 
\kk[x,y,z]/(xyz-x^2-z^3)$, the affine cone over a nodal curve embedded in the 
weighted projective plane 
$\PP^{(3,1,2)}_{x,y,z}$.
\end{itemize}

When $r \geqslant 3$, the explicit description of  $A_{\fm_R}^{\hbar}$ follows from the combination of \cref{gluing_I=J} and the beginning of the 
proof of Theorem \ref{thm_construction_ch3}: we have 
\[\hat{\vartheta}_m
\cdot \hat{\vartheta}_{m'} =
\begin{dcases}
q^{\frac{1}{2} \langle m,m'\rangle}
\hat{\vartheta}_{m+m'} & \text{if $m$ and $m'$ lie in a common cone of $\Sigma$} \\
0 & \text{otherwise.} 
\end{dcases}\] 
In particular, denoting $v_1, \dots, v_r$ the primitive generators of the one-dimensional cones $\rho_1, \dots, \rho_r$ 
of $\Sigma$, $A_{\fm_R}^{\hbar}$ is generated as $\kk_{\hbar}$-algebra by 
$\hat{\vartheta}_{v_1}, \dots, \hat{\vartheta}_{v_r}$.

For $r=2$ and $r=1$, 
computing $A_{\fm_R}^{\hbar}$ is slightly more subtle and the answer is given below in 
Propositions 
\ref{prop_v2} and 
\ref{prop_v1}.
 
Both $\V_1$ and $\V_2$ are hypersurfaces in 
$\A^3_{x,y,z}$.
Evey hypersurface $F(x,y,z)=0$ in
$\A^3_{x,y,z}$ has a natural Poisson structure defined by
\[\{x,y\}=\frac{\partial F}{\partial z} \,,
\{y,z\}=\frac{\partial F}{\partial x} \,,
\{z,x\}=\frac{\partial F}{\partial y} \,,\]
see \cite{MR2734346} for example.

For $\V_2$ and $F(x,y,z)
=z^2-xyz$, we get 
\[\{x,y\}=2z-xy \,, \{y,z\}=-yz \,,\{z,x\}
=-zx \,.\]
It follows from 
$\{y,z\}=-yz$ and $\{z,x\}
=-zx $ that this bracket coincides with the one coming from the standard symplectic form on the two natural copies of 
$(\G_m)^2$ contained in $\V_2$.

For $\V_1$ and $F(x,y,z)
=z^3+x^2-xyz$, we get 
\[\{x,y\}=3z^2-xy \,,
\{y,z\}=2x-yz \,,
\{z,x\}=-zx\,.\]
It follows from
$\{x,z\}=xz$ that the above Poisson structure is 
indeed the one induced by the standard symplectic
form on the
natural copy of 
$(\G_m)^2$ contained in $\V_1$.

We first explain how to recover the 
above Poisson brackets from the formula given 
by Corollary \ref{cor_poisson} in terms of classical
broken lines. We then use the formula of Theorem \ref{thm_product} in terms of quantum broken lines to compute the $q$-commutators deforming these Poisson brackets.

For $\V_2$, the tropicalization $B$
contains two two-dimensional cones 
$\sigma_1$, and $\sigma_2$, and two one-dimensional cones $\rho_1$ and $\rho_2$.
Let $v_1$ and $v_2$ in $B(\Z)$ be the primitive generators of $\rho_1$ and $\rho_2$. Cutting $B$ along $\rho_1$, we can identify $B$ with the union of two cones 
in $\R^2$. More precisely, we can find $w$, $v_2$, $w' \in \Z^2$
such that $\langle w, v_2 \rangle = \langle v_2, w' \rangle =1$, and such that $B$
can be viewed as the union of the two cones
$\R_{\geqslant 0}w+\R_{\geqslant 0}v_2$ and 
$\R_{\geqslant 0}v_2+\R_{\geqslant 0}w'$ with some identification of $\R_{\geqslant 0} w$ and $\R_{\geqslant 0}w'$ identifying $w$ and $w'$.
We have $x=\vartheta_{v_1}=\vartheta_w
=\vartheta_{w'}$, $y=\vartheta_{v_2}$, 
$z=\vartheta_{w+v_2}$.
The broken lines description of the product gives
\[xy=\vartheta_{v_1} \vartheta_{v_2}
=\vartheta_{w+v_2}+\vartheta_{w'+v_2} \,,\]
and 
\[\vartheta_{w+v_2}
\vartheta_{w'+v_2}
=0\,,\]
so 
$\vartheta_{w'+v_2}=xy-z$ and 
$(xy-z)z=0$, which is indeed the equation defining $\V_2$.
We have 
\[\{x,y\}=
\{\vartheta_{v_1},\vartheta_{v_2}\}
=\langle
w,v_2
\rangle 
\vartheta_{w+v_2}
+
\langle
w',
v_2
\rangle \vartheta_{w'+v_2}
=\vartheta_{w+v_2}-\vartheta_{w'+v_2}\,\]
Using $\vartheta_{w'+v_2}=xy-z$, we get 
$\{x,y\}=2z-xy$.
We have 
\[ \{y,z\}=
\{\vartheta_{v_2},
\vartheta_{w+v_2}\}
=\langle
v_2,w+v_2
\rangle 
\vartheta_{w+2v_2}
=-\vartheta_{v_2}
\vartheta_{w+v_2}
=-yz \,.\]
Finally, we have 
\[\{z,x\}
=
\langle
w+v_2,
w
\rangle
\vartheta_{2w+v_2}
=-\vartheta_{w}\vartheta_{w+v_2}
=-zx\,.\]
Using the formula of Theorem \ref{thm_product}, we compute the $q$-commutators deforming the above Poisson brackets.
We have 
\[\hat{x}\hat{y}=
\hat{\vartheta}_{v_1} 
\hat{\vartheta}_{v_2}
=q^{\frac{1}{2}}
\hat{\vartheta}_{w+v_2}
+q^{-\frac{1}{2}}
\hat{\vartheta}_{w'+v_2} \,,\]
so $\hat{\vartheta}_{w'+v_2}
=q^{\frac{1}{2}} \hat{x}\hat{y}-q \hat{z}^2$.
On the other hand, we have
\[\hat{y}\hat{x}=
\hat{\vartheta}_{v_2} 
\hat{\vartheta}_{v_1}
=q^{-\frac{1}{2}}
\hat{\vartheta}_{w'+v_2}
+q^{\frac{1}{2}}
\hat{\vartheta}_{w'+v_2} \,,\]
\[q^{-\frac{1}{2}}\hat{y}\hat{x}
=q^{-1}
\hat{\vartheta}_{w'+v_2}
+
\hat{\vartheta}_{w'+v_2} \,,\]
and so 
\[q^{\frac{1}{2}}
\hat{x}\hat{y}
-q^{-\frac{1}{2}}
\hat{y}\hat{x}=(q-q^{-1})
\hat{z}^2\,.\]
We have 
\[\hat{y}\hat{z}=\hat{\vartheta}_{v_2} 
\hat{\vartheta}_{w+v_2}=
q^{-\frac{1}{2}}
\hat{\vartheta}_{w+2v_2}\,,\]
and
\[\hat{z}\hat{y}=\hat{\vartheta}_{w+v_2}
\hat{\vartheta}_{v_2}
=
q^{\frac{1}{2}}
\hat{\vartheta}_{w+2v_2}\,,\]
so
\[q^{\frac{1}{2}}
\hat{y}\hat{z}-q^{-\frac{1}{2}}\hat{z}\hat{y}
=0\,.\]
We have 
\[\hat{z}\hat{x}
=\hat{\vartheta}_{w+v_2}
\hat{\vartheta}_{w}
=q^{-\frac{1}{2}}
\hat{\vartheta}_{2w+v_2}\,\]
and 
\[\hat{x}\hat{z}=\hat{\vartheta}_{w}
\hat{\vartheta}_{w+v_2}=q^{\frac{1}{2}}\hat{\vartheta}_{2w+v_2}\,,\] 
so
\[q^{\frac{1}{2}}
\hat{z}\hat{x}-q^{-\frac{1}{2}}\hat{x}\hat{z}
=0 \,.\]
Finally, we compute the $q$-deformation of the cubic relation $F=0$: 
\[\hat{x}\hat{y}\hat{z}
=\hat{\vartheta}_{w}
q^{-\frac{1}{2}}
\hat{\vartheta}_{w+2v_2}
=q^{\frac{1}{2}}
\hat{z}^2\,.\]

In summary, we have proved the following proposition.

\begin{prop} \label{prop_v2}
The deformation quantization of 
$\V_2$ given by the product formula of Theorem \ref{thm_product} is the associative $\kk_{\hbar}$-algebra
generated by variables $\hat{x},\hat{y},\hat{z}$
and with relations 
\[q^{\frac{1}{2}}
\hat{x}\hat{y}-q^{-\frac{1}{2}}\hat{y}\hat{x}
=(q-q^{-1})\hat{z}\,,\]
\[q^{\frac{1}{2}}
\hat{y}\hat{z}-q^{-\frac{1}{2}}\hat{z}\hat{y}
=0 \,,\]
\[q^{\frac{1}{2}}
\hat{z}\hat{x}-q^{-\frac{1}{2}}\hat{x}\hat{z}
=0 \,,\]
\[\hat{x}\hat{y}\hat{z}
=q^{\frac{1}{2}}
\hat{z}^2
\,.\]
\end{prop}

For $\V_1$, the tropicalization $B$
contains one two-dimensional cone $\sigma$ and one one-dimensional cone $\rho$.
Let $v$ in $B(\Z)$ be the primitive generator of $\rho$. Cutting $B$ along $\rho$, we can identify $B$ as a quadrant in $\R^2$ with an identification of the two boundary rays. Denote $w=(1,0)$ and 
$w'=(0,1)$. The description of the product of classical theta functions
by broken lines is given in 
\cite[\S 6.2]{MR3415066}. 
We have $x=\vartheta_{2w+w'}$, 
$y=\vartheta_v=\vartheta_w=\vartheta_{w'}$, 
$z=\vartheta_{w+w'}$.
We have
\[\{x,y\}=\{\vartheta_{2w+w'},\vartheta_v\}
=\langle (2,1),(1,0)\rangle \vartheta_{3w+w'}
+\langle (2,1),(0,1)\rangle \vartheta_{2w+2w'}\]
\[=-\vartheta_{3w+w'} +2 \vartheta_{2w+2w'} \,.\]
On the other hand, we have 
$xy=\vartheta_{3w+w'}+\vartheta_{2w+2w'}$ and 
$z^2=\vartheta_{2w+2w'}$, and so 
$\{x,y\}=3z^2-xy$.
We have 
\[\{y,z\}=\{\vartheta_v,\vartheta_{w+w'}\}
=\langle (1,0),(1,1) \rangle 
\vartheta_{2w+w'} 
+\langle (0,1),(1,1)\rangle 
\vartheta_{w+2w'}\]
\[= \vartheta_{2w+w'}-\vartheta_{w+2w'}\,.\]
On the other hand, we have 
$yz=\vartheta_{2w+w'}+\vartheta_{w+2w'}$ and 
$x=\vartheta_{2w+w'}$, and so 
$\{y,z\}=2x-yz$.
We have 
\[\{z,x\}=\{\vartheta_{w+w'},\vartheta_{2w+w'}\}
=\langle (1,1),(2,1) \rangle \vartheta_{3w+2w'}
=- \vartheta_{3w+2w'}\,.\]
On the other hand, we have
$zx=\vartheta_{3w+2w'}$ and so
$\{z,x\}=-zx$.  

Using the formula of Theorem \ref{thm_product}, we compute the $q$-commutators deforming the above Poisson brackets.
We have 
\[\hat{x}\hat{y}
=\hat{\vartheta}_{2w+w'}
\hat{\vartheta}_v
=q^{-\frac{1}{2}} \hat{\vartheta}_{3w+w'}
+q \hat{\vartheta}_{2w+2w'}\,,\]
so
\[\hat{\vartheta}_{3w+w'}=q^{\frac{1}{2}}
\hat{x}\hat{y}-q^{\frac{3}{2}} \hat{z}^2\,.\]
On the other hand, we have 
\[\hat{y}\hat{x}=
\hat{\vartheta}_v \hat{\vartheta}_{2w+w'}
=q^{\frac{1}{2}}
\hat{\vartheta}_{3w+w'}+q^{-1} \hat{\vartheta}_{2w+2w'} \,,\]
\[q^{-\frac{1}{2}}\hat{y}\hat{x}=
\hat{\vartheta}_{3w+w'}+q^{-\frac{3}{2}} \hat{z}^2\,,\]
and so 
\[q^{\frac{1}{2}}
\hat{x}\hat{y}-q^{-\frac{1}{2}}\hat{y}\hat{x}
=(q^{\frac{3}{2}}-q^{-\frac{3}{2}})
\hat{z}^2 \,.\]
We have 
\[\hat{y}\hat{z}=\hat{\vartheta}_v 
\hat{\vartheta}_{w+w'}=q^{\frac{1}{2}}
\hat{\vartheta}_{2w+w'}
+q^{-\frac{1}{2}}
\hat{\vartheta}_{w+2w'}\,,\]
so
\[\hat{\vartheta}_{w+2w'}
=q^{\frac{1}{2}}
\hat{y}\hat{z}
-q \hat{x}\,.\]
On the other hand, we have 
\[\hat{z}\hat{y}=\hat{\vartheta}_{w+w'}
\hat{\vartheta}_v
=q^{-\frac{1}{2}}
\hat{\vartheta}_{2w+w'}
+q^{\frac{1}{2}}
\hat{\vartheta}_{w+2w'}\,,\]
\[q^{-\frac{1}{2}}\hat{z}\hat{y}
=
q^{-1}\hat{x}
+
\hat{\vartheta}_{w+2w'}\,,\]
and so
\[q^{\frac{1}{2}}
\hat{y}\hat{z}-q^{-\frac{1}{2}}\hat{z}\hat{y}
=(q-q^{-1})
\hat{x} \,.\]
We have 
\[\hat{z}\hat{x}
=\hat{\vartheta}_{w+w'}
\hat{\vartheta}_{2w+w'}
=q^{-\frac{1}{2}}
\hat{\vartheta}_{3w+2w'}\,\]
\[\hat{\vartheta}_{3w+2w'}
=q^{\frac{1}{2}}\hat{z}\hat{x}\,.\]
On the other hand, we have 
\[\hat{x}\hat{z}=\hat{\vartheta}_{2w+w'}
\hat{\vartheta}_{w+w'}=q^{\frac{1}{2}}\hat{\vartheta}_{3w+2w'}\,,\]
and so
\[q^{\frac{1}{2}}
\hat{z}\hat{x}-q^{-\frac{1}{2}}\hat{x}\hat{z}
=0 \,.\]

Finally, we compute the $q$-deformation of the cubic relation $F=0$:
\[\hat{x}\hat{y}\hat{z}
=\hat{\vartheta}_{2w+w'}(q^{\frac{1}{2}}
\hat{\vartheta}_{2w+w'}
+q^{-\frac{1}{2}}
\hat{\vartheta}_{w+2w'})
=q^{\frac{1}{2}}
\hat{\vartheta}_{2w+w'}^2
+q^{-\frac{1}{2}}
q^{\frac{3}{2}}
\hat{\vartheta}_{3w+3w'}\,,\]
\[\hat{x}\hat{y}\hat{z}
=q^{\frac{1}{2}}
\hat{x}^2
+
q
\hat{z}^3\,.\]

In summary, we have proved the following proposition.

\begin{prop} \label{prop_v1}
The deformation quantization of 
$\V_1$ given by the product formula of Theorem \ref{thm_product} is the associative $\kk_{\hbar}$-algebra
generated by variables $\hat{x},\hat{y},\hat{z}$
and with relations 
\[q^{\frac{1}{2}}
\hat{x}\hat{y}-q^{-\frac{1}{2}}\hat{y}\hat{x}
=(q^{\frac{3}{2}}-q^{-\frac{3}{2}})
\hat{z}^2 \,,\]
\[q^{\frac{1}{2}}
\hat{y}\hat{z}-q^{-\frac{1}{2}}\hat{z}\hat{y}
=(q-q^{-1})
\hat{x} \,,\]
\[q^{\frac{1}{2}}
\hat{z}\hat{x}-q^{-\frac{1}{2}}\hat{x}\hat{z}
=0 \,,\]
\[\hat{x}\hat{y}\hat{z}
=q^{\frac{1}{2}}
\hat{x}^2
+
q
\hat{z}^3\,.\]
\end{prop}

\subsection{Conclusion of the proof of Theorem \ref{thm_2_ch3}}
\label{subsection_end_proof_thm2}

In this section we finish the proof of 
Theorem \ref{thm_2_ch3}, which is done
by combination of
\mbox{Proposition \ref{prop_ideal}}
and 
\mbox{Proposition 
\ref{prop_I_min}}.
We follow
\cite[\S 6.1]{MR3415066}.

For every monomial ideal $I$ of $P$, we define the free
$R_I^{\hbar}$-module
\[A_I^{\hbar} = \bigoplus_{p\in B(\Z)}
R_I^{\hbar} \hat{\vartheta}_p \,.\]
According to Proposition
\ref{prop_algebra}, if $I$ has radical 
$\fm_R$, then there is a natural 
$R_I^{\hbar}$-algebra structure on 
$A_I^{\hbar}$.

Let $\Gamma \subset B(\Z)$
be a finite collection of 
integral points such that 
the corresponding quantum
theta functions 
$\hat{\vartheta}_p$
generate the 
$\kk_{\hbar}$-algebra 
$A_{\fm_R}^{\hbar}$.
Using the notation of 
\cref{section_quant_v1_v2}, we can take
$\Gamma=\{v_1,\dots,v_r\}$ if 
$r\geqslant 3$, 
$\Gamma=\{v_1,v_2,w+v_2\}$ if $r=2$, and 
$\Gamma=\{v,w+w',2w+w'\}$ if $r=1$.

\begin{prop} \label{prop_ideal}
There exists a unique 
minimal radical monomial ideal
$J_{\min}^{\hbar}$ of $P$
such that, for every monomial ideal $I$
of $P$ 
of radical containing $J_{\min}^{\hbar}$: 
\begin{itemize}
\item there exists a $R_I^{\hbar}$-algebra structure on 
$A_I^{\hbar}$ such that, for every 
$k>0$,
the natural isomorphism of 
$R_{I+\fm^k}^{\hbar}$-modules $A_I^{\hbar} \otimes
R_{I+\fm^k}^{\hbar} = A_{I+\fm^k}^{\hbar}$
is an isomorphism of $R_{I+\fm^k}^{\hbar}$-algebras.
\item the quantum theta functions $\hat{\vartheta}_p$,
$p \in \Gamma$, generate $A_I^{\hbar}$
as an $R_I^{\hbar}$-algebra.
\end{itemize} 
\end{prop}

\begin{proof}
Follows as its classical version, Proposition 6.5 of
\cite{MR3415066}.
\end{proof}

As in
\cite[\S 6.1]{MR3415066}, the first point of Proposition 
\ref{prop_ideal} is equivalent to the fact that for every 
$p_1, p_2 \in B(\Z)$,
at most finitely many
terms 
$\hat{z}^\beta \hat{\vartheta}_p$
with $\beta \notin I$ appear 
in the expansion given by Theorem
\ref{thm_product}
for 
$\hat{\vartheta}_{p_1}
\hat{\vartheta}_{p_2}$.

\begin{prop} \label{prop_I_min}
Suppose that 
$F \subset \sigma_P$ is a face such
that $F$ does not contain the class of every component of 
$D$. Then $J_{\min}^{\hbar} \subset P-F$.
If $(Y,D)$ is positive, then 
$J_{\min}^{\hbar}=0$.
\end{prop}

\begin{proof}
The proof is formally identical to the proof of its classical version, 
\cite[Proposition 6.6]{MR3415066}.
The main input, the $T^D$-equivariance, is given in our case by Proposition 
\ref{prop_torus}.
\end{proof}

Let $J_{\min}$ be the ideal defined by \cite[Proposition 6.5]{MR3415066}. We obviously have 
$J_{\min} \subset J_{\min}^{\hbar}$, as the vanishing of all genus Gromov-Witten invariants includes the vanishing of genus zero Gromov-Witten invariants. 
If $(Y,D)$ is positive then $J_{\min}
=J_{\min}^{\hbar}=0$. In general, it is unclear if we always have $J_{\min}=J_{\min}^{\hbar}$ or if there are examples with $J_{\min} \neq J_{\min}^{\hbar}$.
Geometrically, the question is whether or not some vanishing of genus-0 Gromov--Witten invariants implies a vanishing of all higher-genus Gromov--Witten invariants.

\subsection{Conclusion of the proof of Theorem \ref{thm_q}: $q$-integrality}
\label{subsection_end_proof_thm_q}

The $R_I^{\hbar}$-algebra structure on 
\[ A_I^{\hbar} = \bigoplus_{p\in B(\Z)}
R_I^{\hbar} \hat{\vartheta}_p \]
is given by the product formula of Theorem \ref{thm_product},
\[ \hat{\vartheta}_{p_1}
\hat{\vartheta}_{p_2}
= \sum_{p \in B(\Z)} C_{p_1, p_2}^p \hat{\vartheta}_p \,.\]
A priori, we have $C_{p_1,p_2}^p
\in R_I^{\hbar}
=R_I[\![\hbar]\!]$. Theorem
\ref{thm_q} follows from
the following Proposition.

\begin{prop} \label{prop_q_integ}
For every $p_1,p_2,p_3 \in B(\Z)$, we have 
\[C_{p_1,p_2}^p \in R_I^q=R_I[q^{\pm
\frac{1}{2}}]\,,\] where $q=e^{i \hbar}$.
More precisely, $C_{p_1,p_2}^p$ is the power series expansion around $\hbar=0$ of a Laurent polynomial in $q^{\frac{1}{2}}$ after the change of variables $q=e^{i \hbar}$.
\end{prop}

\begin{proof}
Recall that, if $\gamma$ is a quantum broken line of endpoint 
$Q$ in a cone $\tau$ of $ \Sigma$,
we write the monomial
$\Mono(\gamma)$
attached to the segment ending at $Q$
as 
\[ \Mono(\gamma) = c(\gamma) \hat{z}^{\varphi_{\tau} (s(\gamma))} \]
with $c(\gamma)
\in \kk_{\hbar}[P_{\varphi_\tau}]$
and $s(\gamma) 
\in \Lambda_{\tau}$.

By definition, we have 
\[ C_{p_1,p_2}^p
= \sum_{\gamma_1, \gamma_2}
c(\gamma_1) c(\gamma_2) 
q^{\frac{1}{2} \langle
s(\gamma_1), s(\gamma_2) \rangle} \,,\]
where the sum is over all broken lines $\gamma_1$ and $\gamma_2$, of asymptotic charges $p_1$ and $p_2$,
satisfying $s(\gamma_1)+s(\gamma_2)=p$, and both ending at a given point $z \in B 
-\Supp_I(\fD^\can)$ very close to $p$. 

So it is enough to show that, for every 
$\gamma$ quantum broken line of endpoint 
$Q$ in a cone $\tau$ of $ \Sigma$,
we have $c(\gamma) \in \kk_q[P_{\varphi_\tau}]$. We will show more generally that for every quantum broken line $\gamma$ 
of $\hat{\fD}^\can$, and for every $L$ domain of linearity of $\gamma$, the attached monomial $m_L=c_L \hat{z}^{p_L}$ satisfies
$c_L \in \kk_q$.

This is obviously true if $L$ is the unbounded domain of linearity of $\gamma$ since then $c_L=1$.
Given the formula in Definition 
\ref{defn_broken_line_ch3} specifying the change of monomials when the
quantum broken line bends, it is then enough to show that, for every ray $(\fd,\hat{H}_\fd)$ of 
$\hat{\fD}^\can$, the corresponding 
$\hat{f}_\fd$ is in 
$\kk_q\widehat{[P_{\varphi_{\tau_\fd}}}]$.

Given the argument used in
\cite[\S 6.2]{MR3415066}, we can assume that 
$(Y,D)$ admits a toric model.
Furthermore, by deformation invariance of log Gromov--Witten invariants in log smooth families, we can assume that $(Y,D)$ is obtained from its toric model by blowingup distinct points, that is, that there exists $\fm=(m_1,\dots, m_n)$ such that $(Y,D)=(Y_\fm,\partial
Y_\fm)$, as in 
\cref{subsection_GS_locus}.
In 
\cref{subsection_pushing} we introduced a quantum scattering diagram $\nu(\hat{\fD}^\can)$. 
From the definition of 
$\nu(\hat{\fD}^\can)$ and the explicit formulas given in the proof of \mbox{Lemma
\ref{lem_isom_ch3}}
comparing $\hat{\fD}^\can$ and 
$\nu(\hat{\fD}^\can)$, it is enough to prove the result for outgoing rays $\nu(\hat{\fD}^\can)$.

By
Proposition 
\ref{prop_comp_ch3}, we have 
$\nu(\hat{\fD}^\can) =
S(\hat{\fD}_\fm)$. So it remains to show that, for every outgoing ray 
$(\fd,\hat{H}_\fd)$ of 
$S(\hat{\fD}_\fm)$, the corresponding 
$\hat{f}_\fd$
is in $\kk_q[\widehat{P_{\overline{\varphi}}}]$.

By Proposition 
\ref{prop_scat}, the Hamiltonian $\hat{H}_\fd$ attached 
to an outgoing ray 
$\fd$ of 
$S(\hat{\fD}_\fm)-
\hat{\fD}_\fm$ is given by 
\[\hat{H}_{\fd}=
\left(
\frac{i}{\hbar}
\right)
\sum_{p \in P_{m_\fd}} 
\left(
\sum_{g \geqslant 0} N_{g,\beta_p}^{Y_\fm /
\partial Y_\fm}
\hbar^{2g} \right)
\hat{z}^{\beta_p-\bar{\varphi}(\ell_\beta m_{\fd})} \,. \]
According to
\cite[Theorem 33]{bousseau2018quantum_tropical},
for every 
$p \in P_{m_\fd}$, there exists 
\[ \Omega_p^{Y_\fm}(q^{\frac{1}{2}})
=\sum_{j\in \Z}
\Omega_{p,j}^{Y_\fm}
q^{\frac{j}{2}} \in \Z[q^{\pm \frac{1}{2}}]\,,\]
such that 
\[\left(\frac{i}{\hbar}
\right)
\left( \sum_{g \geqslant 0} N_{g,p}^{Y_\fm} \hbar^{2g-1} \right)
=-
(-1)^{\beta_p.\partial Y_\fm+1}
\sum_{p=\ell p'} \frac{1}{\ell}
\frac{1}{
q^{\frac{\ell}{2}}-q^{-\frac{\ell}{2}}} \Omega_{p'}^{Y_\fm}(q^{\frac{\ell}{2}})  
\,,\]
which can be rewritten  
\[\left(\frac{i}{\hbar}
\right)
\left( \sum_{g \geqslant 0} N_{g,p}^{Y_\fm} \hbar^{2g-1} \right)
=\sum_{j \in \Z}
\sum_{p=\ell p'} \frac{1}{\ell}
\frac{1}{
q^{\frac{\ell}{2}}-q^{-\frac{\ell}{2}}}
(-1)^{\ell \beta_{p'}.\partial Y_\fm} \Omega_{p',j}^{Y_\fm}q^{\frac{j \ell}{2}}  
\,,\]
Using Lemma
\ref{lem_ex_ch3}, we get that 
\[\hat{f}_\fd
=\prod_{p \in P_{m_\fd}}
\prod_{j \in \Z}
(1+q^{\frac{j-1}{2}} \hat{z}^{\beta_p
-\bar{\varphi}(\ell_\beta m_\fd)})^{\Omega_{p,j}^{Y_\fm}}\,,\]
which concludes the proof.
\end{proof}

As the initial rays of the scattering diagram $\hat{\fD}_\fm$ depend on $\hbar$ through rational functions of $q^{\frac{1}{2}}$, it follows directly from the Kontsevich-Soibelman algorithm producing 
$S(\hat{\fD}_\fm)$ that all the dependence on $\hbar$ in $S(\hat{\fD}_\fm)$ is through rational functions of $q^{\frac{1}{2}}$.
But we do not know of an elementary way to see directly that the functions $\hat{f}_\fd$ 
have coefficients which are Laurent polynomials in $q^{\frac{1}{2}}$ and not general rational functions in $q^{\frac{1}{2}}$. In the above proof of Proposition 
\ref{prop_q_integ}, we use
\cite[Theorem 33]{bousseau2018quantum_tropical}, which relies on some quite deep results of 
\cite{MR2851153}.

\section{Example: degree 5 del Pezzo surfaces}
\label{section_examples_ch3}

Let $Y$ be a del Pezzo surface of degree 5, that is, a blow-up of 
$\PP^2$ at four points in general 
position, and let $D$ be an anticanonical cycle of five $(-1)$-curves
on $Y$. Then 
$(Y,D)$ is a positive Looijenga pair.
The Looijenga pair 
$(Y,D)$ is studied in 
\cite[Examples 1.9, 3.7 and 6.12]{MR3415066}.
Remark that the interior $U=Y-D$ has topological Euler characteristic 
$e(U)=2$.

Let $j$ be an index 
modulo 5. We denote by $D_j$ the components of $D$ and $\rho_j$ the corresponding 
one-dimensional cones in the tropicalization
$(B,\Sigma)$ of $(Y,D)$.
Let
$v_j$ be the primitive generator of $\rho_j$ and $E_j$ be the unique $(-1)$-curve 
in $Y$ which is not contained in $D$ and meets $D_j$ transversally at one point. 

The only curve classes contributing to the canonical quantum scattering diagram
$\hat{\fD}^{\can}$ are multiples of some 
$[E_j]$, and so $\hat{\fD}^{\can}$
consists of five rays $(\rho_j, 
\hat{H}_{\rho_j})$.
By \cite[Lemma 23]{bousseau2018quantum_tropical}
we have 
\[ \hat{H}_{\rho_j}
=
i
\sum_{\ell \geqslant 1}\frac{1}{\ell} \frac{(-1)^{\ell-1}}{2 \sin \left( \frac{\ell \hbar}{2} \right)} \hat{z}^{\ell \eta([E_{j}])-\ell \varphi_{\rho_j}(v_j)}   \,.\]
and so, by Lemma
\ref{lem_ex_ch3}, the corresponding 
$\hat{f}_{\rho_j}$ are given by 
\[ \hat{f}_{\rho_j}
=1+q^{-\frac{1}{2}} 
\hat{z}^{E_j-\varphi_{\rho_j}(v_j)}\,.\]

\begin{prop}
The $\kk[NE(Y)]$-algebra 
defined by the product formula
of Theorem 
\ref{thm_product}
is generated by the quantum theta functions 
$\hat{\vartheta}_{v_j}$, satisfying 
the relations
\[ \hat{\vartheta}_{v_{j-1}}
\hat{\vartheta}_{v_{j+1}}
=\hat{z}^{[D_j]}
(\hat{z}^{[E_j]}
+q^{\frac{1}{2}} \hat{\vartheta}_{v_j}) \,,\]
\[ \hat{\vartheta}_{v_{j+1}}
\hat{\vartheta}_{v_{j-1}}
=\hat{z}^{[D_j]}
(\hat{z}^{[E_j]}
+q^{-\frac{1}{2}} \hat{\vartheta}_{v_j}) \,.\]
\end{prop}

\begin{proof}
The description of quantum broken lines
is identical to the description of classical 
broken lines given in 
\cite[Example 3.7]{MR3415066}.

The term $\hat{z}^{[D_j]}
\hat{z}^{[E_j]}$
is the coefficient of 
$\hat{\vartheta}_0=1$.
The final directions of the broken lines 
$\gamma_1$ and 
$\gamma_2$ satisfy 
$s(\gamma_1)+s(\gamma_2)=0$,
so
$\langle s(\gamma_1),
s(\gamma_2) \rangle =0$
and the quantum result is identical to the classical one.

The term $\hat{z}^{[D_j]}
\hat{\vartheta}_{v_j}$ corresponds to two straight
broken lines for 
$v_{j-1}$ and $v_{j+1}$, with endpoint the point 
$v_j$ of 
$\rho_j$. The
corresponding extra power of $q$ in Theorem
\ref{thm_product}
is $q^{\pm \frac{1}{2}
\langle v_{j-1}, v_{j+1} \rangle}
=q^{\pm \frac{1}{2}}$.
\end{proof}

Setting $[E_j]=[D_j]=0$, we recover some well-known description of the $A_2$ quantum $\cX$-cluster algebra: see formula (60) in \cite[\S 3.3]{MR2567745}.

\section{Higher-genus mirror symmetry and string theory}
\label{section_physics}

\subsection{From higher-genus to quantization via Chern--Simons theory}
In \cite[\S 9]{bousseau2018quantum_tropical},
we compared our enumerative
interpretation of the $q$-refined 2-dimensional Kontsevich-Soibelman scattering diagrams in terms of higher-genus log Gromov--Witten invariants of log Calabi-Yau surfaces with the physical derivation of the refined wall-crossing formula from topological string given by Cecotti-Vafa
\cite{cecotti2009bps}.

A parallel discussion shows that the 
main result of the present paper, the connection between higher-genus log Gromov--Witten invariants of log Calabi-Yau surfaces and quantization of the mirror geometry, also fits naturally into this story.

Let $(Y,D)$ be a Looijenga pair. The complement 
$U\coloneqq Y-D$
is a non-compact holomorphic symplectic surface admitting a Lagrangian torus fibration \cite{MR2024634}.
According to the SYZ picture of mirror symmetry, the mirror of $U$ should be obtained by taking the dual Lagrangian torus
fibration, corrected by counts of holomorphic discs in $U$ with boundary on the torus fibers.
In some cases, $U$ admits a hyperkähler metric,
such that the original complex structure of $U$
is the compatible complex structure $J$, and such that 
the SYZ fibration becomes $I$-holomorphic Lagrangian. 
Typical examples include two-dimensional Hitchin moduli spaces: see
\cite{boalch2012hyperkahler}
for a nice review.
From now on, we assume that we are in such case,
and so we should be able to consider the kind of twistorial construction considered by Cecotti and Vafa.

Let $(I,J,K)$ be a quaternionic triple of compatible complex structure, 
$(\omega_I, \omega_J, \omega_K)$ be the corresponding triple of real symplectic forms and 
$(\Omega_I, \Omega_J, \Omega_K)$ be the 
corresponding triple of holomorphic symplectic forms. Let $\Sigma \subset U$ be
a fiber of the original SYZ fibration. 
It is a $I$-holomorphic Lagrangian subvariety of $U$, that is, a submanifold such that $\Omega_I|_\Sigma =0$. It is an example of a $(B,A,A)$-brane in $U$
in the sense of \cite{MR2306566}, that is, a complex subvariety for the complex structure $I$ and a Lagrangian for any of the real symplectic forms
$(\cos \theta) \omega_J + (\sin \theta) \omega_K$, $\theta \in \R$.
There is in fact a twistor sphere 
$J_\zeta$, $\zeta \in \PP^1$, of compatible complex structures, such that $I=J_0$, 
$J=J_1$ and $K=J_i$.
Let $X$ be the non-compact Calabi-Yau 3-fold,
of underlying smooth manifold $U \times \C^{*}$ and 
equipped with a complex structure twisted in a twistorial way, that is, such that the fiber over 
$\zeta \in \C^{*}$ is the complex variety 
$(U,J_\zeta)$. Consider $S^1 \subset \C^{*}$ and 
$L \coloneqq \Sigma \times S^1 \subset X$.

We consider the open topological string $A$-model on $(X,L)$, that is, the count of holomorphic maps 
$(C, \partial C) \rightarrow (X,L)$ from an open Riemann surface $C$ to $X$ with boundary $\partial C$ mapping to $L$\footnote{Usually, A-branes, that is, boundary conditions for the A-model, have to be Lagrangian submanifolds. In fact, $L$ is not Lagrangian in $X$ but only totally real. Combined with specific aspects of the twistorial geometry, it is probably enough to have well-defined worldsheet instanton contributions. As suggested in \cite{cecotti2009bps}, it would be interesting to clarify this point. }. We restrict ourselves to open Riemann surfaces with only one boundary component. Given a class $\beta \in H_2(X,L)$, let $N_{g,\beta} \in \Q$ be the `count' of holomorphic maps $\varphi \colon (C,\partial C) \rightarrow (X,L)$ with $C$ a genus-$g$ Riemann surface with one boundary component and $[\varphi(C, \partial C)]=\beta$.
A 
holomorphic map $\varphi \colon (C,\partial C) \rightarrow (X,L)$ of class $\beta \in H_2(X,L)$ is a $J_{e^{i\theta}}$-holomorphic map to 
$U$, at a constant value $e^{i \theta} \in S^1$, where 
$\theta$ is the argument of $\int_\beta \Omega_I$.

The log Gromov--Witten invariants with insertion of a top lambda class 
$N_{g,\beta}$, introduced in
\cref{section_canonical}, should be viewed as a rigorous definition of the open Gromov--Witten invariants in the twistorial geometry $X$, with 
boundary on a torus fiber $\Sigma$ `near infinity'.
We refer to \cite[Lemma 7]{MR2746343}
for comparison, in the compact analogue given by K3 surfaces, between Gromov--Witten invariants of a holomorphic symplectic surface with insertion of a top lambda class and Gromov--Witten invariants of a corresponding three-dimensional
twistorial geometry. 
The key point is that the
lambda class comes from the
comparison of the deformation theories of stable maps mapping to the surface or to the 3-fold.

According to Witten \cite{MR1362846}, 
in the absence of non-constant worldsheet instantons, 
the effective spacetime theory of the A-model
on the A-brane $L$ is Chern--Simons theory of gauge group $U(1)$. 
The non-constant worldsheet instantons deform this result: see \cite[\S 4.4]{MR1362846}. 
The effective spacetime theory on the $A$-brane $L$ is still a $U(1)$ gauge theory but the Chern--Simons action is deformed by additional terms involving the worldsheet instantons. The genus-0 worldsheet instantons correct the classical action, whereas higher genus worldsheet instantons give higher quantum corrections.

We now arrive at the key point:
the relation between the SYZ mirror construction in terms of dual tori and 
the Chern--Simons story, whose quantization is supposed to be naturally related to 
higher-genus curves.
As $L=\Sigma \times S^1$, we can adopt a Hamiltonian description where $S^1$ plays the role of the time direction. The key point is that the classical phase space of $U(1)$ Chern-Simons theory
on $L=\Sigma \times S^1$ 
is the space 
of $U(1)$ flat connections on $\Sigma$,
that is, it is exactly the dual torus of $\Sigma$ used in the construction of the SYZ mirror. The genus-0 worldsheet instanton corrections to $U(1)$ Chern-Simons theory then translate into the genus-0 worldsheet instantons corrections in the SYZ construction of the mirror.

The Poisson structure on the mirror comes from the natural Poisson structure on the classical phase space of Chern--Simons theory. It is then natural to think that a quantization of the mirror should be obtained from quantization of Chern-Simons theory.
Quantization of the torus of flat connections gives a quantum torus and higher-genus worldsheet instantons corrections to quantum Chern-Simons theory 
imply that these quantum tori should be glued together in a non-trivial way.
We recover the main construction of the present paper: gluing quantum tori together using higher-genus curve counts in the gluing functions.
The fact that we have been able to give a rigorous version of this construction should be viewed as a highly non-trivial mathematical check of the above string-theoretic expectations.

\subsection{Quantization and higher-genus mirror symmetry}

In the previous section we explained how to understand the connection between higher-genus log Gromov--Witten invariants and deformation quantization using Chern-Simons theory as an intermediate step. In this explanation, a key role is played by the non-compact Calabi-Yau 3-fold $X$, a partial twistor family of $U$.

In the present section, we adopt a slightly different point of view, and we also consider a similar non-compact Calabi-Yau 3-fold on the mirror side: 
$Y=V \times \C^{*}$.
It is natural to expect that the mirror symmetry relation between $U$ and $V$ lifts to a mirror symmetry relation between the Calabi-Yau 3-folds $X$ and $Y$.

As explained in the previous section, the higher-genus 
log Gromov--Witten invariants considered in the present paper should be viewed 
as part of an algebraic version of the open higher-genus A-model on $X$.
Open higher-genus A-model should be mirror to open 
higher-genus B-model on $Y$.
We briefly explain below why the open higher-genus 
B-model on $Y=V \times \C^{*}$ has something to do with
quantization of the holomorphic symplectic variety 
$V$.

The string field theory of open higher-genus 
B-model for a single B-brane wrapping $Y$
is the holomorphic Chern-Simons theory, of field 
a $(0,1)$-connection $A$ and of action
\[ S(A)=\int_Y \Omega_Y \wedge A \wedge \bar{\partial} A \,,\]
where $\Omega_Y$ is the holomorphic volume form of $Y$. We will be rather interested 
in a single B-brane wrapping a curve 
$\C^{*}_v \coloneqq \{v\} \times \C^{*} \subset Y$,
where $v$ is a point in $V$. The
dimensional reduction of holomorphic Chern-Simons theory to describe a B-brane wrapping a curve was first studied by Aganagic and Vafa \cite[\S 4]{aganagic2000mirror}. 
Writing locally 
\[ \Omega_Y = dx \wedge dp \wedge \frac{dz}{z} \,,\]
where $(x,p)$ are local holomorphic Darboux
coordinates on $V$ near $v$ and $z$ a linear coordinate along $\C^{*}$, the fields of the reduced theory on 
$\C^{*}_v$ are functions $(x(z,\bar{z}),p(z,\bar{z}))$ and 
the action is 
\[S(x,p)= \int_{\C^{*}_v} \frac{dz}{z} \wedge p \wedge \bar{\partial} x \,.\]
A further dimensional reduction from the cylinder $\C^{*}_v$ to a real line $\R_t$ leads to a theory of a particle moving on $V$, of position $(x(t),z(t))$, of action
\[S(x,p)=\int_{\R_t} p(t) dx(t) \,.\]
In particular, $p(t)$ and $x(t)$ are canonically conjugate variables and in the corresponding quantum theory, 
obtained as dimensional reduction of the higher-genus 
B-model, they should become operators satisfying the
canonical commutation relations 
$[x,p]=\hbar$. We conclude that the higher-genus B-model
of the B-branes $\C^{*}_v$ should lead to a 
quantization of the holomorphic symplectic surface $V$.
The same relation between higher-genus B-model and quantization appears in \cite{MR2191887} and follow-ups.

We conclude that our main result,
Theorem \ref{main_thm_ch3}, should be viewed as an example of a higher genus mirror symmetry relation.

\newpage

\newcommand{\etalchar}[1]{$^{#1}$}

\vspace{+8 pt}
\noindent
Department of Mathematics \\
Imperial College London \\
pierrick.bousseau12@imperial.ac.uk

\end{document}